\documentclass[11pt]{amsart}

\usepackage{url}
\usepackage{hyperref}

\usepackage{amsmath}
\usepackage{amssymb}
\usepackage{amsthm}
\usepackage{amscd}

\usepackage[english]{babel}

\usepackage{graphicx}
\usepackage{tikz-cd}

\DeclareFontFamily{OT1}{rsfs}{}
\DeclareFontShape{OT1}{rsfs}{n}{it}{<-> rsfs10}{}
\DeclareMathAlphabet{\mathscr}{OT1}{rsfs}{n}{it}

\newtheorem{thm}{Theorem}[section]

\newtheorem{lemma}[thm]{Lemma}
\newtheorem{corollary}[thm]{Corollary}
\newtheorem{proposition}[thm]{Proposition}

\theoremstyle{definition}
\newtheorem{remark}[thm]{Remark}
\newtheorem{definition}[thm]{Definition}
\newtheorem{example}[thm]{Example}

\numberwithin{equation}{section}

\DeclareMathOperator{\Gal}{Gal}
\DeclareMathOperator{\Mat}{Mat}
\DeclareMathOperator{\GL}{GL}

\DeclareMathOperator{\Hom}{Hom}

\DeclareMathOperator{\Spec}{Spec}

\DeclareMathOperator{\Spf}{Spf}
\DeclareMathOperator{\Spa}{Spa}
\DeclareMathOperator{\val}{val}

\DeclareMathOperator{\Tr}{Tr}

\DeclareMathOperator{\id}{id}
\DeclareMathOperator{\im}{im}

\renewcommand{\H}{\ensuremath{{\rm{H}}}}
\newcommand*{\R}{\ensuremath{\mathbf{R}}}                        % reals
\newcommand*{\Z}{\ensuremath{\mathbf{Z}}}                        % integers
\newcommand*{\Q}{\ensuremath{\mathbf{Q}}}                        % rationals
\newcommand*{\F}{\ensuremath{\mathbf{F}}}                        % field
\newcommand*{\C}{\ensuremath{\mathbf{C}}}                        % complex's
\newcommand*{\A}{\ensuremath{\mathbf{A}}}                        % affine/adele
\newcommand*{\N}{\ensuremath{\mathbf{N}}}                        % naturals
\newcommand*{\B}{\ensuremath{\mathbf{B}}}
\newcommand*{\E}{\ensuremath{\mathbf{E}}}

\AtBeginDocument{\renewcommand*{\O}{\mathscr{O}}}                                  % 'sheaf' O

\newcommand*{\rig}{\mathrm{rig}}

\newcommand{\htimes}{\mathop{\widehat\otimes}}

\newcommand*{\cyc}{\rm{cyc}}

\begin{document}

\title{Galois representations over pseudorigid spaces}
\author[R. Bellovin]{Rebecca Bellovin}
\address{Rebecca Bellovin	\\
School of Mathematics and Statistics	\\
University of Glasgow	\\
University Place	\\
Glasgow G12 8QQ	\\
United Kingdom
}
\email{rebecca.bellovin@glasgow.ac.uk}
\urladdr{http://rmbellovin.github.io}

%\subjclass[2020]{11S25, 14G22}

%\keywords{$p$-adic Hodge theory, $(\varphi,\Gamma)$-modules, overconvergence, pseudorigid spaces}

\thanks{I am grateful to Kevin Buzzard, Toby Gee, Kiran Kedlaya, and James Newton for many helpful conversations during the course of this work.}

\begin{abstract}
	We study $p$-adic Hodge theory for families of Galois representations over pseudorigid spaces.  Such spaces are non-archimedean analytic spaces which may be of mixed characteristic, and which arise naturally in the study of eigenvarieties at the boundary of weight space.  We introduce perfect and imperfect overconvergent period rings, and we use the Tate--Sen method to construct overconvergent $(\varphi,\Gamma)$-modules for Galois representations over pseudorigid spaces.
\end{abstract}

\maketitle

%\begin{resume}
%	Nous \'etudions la th\'eorie de Hodge $p$-adique pour les familles de repr\'esentations galoisiennes sur les espaces pseudorigides.  Tels espaces sont espaces analytiques non-archim\'ediens qui peuvent \^etre de caract\'eristique mixte, et qui peuvent appara\^itre dans l'\'etude des vari\'et\'es de Hecke au bord de l'espace des poids.  Nous introduisons les anneaux des periodes parfaits et imparfaits, et nous utilisons le methode de Tate--Sen pour construire les $(\varphi,\Gamma)$-modules surconvergentes associ\'ees aux repr\'esentations galoisiennes sur les espaces pseudorigides.
%\end{resume}

\bigskip

\section{Introduction}

In this article, we study $p$-adic Hodge theory for families of representations of $\Gal_K$, where $K/\Q_p$ is a finite extension, and the families vary over certain analytic spaces, in the sense of Huber~\cite{huber1994}.  
Such families have been considered by a number of authors in the classical rigid analytic setting, where $p$ is invertible in $R$, but working in Huber's setting permits us to study Galois representations with coefficients which are characteristic $p$ or mixed characteristic.

Roughly speaking, $p$-adic Hodge theory is the study of representations of Galois groups, where both the Galois group and the coefficients are $p$-adic or characteristic $p$.  One of the most powerful tools for studying $p$-adic Galois representations is the theory of $(\varphi,\Gamma)$-modules, which provide an equivalence between Galois representations and a certain category of modules equipped with an operator $\varphi$ and the action of a $1$-dimensional $p$-adic Lie group, called \emph{\'etale $(\varphi,\Gamma)$-modules}.

The category of \'etale $(\varphi,\Gamma)$-modules is a full subcategory of the category of all $(\varphi,\Gamma)$-modules, and it often happens that the $(\varphi,\Gamma)$-module attached to an irreducible Galois representation becomes reducible in this larger category.  Moreover, this reducibility is closely related to subtle $p$-adic Hodge theoretic invariants of the representation.  If the $(\varphi,\Gamma)$-module attached to a Galois representations is the successive extension of rank-$1$ $(\varphi,\Gamma)$-modules, the representation is said to be \emph{trianguline}.

One key feature of $(\varphi,\Gamma)$-modules is that they behave well in rigid analytic families.  Given a Galois representation with coefficients in a $\Q_p$-affinoid algebra, the work of Berger and Colmez~\cite{berger-colmez} constructs a family of $(\varphi,\Gamma)$-modules.  Their construction is functorial, and so globalizes to sheaves of Galois representations over general rigid analytic spaces.

However, in recent years, interest has developed in families of Galois representations parametrized by analytic spaces which are not defined over a field.  For example, Andreatta--Iovita--Pilloni constructed the eigencurve in mixed characteristic~\cite{aip2018}, and their construction was extended to more general eigenvarieties by Johansson--Newton~\cite{johansson-newton}.

In this note, we study Galois representations with coefficients in similar rings.  More precisely, we consider projective modules $M$ over pseudoaffinoid algebras $R$ equipped with a continuous $R$-linear action of $\Gal_K$ (pseudoaffinoid algebras, and their associated pseudorigid adic spaces, are a class of analytic adic spaces studied in \cite{johansson-newton} and \cite{lourenco}).

Before we can construct and study families of $(\varphi,\Gamma)$-modules, we need to define the appropriate overconvergent period rings.  In the rigid analytic setting, it was enough to take completed tensor products of $\Q_p$-Banach algebras and the overconvergent period rings $\B_{\rig,K}^{\dagger,s}$ defined in e.g.~\cite{berger}.  This is not possible in our setting because neither $R$ nor $\B_{\rig,K}^{\dagger,s}$ has the $p$-adic topology.  However, both are affinoid Tate rings, and in this setting it is possible to define fiber products of the associated adic spaces.  We provide a construction in Appendix~\ref{appendix} for the convenience of the reader.

We then define both perfect and imperfect overconvergent rings, which we denote by $\widetilde\Lambda_{R,[a,b]}$ and $\Lambda_{R,[a,b],K}$.  Let $\chi:\Gal_K\rightarrow \Z_p^\times$ denote the cyclotomic character and let $H_K:=\ker\chi$; there is an action of $H_K$ on $\widetilde\Lambda_{R,[a,b]}$.  We study this Galois action on the perfect rings, and show that we have appropriate normalized trace maps $R_{K,n}:\widetilde\Lambda_{R,[0,b]}^{H_K}\rightarrow \varphi^{-n}\Lambda_{R,[0,p^{-n}b],K}$.  

This lets us prove our main theorem, on the construction of $(\varphi,\Gamma)$-modules attached to Galois representations:
\begin{thm}
Let $R$ be a pseudoaffinoid  algebra, and let $M$ be a finite projective $R$-module of rank $d$ equipped with a continuous action of $\Gal_K$, for a finite extension $K/\Q_p$.  Then there is a functorially associated projective $(\varphi,\Gamma_K)$-module $D_{(0,b],K}(M)$. This $(\varphi,\Gamma)$-module is equipped with a $\Gal_K$- and $\varphi$-equivariant isomorphism 
\[	\widetilde\Lambda_{R,(0,b]}\otimes_{\Lambda_{R,(0,b],K}}D_{(0,b],K}(M)\xrightarrow\sim \widetilde\Lambda_{R,(0,b]}\otimes_RM	\]
\end{thm}

We conclude by showing how to compute the Galois cohomology of $M$ in terms of $D_{(0,b],K}(M)$, using the Fontaine--Herr--Liu complex.

In subsequent work ~\cite{bellovin21}, we study the cohomology of $(\varphi,\Gamma)$-modules over pseudorigid spaces, and give applications to eigenvarieties at the boundary of weight space.

\section{Classical rings of $p$-adic Hodge theory}

Let $\C_p^\flat:=\varprojlim_{x\rightarrow x^p}\C_p$, and let $\mathscr{O}_{\C_p}^\flat$ be the subset of $x\in\C_p^\flat$ such that $x^{(0)}\in\mathscr{O}_{\C_p}$.  Then $\C_p^\flat$ is an algebraically closed field of characteristic $p$ with ring of integers $\mathscr{O}_{\C_p}^\flat$; Colmez calls these rings $\widetilde\E$ and $\widetilde\E^+$, respectively.  There is a valuation $v$ defined by $v((x^{(i)}))=v_p(x^{(0)})$, and $\C_p^\flat$ is complete with respect to this valuation.  There is also a Frobenius (given by raising to the $p$th power).  

Let $F$ be a finite unramified extension of $\Q_p$ with ring of integers $\O_F$ and residue field $k_F$ (so that $\O_F=W(k_F)$).  Let $\varepsilon:=(\varepsilon^{(0)},\varepsilon^{(1)},\varepsilon^{(2)}\ldots)\in \mathscr{O}_{\C_p}^\flat$ be a choice of compatible $p$th power roots of unity with $\varepsilon^{(0)}=1$ and $\varepsilon^{(1)}\neq 1$.  There is a natural map $k_F(\!(\overline\pi)\!)\rightarrow \C_p^\flat$ given by sending $\overline\pi$ to $\varepsilon-1$; we denote its image by $\E_F$, we denote by $\E$ the separable closure of $\E_F$ inside $\C_p^\flat$, and we denote by $\E^+$ the valuation ring of $\E$.

Let $\A_{\mathrm{inf}}:=W(\mathscr{O}_{\C_p}^\flat)$ and $\widetilde\A:=W(\C_p^\flat)$. There are two possible topologies on $\A_{\mathrm{inf}}$ and $\widetilde\A$, the $p$-adic topology or the weak topology; they are complete for both.  

The \emph{$p$-adic topology} is defined by putting the discrete topology on each quotient $W(\C_p^\flat)/p^nW(\C_p^\flat)$, and taking the projective limit topology on $\widetilde\A$; $\A_{\mathrm{inf}}$ is given the subspace topology.  The \emph{weak topology} is defined by putting the valuation topology on $\C_p^\flat$ and giving $\widetilde\A$ the product topology; $\A_{\mathrm{inf}}$ is again given the subspace topology.

Alternatively, the weak topology on $\widetilde\A$ is given by taking the sets 
\[	U_{k,n}:=p^k\widetilde\A+[\widetilde p]^n\A_{\mathrm{inf}}\text{ for }k,n\geq 0	\]
to be a basis of neighborhoods around $0$, where $\widetilde p\in\mathscr{O}_{\C_p}^\flat$ is any fixed element with $\widetilde p^{(0)}=p$ (i.e., $\widetilde p$ is a system of compatible $p$-power roots of $p$).  The weak topology on $\A_{\mathrm{inf}}$ is similarly generated by the sets 
\[	U_{k,n}\cap\A_{\mathrm{inf}} = p^k\A_{\mathrm{inf}}+[\widetilde p]^n\A_{\mathrm{inf}}	\]
The weak topology on $\A_{\mathrm{inf}}$ is equivalent to the $(p,[\varpi])$-adic topology, for any pseudo-uniformizer $\varpi\in\C_p^\flat$.

Both rings carry continuous bijective actions of Frobenius (for either topology).  However, the Galois action is continuous for the weak topology, but it is not continuous for the $p$-adic topology because the Galois actions on $\mathscr{O}_{\C_p}^\flat$ and $\C_p^\flat$ are not discrete.

Explicitly, Frobenius acts by
\[	\varphi\left(\sum_{k=0}^\infty p^k[x_k]\right)=\sum_{k=0}^\infty p^k[x_k^p]	\]
and the Galois group $\Gal_{K}$ acts by
\[	\sigma\left(\sum_{k=0}^\infty p^k[x_k]\right)=\sum_{k=0}^\infty p^k[\sigma(x_k)]	\]

Now consider the pre-adic space $\Spa \A_{\mathrm{inf}}$ and its analytic adic subspace $\mathcal{Y}$ (i.e., $\mathcal{Y}$ is $\Spa\A_{\mathrm{inf}}$ minus the point corresponding to the maximal ideal).  If $\varpi$ is a pseudo-uniformizer of $\C_p^\flat$, there is a surjective continuous map $\kappa:\mathcal{Y}\rightarrow [0,\infty]$ given by
\[	\kappa(x):=\frac{\log\lvert[\varpi](\widetilde x)\rvert}{\log\lvert p(\widetilde x)\rvert}	\]
where $\widetilde x$ is the rank-$1$ generization of $x$.  If $I\subset [0,\infty]$ is an interval, we let $\mathcal{Y}_I:=\kappa^{-1}(I)$.  The Frobenius on $\A_{\mathrm{inf}}$ induces isomorphisms $\mathcal{Y}_I\rightarrow \mathcal{Y}_{pI}$ (since $\kappa\circ\varphi=p\kappa$).  Note that $\log\lvert[\varpi](\widetilde x)\rvert, \log\lvert p(\widetilde x)\rvert\in [-\infty,0)$, since $p$ and $[\varpi]$ are both topologically nilpotent, and therefore $\lvert[\varpi](\widetilde x)\rvert, \lvert p(\widetilde x)\rvert <1$.

Following Scholze, we choose $\varpi=\widetilde p$, that is, a compatible sequence of $p$-power roots of $p$.
Suppose $a,b\in[0,\infty]$ are rational numbers and $a\leq b$.  Then $\mathcal{Y}_{[a,b]}$ is an affinoid subspace of $\mathcal{Y}$, and we write 
\[	\mathcal{Y}_{[a,b]}=\Spa(\widetilde\Lambda_{[a,b]}, \widetilde\Lambda_{[a,b]}^\circ)	\]
The inequalities
\[	a\leq \frac{\log\lvert[\varpi](\widetilde x)\rvert}{\log\lvert p(\widetilde x)\rvert}\leq b	\]
translate to the conditions
\[	a\log\lvert p(\widetilde x)\rvert\geq \log\lvert[\varpi](\widetilde x)\rvert\geq b\log\lvert p(\widetilde x)\rvert	\]
or equivalently, $\lvert p(\widetilde x)\rvert^a\geq \lvert[\varpi](\widetilde x)\rvert\geq \lvert p(\widetilde x)\rvert^b$.  Thus,
\begin{equation*}
\resizebox{\displaywidth}{!}{$
\begin{split}
(\widetilde\Lambda_{[a,b]}, \widetilde\Lambda_{[a,b]}^\circ)&=\left(\A_{\mathrm{inf}}\left\langle \frac{p}{[\varpi]^{1/b}}, \frac{[\varpi]^{1/a}}{p}\right\rangle\left[\frac{1}{[\varpi]}\right],\A_{\mathrm{inf}}\left\langle \frac{p}{[\varpi]^{1/b}}, \frac{[\varpi]^{1/a}}{p}\right\rangle\left[\frac{1}{[\varpi]}\right]^\circ\right)	\\
&= \left(\A_{\mathrm{inf}}\left\langle \frac{p}{[\varpi]^{1/b}}, \frac{[\varpi]^{1/a}}{p}\right\rangle\left[\frac{1}{[\varpi]}\right],\A_{\mathrm{inf}}\left\langle \frac{p}{[\varpi]^{1/b}}, \frac{[\varpi]^{1/a}}{p}\right\rangle\right)
\end{split}
$}
\end{equation*}
Here we take $p/[\varpi]^{\infty}:=1/[\varpi]$ and we take $[\varpi]^\infty/p:=0$.  
If $\frac{p-1}{pa}, \frac{p-1}{pb}\in\Z_{\geq 0}[1/p]\cup\{\infty\}$, this is the pair of rings denoted $(\widetilde\A_{[s(b),s(a)]}[\frac{1}{[\varpi]}], \widetilde\A_{[s(b),s(a)]})$ in~\cite{berger} (since $v(\widetilde p)=1$ and $v(\pi)=p/p-1$).

In the special case $I=[a,b]=[0,b]$, $(\widetilde\Lambda_{[0,b]}, \widetilde\Lambda_{[0,b]}^\circ)$ is the pair of rings denoted $(\widetilde\A^{(0,b]},\widetilde\A^{\dagger,s(b)})$, where $s(b)=(p-1)/pb$.  We can write these rings more explicitly as subrings of $\widetilde\A$:
\[	\widetilde\Lambda_{[0,b]}^\circ=\widetilde\A^{\dagger,s(b)}=\{x=\sum_{k=0}^\infty p^k[x_k]\in\widetilde\A\mid x_k\in\mathscr{O}_{\C_p}^\flat, x_k\varpi^{k/b}\rightarrow 0\}	\]
and
\[	\widetilde\Lambda_{[0,b]}=\widetilde\A^{(0,b]}=\{x=\sum_{k=0}^\infty p^k[x_k]\in\widetilde\A\mid x_k\varpi^{k/b}\rightarrow 0\}	\]

If $[a,b]\neq[0,\infty]$, then the pair $(\widetilde\Lambda_{[a,b]},\widetilde\Lambda_{[a,b]}^\circ)$ is a Tate algebra; if $a\neq 0$, then $p$ is a pseudo-uniformizer, and if $b\neq\infty$, then $[\varpi]$ (and $[\overline\pi]$) is a pseudo-uniformizer.

We can equip $\widetilde\Lambda_{[0,b]}$ with a valuation 
\[	\val^{[0,b]}(x):=\inf_{k\geq 0}(v_{\C_p^\flat}(x_k)+k/b)	\]
It is separated and complete with respect to this valuation, and $\widetilde\Lambda_{[0,b]}^\circ$ is the ring of integers.

Since $\C_p^\flat$ is algebraically closed, we can extract arbitrary roots of $\varpi$; we may therefore define another valuation $v_b$ on $\widetilde\Lambda_{[0,b]}$ by setting
\[	v_b(x):=\sup_{r\in\Q:[\varpi]^rx\in \widetilde\Lambda_{[0,b]}^\circ}-r	\]
We observe that $v_b([\varpi]x)=1+v_b(x)$ for any $x\in\widetilde\Lambda_{[0,b]}$.

\begin{lemma}
For $x\in \widetilde\Lambda_{[0,b]}$, $\val^{(0,b]}(x)= v_b(x)$.
\end{lemma}
\begin{proof}
We observe that $[\varpi]^rx\in\widetilde\Lambda_{[0,b]}^\circ$ if and only if $\val^{[0,b]}([\varpi]^rx)=r+\val^{[0,b]}(x)\geq 0$, which holds if and only if $\val^{[0,b]}(x)\geq -r$.  Since we may approximate $\val^{[0,b)}(x)$ from below by rational numbers, it follows that $\val^{(0,b]}(x)= v_b(x)$.
\end{proof}

There are versions of all of these rings with no tilde; they are imperfect versions of the rings with tildes.

Let $\pi\in\widetilde\A_{\mathrm{inf}}$ denote $[\varepsilon]-1$, where $[\varepsilon]$ denotes the Teichm\"uller lift of $\varepsilon$.  Then there is a well-defined injective map $\mathscr{O}_F[\![X]\!][X^{-1}]\rightarrow \widetilde\A$ given by sending $X$ to $\pi$; we let $\A_F$ denote the $p$-adic completion of the image.  This is a Cohen ring for $\E_F$.  We define $\A$ to be the completion of the integral closure of the image of $\Z_p[\![X]\!][X^{-1}]$ in $\widetilde\Lambda_{[0,0]}$, and we let $\A^+:=\A\cap\A_{\mathrm{inf}}$.

Because extensions of $\E_F$ correspond to unramified extensions of $\A_F[1/p]$, we get a natural Galois action on $\A$ and $\A^+$.  If $K/\Q_p$ is an arbitrary finite extension, we may therefore define $\A_K:=\A^{H_K}$ and $\A_K^+=(\A^+)^{H_K}$.  When $K$ is unramified over $\Q_p$, this agrees with our original definition of these rings.

We define the \emph{overconvergent} subrings of $\A_K$: For $b\in[0,\infty]$, let $\Lambda_{[0,b],K}:=\A_K\cap\widetilde\Lambda_{[0,b]}$ and let $\Lambda_{[0,b],K}^\circ:=\A_K\cap\widetilde{\Lambda}_{[0,b]}^\circ$.  These rings are given the topology induced as closed subspaces of $\widetilde\Lambda_{[0,b]}$. Thus, $\A_K^+=\Lambda_{[0,\infty],K}$ and $\A_K=\Lambda_{[0,0],K}$.

Since we have isomorphisms $\varphi:\widetilde\Lambda_{[a,b]}^{H_K}\xrightarrow{\sim}\widetilde\Lambda_{[a/p,b/p]}^{H_K}$, we have induced Frobenius maps $\varphi:\Lambda_{[0,b],K}\rightarrow\Lambda_{[0,b/p],K}$.  However, as $\E_K$ is imperfect, these maps are no longer isomorphisms.  Indeed, $\Lambda_{[0,0],K}$ is free over $\varphi(\Lambda_{[0,0],K})$ of rank $p$, with a basis given by $\{1, [\varepsilon],\ldots,[\varepsilon^{p-1}\}$.  We may therefore define a left inverse $\psi:\Lambda_{[0,0],K}\rightarrow \Lambda_{[0,0],K}$ to $\varphi$ via 
	\[	\psi:=\frac{1}{p}\varphi^{-1}\circ \Tr_{\Lambda_{[0,0],K}/\varphi(\Lambda_{[0,0],K})}	\]

\begin{proposition}
If $b\in (0,\infty)$, then $\Lambda_{[0,b],K}$ is a Tate ring with ring of definition $\Lambda_{[0,b],K}^\circ$ and pseudo-uniformizer $\pi$.  If $b=\infty$, then $\Lambda_{[0,\infty],K}=\A_K^+$ is an adic ring topologized by the ideal $(p,\pi)$.
\end{proposition}
\begin{proof}
We first observe that the cokernel of the inclusion $\A_K\hookrightarrow\widetilde\A^{H_K}$ has no $p$- or $\pi$-torsion, so the same holds for the cokernel of the inclusions $\Lambda_{[0,b],K}^\circ\hookrightarrow(\widetilde\Lambda_{[0,b]}^\circ)^{H_K}$ for $b>0$.  Thus, the natural map $\Lambda_{[0,\infty],K}/p\rightarrow \A_{\mathrm{inf}}^{H_K}/p=\widehat{\mathscr{O}}_{K_\infty}^\flat$ remains injective; since $\Lambda_{[0,\infty],K}$ has the closed subspace topology from $\A_{\mathrm{inf}}^{H_K}$ and the topology on $\widehat{\mathscr{O}}_{K_\infty}^\flat$ is $\overline\pi$-adic, the topology on $\Lambda_{[0,\infty],K}$ is $(p,\pi)$-adic.  Similarly, for $b\in(0,\infty)$, the natural map $\Lambda_{[0,b],K}^\circ/\pi\rightarrow (\widetilde\Lambda_{[0,b]}^\circ)^{H_K}/\pi$ remains injective.  Since $\pi\in(\widetilde\Lambda_{[0,b]}^\circ)^{H_K}$ and $\pi$ is a topologically nilpotent unit of $\widetilde\Lambda_{[0,b]}^{H_K}$, the ideal $(\pi)\subset (\widetilde\Lambda_{[0,b]}^\circ)^{H_K}$ is an ideal of definition and $(\widetilde\Lambda_{[0,b]}^\circ)^{H_K}/\pi$ is discrete.  It follows that $(\pi)\subset \Lambda_{[0,b],K}^\circ$ is also an ideal of definition.
\end{proof}

We can be more explicit about the structure of $\Lambda_{[0,b],K}$ when $b$ is small.  Recall that $\overline\pi:=\varepsilon-1\in\mathscr{O}_{\C_p}^\flat$, so that it is a uniformizer of $\E_F$.  Then for any ramified extension $K/F$, we may choose a uniformizer $\overline\pi_K$ of $\E_K$, and we may lift $\overline{\pi}_K$ to $\pi_K\in\A$.  We fix a choice of $\pi_K$ for every $K$ and work with it throughout (when $F/\Q_p$ is unramified, we set $\pi_F=\pi$).  Let $F'$ be the maximal unramified extension of $F$ in $K_\infty$ and let $\mathscr{A}_{F'}^{(0,b]}:=\{\sum_{m\in\Z}a_mX^m:a_m\in\mathscr{O}_{F'}, v_p(a_m)+mb\rightarrow\infty\}$ be the ring of integers of the ring of bounded analytic functions on the half-open annulus $0<v_p(X)\leq b$ over $F'$.  Then \cite[Proposition 7.5]{colmez} states that for $b<r_K$ (where $r_K$ is a constant depending on the ramification of $\E_K/\E_F$), the assignment $f\mapsto f(\pi_K)$ is an isomorphism of topological rings from $\mathscr{A}_{F'}^{(0,bv_{\C_p^\flat}(\pi_K)]}$ to $\Lambda_{[0,b],K}$.  Furthermore, if we define a valuation $v^{(0,b]}$ on $\mathscr{A}_{F'}^{(0,b]}$ by $v^{(0,b]}(\sum_{m\in \Z}a_mX^m):=\inf_{m\in\Z}(v_p(a_m)+mb)$, then 
\[	\val^{(0,b]}(f(\pi_K))=\frac{1}{b}v^{(0,b]}(f)	\]
It follows that after inverting $p$, we have an isomorphism from the ring of bounded analytic functions on the half-open annulus to $\Lambda_{[0,b],K}[1/p]$, equipped with the valuation $\val^{(0,b]}$.  Note that when $K/\Q_p$ is ramified, this isomorphism depends on a choice of uniformizer of $\E_K$.

\section{Rings with coefficients}\label{section: rings with coefficients}

Now we wish to introduce coefficients.  We wish to consider Galois representations with coefficients in \emph{pseudoaffinoid algebras}, in the sense of~\cite{lourenco} and~\cite{johansson-newton} (or more generally, Galois representations on vector bundles over pseudorigid spaces).

\begin{definition}
Let $E$ be a discretely-valued non-archimedean field and let $\mathscr{O}_E$ be its ring of integers.  A \emph{pseudoaffinoid $\mathscr{O}_E$-algebra} is a complete Tate $\mathscr{O}_E$-algebra $R$ which has a ring of definition $R_0$ that is formally of finite type over $\mathscr{O}_E$.  A \emph{pseudorigid space} over $\mathscr{O}_E$ is an adic space $X$ over $\Spa(\mathscr{O}_E)$ which is locally of the form $\Spa(R,R^\circ)$ for a pseudoaffinoid $\mathscr{O}_E$-algebra $R$.
\end{definition}
When $R$ is a pseudoaffinoid algebra, we will write $\Spa R$ for $\Spa(R,R^\circ)$.

Let $R$ be a  pseudoaffinoid algebra over $\Z_p$, and let $u\in R$ be a pseudo-uniformizer.  Throughout this section, we assume that $p\notin R^\times$, since if $R$ has the $p$-adic topology, we are in the classical setting treated in~\cite{berger-colmez}.  Borrowing the notation of ~\cite{lourenco}, for any positive rational number $\lambda=\frac{m'}{m}$ with $(m,m')=1$ and $m,m'>0$, we let $(D_\lambda, D_\lambda^\circ)$ denote the pair 
\[	\left(\Z_p[\![u]\!]\left\langle\frac{p^m}{u^{m'}}\right\rangle\left[\frac 1 u\right], \Z_p[\![u]\!]\left\langle\frac{p^m}{u^{m'}}\right\rangle\right)	\]
Then by~\cite[Lemma 4.8]{lourenco}, there is some sufficiently small $\lambda$ such that $R$ is topologically of finite type over $D_\lambda$.  In particular, there is a ring of definition $R_0\subset R$ which is itself strictly topologically of finite type over $D_\lambda^\circ$.

Every element of $D_\lambda$ can be written uniquely as a power series $\sum_{i\in\Z}a_iu^i$ with $a_i\in\Z_p$, and we define a valuation $v_{D_\lambda}$ on $D_\lambda$ via 
\[	v_{D_\lambda}\left(\sum_{i\in\Z}a_iu^i\right):= \inf_i \left\{v_p(a_i)+\frac{i}{\lambda}\right\}	\]
If $R$ is a topologically of finite type over $D_\lambda$, each presentation of $R$ over $D_\lambda$ will induce a valuation on $R$.  Different presentations will induce different (but equivalent) valuations, however, just as in the classical setting.

If $\lambda<\lambda'$ with $\lambda=\frac{m'}{m}$ and $\lambda'=\frac{n'}{n}$, there is a natural map $D_\lambda\rightarrow D_{\lambda'}$.  Indeed, $\left(\frac{p^m}{u^{m'}}\right)^{n'} =  p^{mn'-m'n}\left(\frac{p^n}{u^{n'}}\right)^{m'}$; since $mn'-m'n>0$, we see that $\frac{p^m}{u^{m'}}$ is power-bounded in $D_{\lambda'}$. Thus, we have a totally ordered inverse system of adic space $\{D_\lambda\}$; since they are uniform, the inverse limit exists, and it is straightforward to check that it is equal to $\Spa(\F_p(\!(u)\!))$.
 
We define a descending sequence of ideals $I_j\subset R_0$ via $I_j:=p^jR\cap R_0$.  Then for all $j\geq 1$, $R_0/I_j$ is a $u$-torsion-free $(\Z/p^j)[\![u]\!]$-algebra.  Since $R_0$ is noetherian, the $I_j$ are finitely generated, and $I_j/I_{j+1}$ is a finite $u$-torsion-free $R_0/I_1$-module.

In this section, we will construct and study perfect and imperfect overconvergent period rings ``with coefficients in $R$''.  The rings we construct will depend on our choices of both $R_0$ and $u$, but for compactness of notation, we suppress $u$ from the notation.

\subsection{Perfect overconvergent rings}

We fix a pseudoaffinoid $\mathcal{O}_E$-algebra $R$, for some finite extension $E/\Q_p$, and we fix a ring of definition $R_0\subset R$ and pseudouniformizer $u\in R_0$ such that $R_0$ is strictly topologically of finite type over $D_\lambda^\circ$ for some sufficiently small $\lambda\in \Q_{>0}$.

The adic space $\mathcal{Y}$ is covered by the two open subspaces $\mathcal{Y}_{(0,\infty]}$ and $\mathcal{Y}_{[0,\infty)}$, which are the subspaces where $p\neq 0$ and $[\varpi]\neq 0$, respectively.  Thus, to study the fiber products $\mathcal{Y}_{R_0}:=\Spa R_0\times \mathcal{Y}$ and $\mathcal{Y}_R:=\Spa R\times \mathcal{Y}$, it suffices to study the fiber products of $\Spa R_0$ or $\Spa R$ with each of these subspaces.  We are primarily interested in $\Spa R_0 \times_{\Z_p}\mathcal{Y}_{[0,\infty)}$ and $\Spa R \times_{\Z_p}\mathcal{Y}_{[0,\infty)}$ (since if $p$ is invertible, the theory reduces to the case of classical rigid analytic spaces).

We define perfect overconvergent rings with coefficients in $R$:
\begin{definition}
	Fix $a\in\Q_{\geq 0}$ and $b\in\Q_{>0}\cup\{\infty\}$ such that $a\leq b$.  Then we define $\widetilde\Lambda_{R_0,[a,b]}$ to be the evaluation of $\mathscr{O}_{\mathcal{Y}_{R_0}}$ on the affinoid subspace defined by the conditions $u\leq [\varpi]^{1/b}\neq 0$ and $[\varpi]^{1/a}\leq u\neq 0$.  We let $\widetilde\Lambda_{R_0,[a,b],0}\subset \widetilde\Lambda_{R_0,[a,b]}$ be the ring of definition $(R_0\htimes\A_{\mathrm{inf}})\left\langle\frac{u}{[\varpi]^{1/b}},\frac{[\varpi]^{1/a}}{u}\right\rangle$, and we let $\widetilde\Lambda_{R,[a,b]}:=\widetilde\Lambda_{R_0,[a,b]}\left[\frac 1 u\right]$ and $\widetilde\Lambda_{R,[a,b],0}:=\widetilde\Lambda_{R_0,[a,b],0}\left[\frac 1 u\right]$.  Here we fix the conventions $\frac{u}{[\varpi]^{1/\infty}}:=u$ and $\frac{[\varpi]^{1/0}}{u}:=0$, so that $\Lambda_{R_0,[0,\infty]}=R_0\htimes\A_{\mathrm{inf}}$.

We also make auxiliary definitions 
\[	\widetilde\Lambda_{R_0,[0,0]}:=(R_0\htimes\A_{\mathrm{inf}})\left[\frac{1}{[\varpi]}\right]^\wedge \qquad\text{ and }\qquad	\widetilde\Lambda_{R,[0,0]}:=\widetilde\Lambda_{R_0,[0,0]}\left[\frac 1 u\right]	\]
where the completion is $u$-adic. We give $\widetilde\Lambda_{R,[0,0]}$ the \emph{weak topology} generated by the basis 
\[	\left\{u^k(R_0\htimes\A_{\mathrm{inf}})\left\langle\frac{1}{[\varpi]}\right\rangle+[\varpi]^n(R_0\htimes\A_{\mathrm{inf}})\right\}_{k,n}	\]
\end{definition}
Note that $\widetilde\Lambda_{R_0,[a,b]}$ and $\widetilde\Lambda_{R,[a,b]}$ only differ when $a=0$ (since otherwise both rings contain $\frac 1 u$).

The motivation for this definition is to describe natural affinoid subspaces of 
\[	\mathcal{Y}_{R_0}=\cup_{r>0}\Spa R_0 \times_{\Z_p}\Spa(\widetilde\Lambda_{[0,r]},\widetilde\Lambda_{[0,r]}^\circ)	\]
and 
\[	\mathcal{Y}_R=\cup_{r>0}\Spa R \times_{\Z_p}\Spa(\widetilde\Lambda_{[0,r]},\widetilde\Lambda_{[0,r]}^\circ)	\]
These will be pre-adic spaces (and in general not affinoid or quasi-compact), and they will be exhausted by pre-adic spaces of the form
\[	\Spa^{\mathrm{ind}}\left((R_0\htimes_{\Z_p}\widetilde\Lambda_{[0,r]}^{\circ})\left\langle\frac{u}{[\varpi]^{1/b}}\right\rangle\left[\frac{1}{[\varpi]}\right],\left((R_0\otimes_{\Z_p}\widetilde\Lambda_{[0,r]}^{\circ})\left[ \frac{u}{[\varpi]^{1/b}}\right]^{\mathrm{int}}\right)^\wedge\right)	\]
or
\begin{equation*}
\resizebox{\displaywidth}{!}{$
\Spa^{\mathrm{ind}}\left((R_0\htimes_{\Z_p}\widetilde\Lambda_{[0,r]}^\circ)\left\langle \frac{u}{[\varpi]^{1/b}},\frac{[\varpi]^b}{u}\right\rangle\left[\frac{1}{[\varpi]}, \frac 1 u\right], \left((R^+\otimes_{\Z_p}\widetilde\Lambda_{[0,r]}^\circ)\left[ \frac{u}{[\varpi]^{1/b}},\frac{[\varpi]^b}{u}\right]^{\mathrm{int}}\right)^\wedge\right)
$}
\end{equation*}
respectively, where $\Spa^{\mathrm{ind}}$ is as in Definition~\ref{def: pre-adic} (which is taken from ~\cite{scholze-weinstein-berkeley}). Here $b$ is a positive rational number, $(R_0\otimes_{\Z_p}\widetilde\Lambda_{[0,r]}^{\circ})\left[ \frac{u}{[\varpi]^{1/b}}\right]^{\mathrm{int}}$ denotes the integral closure of $(R_0\otimes_{\Z_p}\widetilde\Lambda_{[0,r]})\left[\frac{u}{[\varpi]^{1/b}}\right]$ in $(R_0\htimes_{\Z_p}\widetilde\Lambda_{[0,r]}^\circ)\left\langle \frac{u}{[\varpi]^{1/b}},\frac{[\varpi]^b}{u}\right\rangle\left[\frac{1}{[\varpi]}\right]$, and $(R^\circ\otimes_{\Z_p}\widetilde\Lambda_{[0,r]}^\circ)\left[ \frac{u}{[\varpi]^{1/b}},\frac{[\varpi]^b}{u}\right]^{\mathrm{int}}$ denotes the integral closure of $(R^\circ\otimes_{\Z_p}\widetilde\Lambda_{[0,r]}^\circ)\left[ \frac{u}{[\varpi]^{1/b}},\frac{[\varpi]^b}{u}\right]$ in $(R_0\otimes_{\Z_p}\widetilde\Lambda_{[0,r]}^\circ)\left[ \frac{u}{[\varpi]^{1/b}},\frac{[\varpi]^b}{u}\right]\left[\frac{1}{[\varpi]}, \frac 1 u\right]$.

\begin{proposition}\label{prop: lambda-r-r-o-d}
	Suppose that $R$ is a $D_\lambda$-algebra and $r\in \Q_{>0}\cup\{\infty\}$ satisfies $r\geq \frac{b}{\lambda}$. Then there is a continuous homomorphism 
\[	(R_0\htimes_{\Z_p}\widetilde\Lambda_{[0,\frac{b}{\lambda}]}^\circ)\left\langle \frac{u}{[\varpi]^{1/b}}\right\rangle\left[\frac{1}{[\varpi]}\right]\rightarrow(R_0\htimes_{\Z_p}\widetilde\Lambda_{[0,r]}^\circ)\left\langle\frac{u}{[\varpi]^{1/b}}\right\rangle\left[\frac{1}{[\varpi]}\right]        \]
and the image of $(R_0\htimes\widetilde\Lambda_{[0,\frac{b}{\lambda}]}^\circ)\left\langle \frac{u}{[\varpi]^{1/b}}\right\rangle$ in $(R_0\htimes_{\Z_p}\widetilde\Lambda_{[0,r]}^\circ)\left\langle\frac{u}{[\varpi]^{1/b}}\right\rangle\left[\frac{1}{[\varpi]}\right]$ is a ring of definition.
In particular, the image of $(R_0\htimes\widetilde\Lambda_{[0,\frac{b}{\lambda}]}^\circ)\left\langle \frac{u}{[\varpi]^{1/b}}\right\rangle$ in $\Lambda_{R_0,[0,\infty]}\left\langle\frac{u}{[\varpi]^{1/b}}\right\rangle\left[\frac{1}{[\varpi]}\right]$ is a ring of definition.  
\end{proposition}
\begin{remark}
	We write $\widetilde\Lambda_{R_0,[a,b],0,\lambda}$ to denote the image of the ring of definition 
	\[	(R_0\htimes\widetilde\Lambda_{[0,\frac{b}{\lambda}]})\left\langle\frac{u}{[\varpi]^{1/b}},\frac{[\varpi]^{1/a}}{u}\right\rangle\subset (R_0\htimes_{\Z_p}\widetilde\Lambda_{[0,\frac{b}{\lambda}]}^\circ)\left\langle \frac{u}{[\varpi]^{1/b}}\right\rangle\left[\frac{1}{[\varpi]}\right]	\]
	inside $\widetilde\Lambda_{R_0,[a,b]}$.
\end{remark}
\begin{proof}
	Since $R_0$ is a $D_\lambda^\circ$-algebra, we may compute that
	\[      \left(\frac{p^m}{u^{m'}}\right)\cdot \left(\frac{u}{[\varpi]^{1/b}}\right)^{m'}= \frac{p^m}{[\varpi]^{m'/b}}=\left(\frac{p}{[\varpi]^{\lambda/b}}\right)^m      \]
in $(R_0\htimes\widetilde\A_{\mathrm{inf}})\left\langle\frac{u}{[\varpi]^{1/b}}\right\rangle\left[\frac{1}{[\varpi]}\right]$. 
It follows that $\frac{p}{[\varpi]^{\lambda/b}}$ is power-bounded in $(R_0\htimes\A_{\mathrm{inf}})\left\langle\frac{u}{[\varpi]^{1/b}}\right\rangle\left[\frac{1}{[\varpi]}\right]$, hence in $(R_0\htimes\widetilde\Lambda_{[0,r]})\left\langle\frac{u}{[\varpi]^{1/b}}\right\rangle\left[\frac{1}{[\varpi]}\right]$, and we obtain the desired continuous maps
\[      (R_0\htimes_{\Z_p}\widetilde\Lambda_{[0,\frac{b}{\lambda}]}^\circ)\left\langle \frac{u}{[\varpi]^{1/b}}\right\rangle\left[\frac{1}{[\varpi]}\right]\rightarrow(R_0\htimes_{\Z_p}\widetilde\Lambda_{[0,r]}^\circ)\left\langle\frac{u}{[\varpi]^{1/b}}\right\rangle\left[\frac{1}{[\varpi]}\right]        \]

On the other hand, if $r\geq \frac{b}{\lambda}$, then $(R_0\htimes_{\Z_p}\widetilde\Lambda_{[0,\frac{b}{\lambda}]}^\circ)\left\langle\frac{u}{[\varpi]^{1/b}}\right\rangle\left[\frac{1}{[\varpi]}\right]$ is an $R_0\htimes_{\Z_p}\widetilde\Lambda_{[0,r]}^\circ$-algebra in which $[\varpi]$ is invertible and $\frac{u}{[\varpi]^{1/b}}$ is power-bounded. Therefore, there is a canonical continuous map 
\[	(R_0\htimes_{\Z_p}\widetilde\Lambda_{[0,r]}^\circ)\left\langle\frac{u}{[\varpi]^{1/b}}\right\rangle\left[\frac{1}{[\varpi]}\right]\rightarrow(R_0\htimes\widetilde\Lambda_{[0,\frac{b}{\lambda}]}^\circ)\left\langle \frac{u}{[\varpi]^{1/b}}\right\rangle\left[\frac{1}{[\varpi]}\right]	\]
These maps are clearly inverse to one another, and since they are continuous, the images of $(R_0\htimes_{\Z_p}\widetilde\Lambda_{[0,\frac{b}{\lambda}]}^\circ)\left\langle\frac{u}{[\varpi]^{1/b}}\right\rangle$ and $(R_0\htimes_{\Z_p}\widetilde\Lambda_{[0,r]}^\circ)\left\langle\frac{u}{[\varpi]^{1/b}}\right\rangle$ are both rings of definition.
\end{proof}

In other words, for any $r\in\Q_{\geq 0}$ and any choice of $b\leq r\lambda$, the pre-adic space 
\[	\Spa\left((R_0\htimes_{\Z_p}\widetilde\Lambda_{[0,r]}^\circ)\left\langle \frac{u}{[\varpi]^{1/b}}\right\rangle\left[\frac{1}{[\varpi]}\right], (R_0\htimes_{\Z_p}\widetilde\Lambda_{[0,r]}^\circ)\left\langle \frac{u}{[\varpi]^{1/b}}\right\rangle^{\mathrm{int}}\right)	\]
is isomorphic to the rational localization $\Spa\Lambda_{R_0,[0,\infty]}\left\langle\frac{u}{[\varpi]^{1/b}}\right\rangle$.  

\begin{corollary}
The pre-adic space $\Spa R_0\times_{\Spa(\Z_p)}\mathcal{Y}_{[0,\infty)}$ is exhausted by open subspaces of the form $\Spa\Lambda_{R_0,[0,\infty]}\left\langle\frac{u}{[\varpi]^{1/b}}\right\rangle$ for positive rational numbers $b\in\Q_{>0}$.
\end{corollary}

Similarly, we prove
\begin{corollary}
The pre-adic space $\Spa R\times_{\Spa(\Z_p)}\mathcal{Y}_{[0,\infty)}$ is exhausted by open subspaces of the form $\Spa\Lambda_{R_0,[0,\infty]}\left\langle\frac{u}{[\varpi]^{1/b}},\frac{[\varpi]^{1/a}}{u}\right\rangle$, for positive rational numbers $a,b\in\Q_{>0}$ with $a\leq b$.
\end{corollary}

\begin{remark}
If $R$ is a ring such that $p=0$, the natural maps 
\[	(R_0\htimes\widetilde\Lambda_{[0,\frac{b}{\lambda}]}^\circ)\left\langle\frac{u}{[\varpi]^{1/b}}\right\rangle\rightarrow \Lambda_{R_0,[0,\infty]}\left\langle\frac{u}{[\varpi]^{1/b}}\right\rangle\left[\frac{1}{[\varpi]}\right]	\]
are not injective, because $\widetilde\Lambda_{[0,\frac{b}{\lambda}]}^\circ/p$ has $[\varpi]$-torsion.  However, the images still yield rings of definition.  In fact, $\widetilde\Lambda_{R_0,[a,b],0,\lambda}=\widetilde\Lambda_{R_0,[a,b],0}$.
\end{remark}

If $b\in (0,\infty)$, $\widetilde\Lambda_{R_0,[a,b]}$ is a Tate ring with pseudo-uniformizer $[\varpi]$.  We may define valuations $v_{R,[a,b]}$ and $v_{R,[a,b],\lambda}$ on it via
\[	v_{R,[a,b]}(x):=\sup_{\alpha\in\C_p^\flat:[\alpha]x\in \widetilde\Lambda_{R_0,[a,b],0}}-v_{\C_p^\flat}(\alpha)	\]
and
\[	v_{R,[a,b],\lambda}(x):=\sup_{\alpha\in\C_p^\flat:[\alpha]x\in \widetilde\Lambda_{R_0,[a,b],0,\lambda}}-v_{\C_p^\flat}(\alpha)	\]
When $a=0$, we also denote these valuations by $v_{R,b}$ and $v_{R,b,\lambda}$, respectively.

If $[a,b]\subset(0,\infty)$, then ideals of definition of $\widetilde\Lambda_{R_0,[a,b],0}$ are principal, and $u$ and $[\varpi]$ both generate ideals of definition.  In particular, $\widetilde\Lambda_{R_0,[a,b]}=\widetilde\Lambda_{R,[a,b]}$ is Tate with pseudo-uniformizers $u$ and $[\varpi]$.

If $(R,R^\circ)=(\Q_p,\Z_p)$ with $u=p$, this definition recovers $\widetilde\Lambda_{[a,b]}$.

We have seen that if $p^mu^{-m'}$ is power-bounded in $R$, then there is a natural map $(R_0\htimes\widetilde\Lambda_{[0,\frac{b}{\lambda}]}^\circ)\left\langle\frac{u}{[\varpi]^{1/b}}\right\rangle\rightarrow \widetilde\Lambda_{R_0,[0,b]}$, whose image is a ring of definition.  In particular, the image of $\widetilde\Lambda_{[0,\frac{b}{\lambda}]}^\circ$ is bounded.  The next proposition shows that $\widetilde\Lambda_{R_0,[a,b],0,\lambda}$ behaves better than $\widetilde\Lambda_{R_0,[a,b],0}$ for making some comparisons with the classical story:
\begin{proposition}\label{prop: valuation bound}
	If $x\in \widetilde\Lambda_{[0,\frac{b}{\lambda}]}$, then $v_{R,b}(x)\geq v_{b/\lambda}(x)-\frac{(m-1)\lambda}{b}$.  If $R_0=D_\lambda^\circ$ for some $\lambda=\frac{m'}{m}$, then $v_{b/\lambda}(x)\geq v_{R,b}(x)$, as well.  Similarly, $v_{R,b,\lambda}(x)\geq v_{b/\lambda}(x)$, and if $R_0=D_\lambda^\circ$, we have equality.
\end{proposition}
\begin{proof}
It suffices to consider the image of $\frac{p^k}{[\varpi]^{k\lambda/b}}$ for all $k\geq 0$.  Since the image of $\frac{p^k}{[\varpi]^{k\lambda/b}}$ is in $\widetilde\Lambda_{R_0,[0,b],0,\lambda}$, we see that $v_{R,b,\lambda}(x)\geq v_b(x)$.  Additionally, since $\left(\frac{p}{[\varpi]^{\lambda/b}}\right)^m\in \widetilde\Lambda_{R_0,[0,b],0}$, we may assume $0\leq k< m$. But then we may multiply by $[\varpi]^{k\lambda/b}$ to get an element of $\widetilde\Lambda_{R_0,[0,b],0}$, so 
\[	v_{R,b}\left(\frac{p^k}{[\varpi]^{k\lambda/b}}\right)\geq -\frac{k\lambda}{b}\geq -\frac{(m-1)\lambda}{b}	\]

When $R_0=D_\lambda^\circ$, we observe that $\left(\frac{p}{[\varpi]^{\lambda/b}}\right)^m=\left(\frac{p^m}{u^{m'}}\right)\cdot\left(\frac{u}{[\varpi]^{1/b}}\right)^{m'}$, and we claim that 
\[	v_{R,b}\left(\left(\frac{p^m}{u^{m'}}\right)\cdot\left(\frac{u}{[\varpi]^{1/b}}\right)^{m'}\right)=0	\]
and 
\[	v_{R,b,\lambda}\left(\left(\frac{p^m}{u^{m'}}\right)\cdot\left(\frac{u}{[\varpi]^{1/b}}\right)^{m'}\right)=0	\]
Then 
\[	0\geq m\cdot v_{R,b}\left(\frac{p}{[\varpi]^{\lambda/b}}\right), m\cdot v_{R,b,\lambda}\left(\frac{p}{[\varpi]^{\lambda/b}}\right)	\]
so 
\[	0\geq v_{R,b}\left(\frac{p}{[\varpi]^{\lambda/b}}\right), v_{R,b,\lambda}\left(\frac{p}{[\varpi]^{\lambda/b}}\right)	\]
On the other hand, $v_{b/\lambda}\left(\frac{p}{[\varpi]^{\lambda/b}}\right)=0$ and the result follows.

To prove the claims, we write 
\[	\widetilde\Lambda_{D_\lambda^\circ,[0,b],0}\cong\left(\left(\Z_p[\![u]\!]\left\langle X\right\rangle/(u^{m'}X-p^m)\right)\htimes\A_{\mathrm{inf}}\right)\left\langle Y\right\rangle/([\varpi]^{1/b}Y-u)	\]
and $\widetilde\Lambda_{D_\lambda^\circ,[0,b],0,\lambda}$ as the quotient of 
\begin{equation*}
\resizebox{\displaywidth}{!}{$
\left(\left(\Z_p[\![u]\!]\left\langle X\right\rangle/(u^{m'}X-p^m)\right)\htimes\left(\A_{\mathrm{inf}}\left\langle X'\right\rangle/([\varpi]^{\lambda/b}X'-p)\right)\right)\left\langle Y\right\rangle/([\varpi]^{1/b}Y-u)
$}
\end{equation*}
by $[\varpi]$-torsion (in particular, $XY^{m'}={X'}^m$). 
Then the claim is that if we reduce modulo the ideal $[\mathfrak{m}_{\mathscr{O}_{\C_p}^\flat}]$, the image of $X\cdot Y^{m'}$ is non-zero.  We actually compute in 
\[	\widetilde\Lambda_{D_\lambda^\circ,[0,b],0}/(p,[\mathfrak{m}_{\mathscr{O}_{\C_p}^\flat}])\cong \overline\F_p[X,Y]	\]
and
\[	\widetilde\Lambda_{D_\lambda^\circ,[0,b],0,\lambda}/[\mathfrak{m}_{\mathscr{O}_{\C_p}^\flat}]\cong \overline\F_p[X,X',Y]/(XY^{m'}-{X'}^m)	\]
which are visibly domains, so we are done.
\end{proof}
In the special case $p\in uR_0$ (in particular, if $p=0$ in $R$), this implies that there is a map $\widetilde\Lambda_{[0,b]}\rightarrow \widetilde\Lambda_{R,[0,b]}$ such that $v_{R,b}(x)\geq v_b(x)$ for all $x\in\widetilde\Lambda_{[0,b]}$.

\begin{remark}
The rings $\widetilde\Lambda_{R_0,[a,b]}$, $\widetilde\Lambda_{R_0,[a,b],0}$, and $\widetilde\Lambda_{R_0,[a.b],0,\lambda}$ depend on the choice of ring of definition $R_0\subset R$, but $\widetilde\Lambda_{R,[a,b]}$ and $\widetilde\Lambda_{R,[a,b],0}$ do not.  However, both depend on our choice of $u\in R$.
\end{remark}
\begin{proposition}
Let $f:R\rightarrow R'$ be a homomorphism of pseudoaffinoid rings, and let $R_0':=f(R_0)$ and $u':=f(u)$.  Then $R_0'\htimes_{R_0}\widetilde\Lambda_{R_0,[a,b]}=\widetilde\Lambda_{R_0',[a,b]}$.
\end{proposition}
\begin{proof}
By~\cite[Lemma 2.2.5]{johansson-newton}, $f$ is topologically of finite type , so $R_0'$ is a ring of definition of $R'$ formally of finite type over $\Z_p$.  Then 
\begin{equation*}
\resizebox{\displaywidth}{!}{$
\begin{split}
R_0'\htimes_{R_0}\widetilde\Lambda_{R_0,[a,b]} &= R_0'\htimes_{R_0}\Lambda_{R_0,[0,\infty]}\left\langle X,Y\right\rangle/([\varpi]^{1/b}X-u, uY-[\varpi]^{1/a}, XY-[\varpi]^{1/a-1/b})	\\
&=\Lambda_{R_0',[0,\infty]}\left\langle X,Y\right\rangle/([\varpi]^{1/b}X-u', u'Y-[\varpi]^{1/a}, XY-[\varpi]^{1/a-1/b}) 	\\
&= \widetilde\Lambda_{R_0',[a,b]}
\end{split}
$}
\end{equation*}
\end{proof}

We do not know whether $\mathcal{Y}_{R_0}$ and $\mathcal{Y}_R$ are adic spaces (even though $\Spa R_0$, $\Spa R$, and $\mathcal{Y}$ are).  However, we have the following partial result:
\begin{proposition}
Suppose $\{\Spa R_i\}_i$ is an affinoid cover of $\Spa R$, and let $\Spa R_{ij}:=\Spa R_i\cap\Spa R_j$.  Then if $[a,b]\subset [0,\infty)$ we have a strict exact sequence
\[	0\rightarrow \widetilde\Lambda_{R,[a,b]}\rightarrow\prod_i\widetilde\Lambda_{R_i,[a,b]}\rightarrow\prod_{i,j}\widetilde\Lambda_{R_{ij},[a,b]}	\]
\end{proposition}
\begin{proof}
We may assume that $\{\Spa R_i\}_i$ is a finite cover by rational subspaces of $\Spa R$, where $(R_i,R_i^\circ)=(R,R^\circ)\left\langle\frac{f_0,\ldots,f_n}{f_i} \right\rangle$ for some finite set $f_0,\ldots ,f_n\in R$ which generate the unit ideal.  We may further assume that $f_i\in R_0$ for all $i$.  The ring of definition $R_0$ is admissible in the sense of~\cite{bosch-lut}, so we may consider the scheme-theoretic blowing up $X\rightarrow \Spec R_0$ and the admissible formal blowing up $\mathcal{X}\rightarrow\Spf(R_0)$ along the ideal $(f_0,\ldots,f_n)$.  Then $\mathscr{O}_{X}\otimes_{R_0}\widetilde\Lambda_{R_0,[a,b],0}$ is a quasi-coherent $\mathscr{O}_{X}$-module, so we have an exact sequence
\begin{equation*}
	\begin{split}
0\rightarrow \Gamma(X,\mathscr{O}_{X})\otimes_{R_0}\widetilde\Lambda_{R_0,[a,b],0}\rightarrow & \prod_i R_0\left[\frac{f_0}{f_i},\ldots,\frac{f_n}{f_i} \right]\otimes_{R_0}\widetilde\Lambda_{R_0,[a,b],0} 	\\
&\rightarrow \prod_{i,j}R_0\left[\frac{f_0}{f_if_j},\ldots,\frac{f_n}{f_if_j} \right]\otimes_{R_0}\widetilde\Lambda_{R_0,[a,b],0}
\end{split}
\end{equation*}

The $[\varpi]$-adic completion
\begin{equation*}
	\begin{split}
0\rightarrow \Gamma(\mathcal{X},\mathscr{O}_{\mathcal{X}})\htimes_{R_0}\widetilde\Lambda_{R_0,[a,b],0}\rightarrow &\prod_i R_0\left\langle\frac{f_0}{f_i},\ldots,\frac{f_n}{f_i} \right\rangle\htimes_{R_0}\widetilde\Lambda_{R_0,[a,b],0}	\\
&\rightarrow \prod_{i,j}R_{0}\left\langle\frac{f_0}{f_i},\ldots,\frac{f_n}{f_i} \right\rangle\htimes_{R_0}\widetilde\Lambda_{R_0,[a,b],0}
\end{split}
\end{equation*}
is exact, as well, and since $\Gamma(\mathcal{X},\mathscr{O}_{\mathcal{X}})$ is $R_0$-finite and satisfies 
\[	\Gamma(\mathcal{X},\mathscr{O}_{\mathcal{X}})\left[\frac 1 u\right]=R	\]
the result now follows from inverting $u$.
\end{proof}

We have a similar sheaf property with respect to rational localization on $\mathcal{Y}$. 
Since $R_0\htimes_{\Z_p}\A_{\mathrm{inf}}$ has no $u$- or $[\varpi]$-torsion, if $[a,b]\supset [a',b']$, there is an injective map
\begin{equation*}
\Lambda_{R_0,[0,\infty]}\left[\frac{u}{[\varpi]^{1/b}},\frac{[\varpi]^{1/a}}{u}\right]\rightarrow \Lambda_{R_0,[0,\infty]}\left[\frac{u}{[\varpi]^{1/b'}},\frac{[\varpi]^{1/a'}}{u}\right]
\end{equation*}
(since both rings are subrings of $\Lambda_{R_0,[0,\infty]}\left[\frac{1}{[\varpi]},\frac{1}{u}\right]$).

\begin{proposition}\label{prop: inj overconvergent}
The above map extends to an injection 
\[	\Lambda_{R_0,[0,\infty]}\left\langle\frac{u}{[\varpi]^{1/b}},\frac{[\varpi]^{1/a}}{u}\right\rangle\rightarrow \Lambda_{R_0,[0,\infty]}\left\langle\frac{u}{[\varpi]^{1/b'}},\frac{[\varpi]^{1/a'}}{u}\right\rangle	\]
and therefore an injection $\widetilde\Lambda_{R,[a,b]}\rightarrow\widetilde\Lambda_{R,[a',b']}$
\end{proposition}
\begin{proof}
As in the proof of~\cite[Lemme 2.5]{berger}, the map $\widetilde\Lambda_{R_0,[a,b],0}\rightarrow\widetilde\Lambda_{R_0,[a',b'],0}$ factors as $\widetilde\Lambda_{R_0,[a,b],0}\rightarrow \widetilde\Lambda_{R_0,[a,b'],0}\rightarrow\widetilde\Lambda_{R_0,[a',b'],0}$, and we may assume that either $a=a'$ or $b=b'$.  We treat the case $a=a'=0$ here; the other cases follow as in~\cite{berger}.

We need to show that the natural map 
\[	\Lambda_{R_0,[0,\infty]}\left\langle X\right\rangle/([\varpi]^{1/b}X-u)\rightarrow \Lambda_{R_0,[0,\infty]}\left[X\right]_u^\wedge/([\varpi]X-1)	\]
is injective.  This map carries $f(X)\in \Lambda_{R_0,[0,\infty]}\left\langle X\right\rangle$ to $f(u[\varpi]^{1-1/b}X)$; to show it is injective, we need to check that if $f(u[\varpi]^{1-1/b}X)$ is a multiple of $[\varpi]X-1$ in $\Lambda_{R_0,[0,\infty]}\left[ X\right]_u^\wedge$, then $f(uX)$ is a multiple of $u[\varpi]X-u$ in $\Lambda_{R_0,[0,\infty]}\left\langle u[\varpi]^{1-1/b}X\right\rangle$.

Writing $f(uX)=([\varpi]X-1)g(X)$, where $g(X)=\sum_{j=0}^\infty c_jX^j$ and $c_j\rightarrow 0$ $u$-adically, we need to show that $c_j\in u^{j+1}[\varpi]^{j(1-1/b)}\Lambda_{R_0,[0,\infty]}$ for all $j\geq 0$.  We may also write $f(uX)=\sum_{j=0}^\infty d_j(u[\varpi]^{1-1/b}X)^j$, where $d_j\rightarrow 0$ $u$-adically.  Then for $j\geq 1$, we have $d_ju^j[\varpi]^{j(1-1/b)}=[\varpi]c_{j-1}-c_j$.

Since the $c_j$ tend to $0$ $u$-adically, for each $j$ there is some $N_j\gg j$ such that $c_{N_j}$ is a multiple of $u^j$.  This implies that $[\varpi]c_{N_j-1}$ is also a multiple of $u^j$, and since $(R_0/u^j)\htimes\A_{\mathrm{inf}}$ has no $[\varpi]$-torsion, $c_{N_j-1}$ itself is a multiple of $u^j$.  Repeating this argument, we see that $c_{j-1}$ is a multiple of $u^j$.

We may write $c_j=u^{j+1}c_j'$.  Since $\Lambda_{R_0,[0,\infty]}$ has no $u$-torsion, we have $d_j[\varpi]^{j(1-1/b)}=[\varpi]c_{j-1}'-uc_j'$ for all $j\geq 1$, which implies that $uc_1'$ is a multiple of $[\varpi]^{1-1/b}$.  But $R_0\htimes(\A_{\mathrm{inf}}/[\varpi]^{1-1/b})$ has no $u$-torsion, so $c_1'$ itself is a multiple of $[\varpi]^{1-1/b}$.  We now proceed by induction on $j$; if $c_{j-1}'$ is a multiple of $[\varpi]^{(j-1)(1-1/b)}$, then $uc_j'$ is a multiple of $[\varpi]^{j(1-1/b)}$, which implies that $c_j'$ itself is a multiple of $[\varpi]^{j(1-1/b)}$.
\end{proof}

\begin{proposition}\label{prop: partial sheafiness}
Suppose $[a,b]\neq[0,\infty]$. Then there is an exact sequence (of $R_0$-modules)
\begin{equation*}
	\begin{split}
0\rightarrow \Lambda_{R_0,[0,\infty]}\rightarrow \Lambda_{R_0,[0,\infty]}\left\langle\frac{u}{[\varpi]^{1/b}}\right\rangle\oplus&\Lambda_{R_0,[0,\infty]}\left\langle\frac{[\varpi]^{1/a}}{u}\right\rangle	\\
&\rightarrow \Lambda_{R_0,[0,\infty]}\left\langle\frac{u}{[\varpi]^{1/b}},\frac{[\varpi]^{1/a}}{u}\right\rangle\rightarrow 0
\end{split}
\end{equation*}
\end{proposition}
\begin{proof}
This is an adaptation of the proof of ~\cite[Lemme 2.15]{berger}.  We first treat the case where $a=b$.

The natural map
\begin{equation*}
	\begin{split}
\Lambda_{R_0,[0,\infty]}\left\langle\frac{u}{[\varpi]^{1/a}}\right\rangle\oplus\Lambda_{R_0,[0,\infty]}\left\langle\frac{[\varpi]^{1/a}}{u}\right\rangle\rightarrow \Lambda_{R_0,[0,\infty]}\left\langle\frac{u}{[\varpi]^{1/a}},\frac{[\varpi]^{1/a}}{u}\right\rangle
\end{split}
\end{equation*}
is clearly surjective, so it only remains to check that the kernel is exactly $R_0\htimes_{\Z_p}\A_{\mathrm{inf}}$.

We first verify this modulo $u$. 
\begin{equation*}
\resizebox{\displaywidth}{!}{$
\begin{split}
\left(\Lambda_{R_0,[0,\infty]}\left\langle\frac{u}{[\varpi]^{1/a}}\right\rangle\right)/(u)&\cong \Lambda_{R_0,[0,\infty]}\left[Y\right]/([\varpi]^{1/a}Y,u)	\\
\left(\Lambda_{R_0,[0,\infty]}\left\langle\frac{[\varpi]^{1/a}}{u}\right\rangle\right)/(u)&\cong (R_0/u)\otimes_{\Z_p}\A_{\mathrm{inf}}\left[X\right]/([\varpi]^{1/a})	\\
\left(\Lambda_{R_0,[0,\infty]}\left\langle\frac{u}{[\varpi]^{1/a}},\frac{[\varpi]^{1/a}}{u}\right\rangle\right)/(u) &\cong (R_0/u)\otimes_{\Z_p}\A_{\mathrm{inf}}\left[X,X^{-1}\right]/([\varpi]^{1/a})
\end{split}
$}
\end{equation*}
Moreover, the map 
\[	\left(\Lambda_{R_0,[0,\infty]}\left\langle\frac{u}{[\varpi]^{1/a}}\right\rangle\right)/(u)\rightarrow \left(\Lambda_{R_0,[0,\infty]}\left\langle\frac{u}{[\varpi]^{1/a}},\frac{[\varpi]^{1/a}}{u}\right\rangle\right)/(u)	\]
 is given by $Y\mapsto X^{-1}$, and it factors through 
\[	\left(\Lambda_{R_0,[0,\infty]}\left\langle\frac{u}{[\varpi]^{1/a}}\right\rangle\right)/(u,[\varpi]^{1/a})=(R_0/u)\otimes_{\Z_p}\A_{\mathrm{inf}}[Y]/([\varpi]^{1/a})	\]

But the intersection 
\begin{equation*}
\left((R_0/u)\otimes\A_{\mathrm{inf}}/([\varpi]^{1/a})\right)[X]\cap \left((R_0/u)\otimes\A_{\mathrm{inf}}/([\varpi]^{1/a})\right)[X^{-1}]
\end{equation*}
inside $\left((R_0/u)\otimes\A_{\mathrm{inf}}/([\varpi]^{1/a})\right)[X,X^{-1}]$ is just $(R_0/u)\otimes\A_{\mathrm{inf}}/([\varpi]^{1/a})$.  It follows that the image of 
\[	\Lambda_{R_0,[0,\infty]}\left\langle\frac{u}{[\varpi]^{1/a}}\right\rangle\cap \Lambda_{R_0,[0,\infty]}\left\langle\frac{[\varpi]^{1/a}}{u}\right\rangle	\]
is contained in the image of $\Lambda_{R_0,[0,\infty]}$.

Thus, given $x\in \Lambda_{R_0,[0,\infty]}\left\langle\frac{u}{[\varpi]^{1/a}}\right\rangle\cap\Lambda_{R_0,[0,\infty]}\left\langle\frac{[\varpi]^{1/a}}{u}\right\rangle$, there is some element $y\in \Lambda_{R_0,[0,\infty]}$ such that $x-y\in (u)\Lambda_{R_0,[0,\infty]}\left\langle\frac{u}{[\varpi]^{1/a}},\frac{[\varpi]^{1/a}}{u}\right\rangle$.  By considering the reductions of all of these rings modulo $u$, this implies that 
\[	x-y\in u\Lambda_{R_0,[0,\infty]}\left\langle\frac{[\varpi]^{1/a}}{u}\right\rangle	\]
and 
\[	x-y\in u\Lambda_{R_0,[0,\infty]}\left\langle\frac{u}{[\varpi]^{1/a}}\right\rangle+[\varpi]^{1/a}\Lambda_{R_0,[0,\infty]}	\]

There is therefore some $z\in [\varpi]^{1/a}\Lambda_{R_0,[0,\infty]}$ such that 
\[	x-y-z\in u\Lambda_{R_0,[0,\infty]}\left\langle\frac{u}{[\varpi]^{1/a}}\right\rangle	\]
Since $\frac{[\varpi]^{1/a}}{u}\in \Lambda_{R_0,[0,\infty]}\left\langle\frac{[\varpi]^{1/a}}{u}\right\rangle$ and the intersection 
\[	\Lambda_{R_0,[0,\infty]}\left\langle\frac{u}{[\varpi]^{1/a}}\right\rangle\cap\Lambda_{R_0,[0,\infty]}\left\langle\frac{[\varpi]^{1/a}}{u}\right\rangle	\]
inside $\Lambda_{R_0,[0,\infty]}\left\langle \frac{u}{[\varpi]^{1/a}},\frac{[\varpi]^{1/a}}{u}\right\rangle$ has no $u$-torsion, we have 
\[	x-y-z\in u\left(\Lambda_{R_0,[0,\infty]}\left\langle\frac{u}{[\varpi]^{1/a}}\right\rangle\cap\Lambda_{R_0,[0,\infty]}\left\langle\frac{[\varpi]^{1/a}}{u}\right\rangle\right)	\]
Since $\Lambda_{R_0,[0,\infty]}$ is $(u,[\varpi]^{1/a})$-adically separated and complete, we can iterate this argument, and the conclusion follows.

To handle the general case, we consider the diagram
\begin{equation*}
\resizebox{\displaywidth}{!}{
$
\begin{tikzcd}[ampersand replacement=\&]
0 \ar[r]	\& \Lambda_{R_0,[0,\infty]}\arrow[hook,r]\arrow[equal,d]	\& \Lambda_{R_0,[0,\infty]}\left\langle\frac{u}{[\varpi]^{1/b}}\right\rangle\oplus\Lambda_{R_0,[0,\infty]}\left\langle\frac{[\varpi]^{1/a}}{u}\right\rangle\arrow[two heads,r]\arrow[hook,d]	\& \Lambda_{R_0,[0,\infty]}\left\langle\frac{u}{[\varpi]^{1/b}},\frac{[\varpi]^{1/a}}{u}\right\rangle \ar[r]\arrow[hook,d]	\& 0	\\
0 \ar[r]	\& \Lambda_{R_0,[0,\infty]}\arrow[hook,r]	\& \Lambda_{R_0,[0,\infty]}\left\langle\frac{u}{[\varpi]^{1/a}}\right\rangle\oplus\Lambda_{R_0,[0,\infty]}\left\langle\frac{[\varpi]^{1/a}}{u}\right\rangle\arrow[two heads,r]	\& \Lambda_{R_0,[0,\infty]}\left\langle\frac{u}{[\varpi]^{1/a}},\frac{[\varpi]^{1/a}}{u}\right\rangle \arrow[r]	\& 0
\end{tikzcd}
$
}
\end{equation*}
Since the bottom row is exact, and the top row is exact except possibly in the middle, and the vertical arrows are injections, a diagram chase shows that the top row is exact.
\end{proof}

\begin{corollary}
Suppose $I_1=[a_1,b_1]$ and $I_2=[a_2,b_2]$ are intervals such that $I_1\cap I_2\neq \emptyset$.  Then $\widetilde\Lambda_{R_0,I_1,0}\cap\widetilde\Lambda_{R_0,I_2,0}=\widetilde\Lambda_{R_0,I_1\cup I_2,0}$ (where the intersection is taken inside $\widetilde\Lambda_{R_0,I_1\cap I_2,0}$).
\end{corollary}
\begin{proof}
We may assume that $a_1\leq a_2\leq b_1\leq b_2$, so that $I_1\cap I_2=[a_2,b_1]$ and $I_1\cup I_2=[a_1,b_2]$. 
If $f\in \widetilde\Lambda_{R_0,I_1,0}\cap\widetilde\Lambda_{R_0,I_2,0}$, we may write
\[	f=g_1+h_1=g_2+h_2	\]
with 
\begin{equation*}
	\begin{split}
		g_1\in\widetilde\Lambda_{R_0,[a_1,\infty],0}\qquad \text{ and } \qquad h_1\in\widetilde\Lambda_{R_0,[0,b_1],0}	\\
		g_2\in\widetilde\Lambda_{R_0,[a_2,\infty],0}\qquad \text{ and } \qquad h_2\in\widetilde\Lambda_{R_0,[0,b_2],0}
	\end{split}
\end{equation*}
Then $(g_1-g_2)+(h_1-h_2)=0$, with 
\[	g_1-g_2\in\widetilde\Lambda_{R_0,[a_2,\infty],0}\qquad \text{ and }\qquad h_1-h_2\in\widetilde\Lambda_{R_0,[0,b_1],0}	\]
	It follows from the previous proposition that 
	\[	g_1-g_2=h_2-h_1=f'\in \Lambda_{R_0,[0,\infty]}	\]
	Then $f=(g_1-f')+(h_1+f')=(g_1-f')+h_2$; since $g_1-f'\in\widetilde\Lambda_{R_0,[a_1,\infty],0}$ and $h_2\in \widetilde\Lambda_{R_0,[0,b_2],0}$, we are done.
\end{proof}

\subsection{The action of $H_K$ on $\widetilde\Lambda_{R,[a,b]}$}\label{subsection: H_K-invts}

Since $v_{\C_p^\flat}([\varpi])=\frac{p-1}{p}v_{\C_p^\flat}([\overline\pi])$, we define $s:(0,\infty)\rightarrow (0,\infty)$ via $s(a)=\frac{p-1}{pa}$.  From now on, we assume that $s(a), s(b)\in p^{\Z}$.  Then we can rewrite $\widetilde\Lambda_{R_0,[a,b],0}$ as the rational localization $\Lambda_{R_0,[0,\infty]}\left\langle\frac{u}{[\overline\pi]^{s(b)}},\frac{[\overline\pi]^{s(a)}}{u}\right\rangle$.  Since $H_K$ acts trivially on $u$ and $[\overline\pi]$, this implies that the action of $H_K$ on $\A_{\mathrm{inf}}$ induces an action on $\widetilde\Lambda_{R,[a,b]}$.

If $(R,R^\circ)=(\Q_p,\Z_p)$ and $u=p$, it is easy to read off from the Witt vector description of $(\widetilde\Lambda_{[0,r]},\widetilde\Lambda_{[0,r]}^\circ)$ that
\[	\left(\widetilde\Lambda_{[0,r]},\widetilde\Lambda_{[0,r]}^\circ\right)^{H_K}=\left(\A_{\mathrm{inf}}^{H_K}\left\langle\frac{p}{[\overline\pi]^{s(r)}}\right\rangle\left[\frac{1}{[\overline\pi]}\right], \A_{\mathrm{inf}}^{H_K}\left\langle\frac{p}{[\overline\pi]^{s(r)}}\right\rangle\right)	\]
Moreover, it follows from ~\cite[Lemme 2.29]{berger} that 
\begin{equation*}
\resizebox{\displaywidth}{!}{$
\left(\widetilde\Lambda_{[p^{-n},\infty]},\widetilde\Lambda_{[p^{-n},\infty]}^\circ\right)^{H_K}=\left(\A_{\mathrm{inf}}^{H_K}\left\langle\frac{[\overline\pi]^{(p-1)p^{n-1}}}{p}\right\rangle\left[\frac{1}{[\overline\pi]}\right],\A_{\mathrm{inf}}^{H_K}\left\langle\frac{[\overline\pi]^{(p-1)p^{n-1}}}{p}\right\rangle\right)
$}
\end{equation*}

The same argument as in the proof of Proposition~\ref{prop: partial sheafiness} shows the following:
\begin{proposition}
There is an exact sequence (of $R_0$-modules)
\begin{equation*}
	\begin{split}
%\resizebox{\displaywidth}{!}{ $
0\rightarrow \Lambda_{R_0,[0,\infty]}^{H_K}\rightarrow \Lambda_{R_0,[0,\infty]}^{H_K}\left\langle\frac{u}{[\overline\pi]^{s(b)}}\right\rangle&\oplus \Lambda_{R_0,[0,\infty]}^{H_K}\left\langle\frac{[\overline\pi]^{s(a)}}{u}\right\rangle	\\
&\rightarrow\Lambda_{R_0,[0,\infty]}^{H_K}\left\langle\frac{u}{[\overline\pi]^{s(b)}},\frac{[\overline\pi]^{s(a)}}{u}\right\rangle\rightarrow 0
%$}
\end{split}
\end{equation*}
\end{proposition}

\begin{lemma}
There is a connecting homomorphism 
\[	\widetilde\Lambda_{R_0,[a,b],0}^{H_K}\rightarrow\H^1(H_K,\Lambda_{R_0,[0,\infty]})	\]
\end{lemma}
\begin{proof}
We need to construct a continuous set-theoretic section 
\[	\widetilde\Lambda_{R_0,[a,b],0}\rightarrow \widetilde\Lambda_{R_0,[0,b],0}\oplus \widetilde\Lambda_{R_0,[a,\infty],0}	\]
If $[a,b]=[0,b]$, we choose the map $(\id, 0)$.  If $a\neq 0$ but $b=\infty$, we choose the map $(0,\id)$.

Otherwise, $[a,b]\subset (0,\infty)$ and we have an exact sequence 
\begin{equation*}
	\begin{split}
0\rightarrow \Lambda_{R_0,[0,\infty]}\rightarrow \Lambda_{R_0,[0,\infty]}\left\langle\frac{u}{[\varpi]^{1/b}}\right\rangle&\oplus\Lambda_{R_0,[0,\infty]}\left\langle\frac{[\varpi]^{1/a}}{u}\right\rangle	\\
&\rightarrow \Lambda_{R_0,[0,\infty]}\left\langle\frac{u}{[\varpi]^{1/b}},\frac{[\varpi]^{1/a}}{u}\right\rangle\rightarrow 0
\end{split}
\end{equation*}
and we need to produce a continuous set-theoretic section 
\[	s:\widetilde\Lambda_{R,[a,b],0}\rightarrow \widetilde\Lambda_{R,[0,b],0}\oplus\widetilde\Lambda_{R,[a,\infty],0}	\]
We first construct a continuous set-theoretic map 
\[	s_1:\widetilde\Lambda_{R_0,[a,b],0}\rightarrow \widetilde\Lambda_{R_0,[a,\infty],0}	\]
Both rings are $u$-adically separated and complete, so it suffices to construct a set-theoretic map $\overline{s}_1:\widetilde\Lambda_{R_0,[a,b],0}/u\rightarrow \widetilde\Lambda_{R_0,[a,\infty],0}/u$. We may write 
\[	\widetilde\Lambda_{R_0,[a,b],0}/u = (R_0\otimes\A_{\mathrm{inf}})[X,Y]/(u,[\varpi]^{1/a},[\varpi]^{1/b}X,XY-[\varpi]^{1/a-1/b})	\]
and given $\overline x\in\widetilde\Lambda_{R_0,[a,b],0}/u$, we may choose a lift to $(R_0\otimes\A_{\mathrm{inf}})\left[X,Y\right]/(u)$ of the form $\sum_{i\geq 1}\alpha_iX^i+\sum_{j\geq 0}\beta_jY^j$, then project $\sum_{j\geq 0}\beta_jY^j$ to $\widetilde\Lambda_{R_0,[a,\infty],0}/u$. 

The resulting map $s_1$ is not necessarily a section, but it has the property that for $x\in \widetilde\Lambda_{R_0,[a,b],0}$, $x-\im(s_1(x))\in\im(\widetilde\Lambda_{R_0,[0,b],0})$.  Since the restriction map $\widetilde\Lambda_{R,[0,b]}\rightarrow\widetilde\Lambda_{R,[a,b]}$ is injective, the map $x\mapsto x-\im(s_1(x))$ defines another map $s_2:\widetilde\Lambda_{R,[a,b]}\rightarrow \widetilde\Lambda_{R,[0,b]}$.  Then 
\[	s:=-s_1\oplus s_2:\widetilde\Lambda_{R_0,[a,b],0}\rightarrow \widetilde\Lambda_{R_0,[a,\infty],0}\oplus \widetilde\Lambda_{R,[0,b]}	\]
is the desired section.
\end{proof}

\begin{lemma}
There is an exact sequence
\[	0\rightarrow \Lambda_{R_0,[0,\infty]}^{H_K}\left[\frac{1}{u}\right]\rightarrow \widetilde\Lambda_{R_0,[0,b],0}^{H_K}\left[\frac{1}{u}\right]\oplus \widetilde\Lambda_{R,[a,\infty]}^{H_K}\rightarrow \widetilde\Lambda_{R,[a,b]}^{H_K}\rightarrow 0	\]
\end{lemma}
\begin{proof}
This follows as in Berger's Lemme 2.27.  If $a=0$ or $b=\infty$, the result is trivial.  If not, we have an exact sequence
\begin{equation*}
	\begin{split}
0\rightarrow \Lambda_{R_0,[0,\infty]}^{H_K}\left[\frac 1 u\right]
\rightarrow &\widetilde\Lambda_{R_0,[0,b],0}^{H_K}\left[\frac 1 u\right]\oplus \widetilde\Lambda_{R,[a,\infty]}^{H_K}	\\
&\rightarrow  \widetilde\Lambda_{R,[a,b]}^{H_K}\rightarrow \H^1(H_K,\Lambda_{R_0,[0,\infty]}\left[\frac 1 u\right])
\end{split}
\end{equation*}
and we need to show that the map $\delta:\widetilde\Lambda_{R,[a,b]}^{H_K}
\rightarrow H^1(H_K,\Lambda_{R_0,[0,\infty]}\left[\frac 1 u\right]$ is the zero map.

If $x\in \widetilde\Lambda_{R_0,[a,b],0}^{H_K}$, then $x\cdot\left(\frac{u}{[\overline\pi]^{s(b)}}\right)\in\widetilde\Lambda_{R_0,[a,b],0}^{H_K}$ as well, and after multiplying by a suitable power of $u$, we may assume that $\delta(x)$ and $\delta(x\cdot\left(\frac{u}{[\overline\pi]^{s(b)}}\right))$ are both elements of $\H^1(H_K,\Lambda_{R_0,[0,\infty]})$.  But then
\[	u\cdot\delta(x)=\delta(ux)=\delta(x\cdot\left(\frac{u}{[\overline\pi]^{s(b)}}\right)\cdot[\overline\pi]^{s(b)})=[\overline\pi]^{s(b)}\delta(x\cdot\left(\frac{u}{[\overline\pi]^{s(b)}}\right))	\]
Then the result follows from the next lemma.
\end{proof}

\begin{lemma}
If $\alpha\in\mathscr{O}_{\C_p}^\flat$ is in the maximal ideal, then the homomorphism $\H^1(H_K,\Lambda_{R_0,[0,\infty]})\xrightarrow{\times[\alpha]}\H^1(H_K,\Lambda_{R_0,[0,\infty]})$ induced by multiplication by $[\alpha]$ is the zero map.
\end{lemma}
\begin{proof}
If $c_\tau\in \H^1(H_K,\Lambda_{R_0,[0,\infty]})$ is in the image of this map, we may consider its image in $\H^1(H_K,(R_0/u)\htimes\A_{\mathrm{inf}})$.  Then a standard Tate--Sen argument shows that this image is the trivial cocycle.  Working modulo successive powers of $u$, it follows that $c_\tau$ is itself trivial.
\end{proof}

Thus, to study the $H_K$-invariants of $\widetilde\Lambda_{R,[a,b]}\left[\frac{1}{u}\right]$, it suffices to study the $H_K$-invariants of $\widetilde\Lambda_{R_0,[0,b],0}\left[\frac{1}{u}\right]$ and $\widetilde\Lambda_{R_0,[a,\infty]}\left[\frac{1}{u}\right]$.

We first study the ring $\widetilde\Lambda_{R_0,[0,\infty]}^{H_K}=(R_0\htimes\A_{\mathrm{inf}})^{H_K}$.

\begin{lemma}\label{lemma: ainf-invariants}
\begin{enumerate}
\item	If $N$ is a finite module over a noetherian discrete $\Z/p^n$-algebra $R_0$, then 
\[	\left(N\htimes\A_{\mathrm{inf}}\right)^{H_K}=N\htimes\A_{\mathrm{inf}}^{H_K}	\]
\item	If $R_0$ is a Huber ring with principal ideal of definition, then 
\[	\left(R_0\htimes\A_{\mathrm{inf}}\right)^{H_K}=R_0\htimes\A_{\mathrm{inf}}^{H_K}	\]
\item	If $N$ is a finite module over a noetherian discrete $\Z/p^n$-algebra and $[\mathfrak{m}_{\C_p^\flat}]\subset\A_{\mathrm{inf}}$ denotes the ideal generated by $\{[\alpha]\}_{\alpha\in\mathfrak{m}_{\C_p^\flat}}$, then 
\[	\left(N\otimes [\mathfrak{m}_{\C_p^\flat}]/[\overline\pi]^{s(a)}[\mathfrak{m}_{\C_p^\flat}]\right)^{H_K}=N\otimes [\mathfrak{m}_{\widehat{K}_\infty^\flat}]/[\overline\pi]^{s(a)}[\mathfrak{m}_{\widehat{K}_\infty^\flat}]	\]
\end{enumerate}
\end{lemma}
\begin{proof}
In order to compute the $H_K$-invariants of $N\htimes\A_{\mathrm{inf}}$, we proceed by induction on $n$.  If $n=1$, then $N$ is a $\F_p$-vector space and we may choose an $\F_p$-basis $\{e_j\}_{j\in J}$.  Then $N\htimes\A_{\mathrm{inf}}=\{\sum_j b_je_j\mid b_j\in \mathscr{O}_{\C_p}^\flat, b_j\rightarrow 0\}$, where the coefficients $b_j$ tend to $0$ with respect to the cofinite filter.  It follows that 
\[	\left(N\htimes\A_{\mathrm{inf}}\right)^{H_K}=\{\sum_j b_je_j\mid b_j\in \widehat{\mathscr{O}}_{K_\infty}^\flat, b_j\rightarrow 0\}=N\htimes\A_{\mathrm{inf}}^{H_K}	\]
Now assume the result holds for $n-1$.  We have an exact sequence
\[	0\rightarrow p^{n-1}(N\htimes\A_{\mathrm{inf}})\rightarrow N\htimes\A_{\mathrm{inf}}\rightarrow (N\htimes\A_{\mathrm{inf}})/(p^{n-1})\rightarrow 0	\]
Since $p^{n-1}N$ is a finite module over the discrete $\F_p$-algebra $R_0/p$, the inductive hypothesis implies that we have an exact sequence
\[	0\rightarrow p^{n-1}N\htimes\A_{\mathrm{inf}}^{H_K}\rightarrow \left(N\htimes\A_{\mathrm{inf}}\right)^{H_K}\rightarrow (N\htimes\A_{\mathrm{inf}}^{H_K})/(p^{n-1})	\]
There is furthermore a natural map $N\htimes\A_{\mathrm{inf}}^{H_K}\rightarrow \left(N\htimes\A_{\mathrm{inf}}\right)^{H_K}$ and a commutative diagram
\begin{equation*}
\resizebox{\displaywidth}{!}{$
\begin{tikzcd}[ampersand replacement=\&]
0\arrow[r] \& \left(p^{n-1}N\htimes\A_{\mathrm{inf}}\right)^{H_K}\arrow[r]\arrow[equal,d] \& \left(N\htimes\A_{\mathrm{inf}}\right)^{H_K} \arrow[r] \& \left((N/p^{n-1})\htimes\A_{\mathrm{inf}}\right)^{H_K} \arrow[equal,d]	\\
0\arrow[r] \& p^{n-1}N\htimes\A_{\mathrm{inf}}^{H_K} \arrow[r] \& N\htimes\A_{\mathrm{inf}}^{H_K} \arrow[r]\arrow[u] \& (N/p^{n-1})\htimes\A_{\mathrm{inf}}^{H_K} \arrow[r] \& 0
\end{tikzcd}
$}
\end{equation*}
A diagram chase shows that the map $N\htimes\A_{\mathrm{inf}}^{H_K}\rightarrow \left(N\htimes\A_{\mathrm{inf}}\right)^{H_K}$ is an isomorphism.

For the second part, let $u\in R_0$ generate the ideal of definition.  Since $p$ is topologically nilpotent in $R_0$, each quotient ring $R_0/u^k$ is $p$-power torsion.  We observe that $R_0\htimes\A_{\mathrm{inf}}=\varprojlim_k\varprojlim_{k'}(R_0/u^k)\otimes(\A_{\mathrm{inf}}/[\overline\pi]^{k'})$, and so
\begin{equation*}
\begin{split}
\left(R_0\htimes\A_{\mathrm{inf}}\right)^{H_K}&=\varprojlim_k\left(\varprojlim_{k'}(R_0/u^k)\otimes(\A_{\mathrm{inf}}/[\overline\pi]^{k'})\right)^{H_K} 	\\
&= \varprojlim_k\left((R_0/u^k)\htimes\A_{\mathrm{inf}}\right)^{H_K} \\
&= \varprojlim_k(R_0/u^k)\htimes\A_{\mathrm{inf}}^{H_K} = R_0\htimes\A_{\mathrm{inf}}^{H_K}
\end{split}
\end{equation*}

The last part follows similarly, using the fact that $\left(\mathfrak{m}_{\C_p^\flat}/\overline\pi^{s(a)}\mathfrak{m}_{\C_p^\flat}\right)^{H_K}=\mathfrak{m}_{\widehat{K}_\infty^\flat}/\overline\pi^{s(a)}\mathfrak{m}_{\widehat{K}_\infty^\flat}$.
\end{proof}

To study the rings $\widetilde\Lambda_{R_0,[a,b]}^{H_K}$ when $[a,b]\neq [0,\infty]$, we require a number of preparatory results.  We will proceed by making a careful study of $(R_0\htimes\A_{\mathrm{inf}})/(u-[\overline\pi]^{s(a)})$ and bootstrapping from characteristic $p$ to characteristic $0$.  However, our techniques specifically exclude the classical case, where $p$ is a pseudo-uniformizer of $R$; we use the ideals $I_j\subset R_0$ defined at the beginning of \textsection~\ref{section: rings with coefficients}, which only differ from $R_0$ itself when $p\notin R^\times$.

We first record two purely algebraic lemmas.
\begin{lemma}\label{lemma: m-regular top free}
Let $R$ be a ring and $M$ an $R$-module, and suppose that $R$ and $M$ are $I$-adically separated and complete, for some ideal $I\subset R$ generated by an $M$-regular sequence $r_1,\ldots,r_n$.  If $M/I$ is free over $R/I$, then $M$ is topologically free over $R$, i.e., there is a subset $\{e_i\}_{i\in I}\subset M$ such that the natural map
\[	\{\sum_{i\in I}a_i\mathbf{e}_i\mid a_i\in R, a_i\rightarrow 0\}\rightarrow M	\]
given by $\mathbf{e}_i\mapsto e_i$ is an isomorphism.  
\end{lemma}
\begin{proof}
We proceed by induction on the number of generators of $I$.  If $n=0$, there is nothing to prove.

Assume the result for $I$ generated by $n-1$ elements.  Then $M/r_1$ is topologically free over $R/r_1$, and we may lift a topological basis of $M/r_1$ to a subset $\{e_i\}_{i\in I}\subset M$.  Then we have a homomorphism of $R$-modules
\[      \{\sum_{i\in I}a_i\mathbf{e}_i\mid a_i\in R, a_i\rightarrow 0\}\rightarrow M    \]
which is an isomorphism modulo $r_1$.  By~\cite[Tag 07RC(12)]{stacks-project}, it is surjective modulo all powers of $r_1$.  Moreover, if $K$ denotes the kernel, the assumption that $r_1$ is $M$-regular implies that the kernel $K_m$ of the reduction modulo $r_1^m$ is simply $K/r_1^m$.  Thus, $R\lim^1 K_m=0$ and our map is an isomorphism.
\end{proof}

\begin{lemma}\label{lemma: fancy crt}
For any ring $R$ and any ideals $I_1,I_2\subset R$, there is an exact sequence
\[	0\rightarrow R/(I_1\cap I_2)\rightarrow R/I_1\oplus R/I_2\rightarrow R/(I_1+I_2)\rightarrow 0	\]
where the map $R/I_1\oplus R/I_2\rightarrow R/I_1+I_2$ is given by $(f_1,f_2)\mapsto f_1-f_2$.
\end{lemma}
\begin{proof}
The map $R/I_1\oplus R/I_2\rightarrow R/(I_1+I_2)$ is clearly surjective, and the map $R/I_1\cap I_2\rightarrow R/I_1\oplus R/I_2$ is clearly injective.  It remains to check exactness in the middle.  So suppose we have a pair $(f_1,f_2)\in R/I_1\oplus R/I_2$ such that $f_1-f_2=0$ in $R/(I_1+I_2)$.  Since the map $R/(I_1\cap I_2)\rightarrow R/I_2$ is surjective, we may assume that $f_2=0$, and therefore that $f_1\in (I_1+I_2)/I_1$.  But $(I_1+I_2)/I_1\cong I_2/(I_1\cap I_2)$ as $R$-modules; given a representation $f_1=g_1+g_2$ with $g_i\in I_i$, this isomorphism sends $f_1$ to the image of $g_2$ modulo $I_1\cap I_2$.  Then the natural map $R/(I_1\cap I_2)\rightarrow R/I_1\oplus R/I_2$ carries $g_2$ to $(f_1,0)$, as desired.
\end{proof}

We now return to the setting of interest.
\begin{corollary}\label{cor: ainf flat}
	For any $\alpha\in\mathfrak{m}_{K_\infty}^\flat$, the ring homomorphisms $\Z_p[\![[\alpha]]\!]\rightarrow\A_{\mathrm{inf}}$ and $\Z_p[\![[\alpha]]\!]\rightarrow \A_{\mathrm{inf}}^{H_K}$ are flat.
\end{corollary}
\begin{proof}
This follows from Lemma~\ref{lemma: m-regular top free} and \cite[Tag 06LE]{stacks-project}.
\end{proof}

\begin{corollary}
For any $\alpha\in\mathfrak{m}_{K_\infty^\flat}$, $\A_{\mathrm{inf}}^{H_K}$ and $\A_{\mathrm{inf}}$ are topologically free over $\Z_p[\![[\alpha]]\!]$.
\end{corollary}

\begin{corollary}\label{cor: top free}
For any $a\in\Q_{>0}$, the rings $\Lambda_{R_0,[0,\infty]}^{H_K}/(u-[\overline\pi]^{s(a)})$ and $\Lambda_{R_0,[0,\infty]}/(u-[\overline\pi]^{s(a)})$ are topologically free $(R_0\htimes\Z_p[\![[\overline\pi]^{s(a)}]\!])/(u-[\overline\pi]^{s(a)})$-modules.
\end{corollary}

\begin{lemma}\label{lemma: u-pi r0-flat}
Let $R$ be a pseudoaffinoid algebra over $\Z_p$,  with $R_0\subset R$ a noetherian ring of definition formally of finite type over $\Z_p$ $u\in R_0$ a pseudo-uniformizer $u$.  Suppose that $N$ is a finite flat $R_0$-module. Then $(N\htimes\A_{\mathrm{inf}})/(u-[\overline\pi]^{s(a)})$ and $(N\htimes\A_{\mathrm{inf}}^{H_K})/(u-[\overline\pi]^{s(a)})$ are flat over $R_0$.  
\end{lemma}
In particular, $\Lambda_{R_0,[0,\infty]}/(u-[\overline\pi]^{s(a)})$ and $\Lambda_{R_0,[0,\infty]}^{H_K}/(u-[\overline\pi]^{s(a)})$ are flat over $R_0$.
\begin{proof}
We first observe that
\[	(N\htimes\A_{\mathrm{inf}})/(u-[\overline\pi]^{s(a)})\cong \left((N\htimes\Z_p[\![[\overline\pi]^{s(a)}]\!])/(u-[\overline\pi]^{s(a)})\right)\htimes_{\Z_p[\![[\overline\pi]^{s(a)}]\!]}\A_{\mathrm{inf}}	\]
and
\[	(N\htimes\A_{\mathrm{inf}}^{H_K})/(u-[\overline\pi]^{s(a)})\cong \left((N\htimes\Z_p[\![[\overline\pi]^{s(a)}]\!])/(u-[\overline\pi]^{s(a)})\right)\htimes_{\Z_p[\![[\overline\pi]^{s(a)}]\!]}\A_{\mathrm{inf}}^{H_K}	\]
By Corollary~\ref{cor: top free}, it suffices to show that $(N\htimes\Z_p[\![[\overline\pi]^{s(a)}]\!])/(u-[\overline\pi]^{s(a)})$ is flat over $R_0$.

Since $R_0$, $\Z_p[\![[\overline\pi]^{s(a)}]\!]$, and $R_0\htimes\Z_p[\![[\overline\pi]^{s(a)}]\!]$ are all noetherian, to prove this we may apply~\cite[Theorem 22.6]{matsumura} with $A=R_0$ (resp. $\Z_p[\![[\overline\pi]^{s(a)}]\!]$), $B=R_0\htimes\Z_p[\![[\overline\pi]^{s(a)}]\!]$, $M=N\htimes\Z_p[\![[\overline\pi]^{s(a)}]\!]$, and $b=u-[\overline\pi]^{s(a)}$.  Since $B$ is flat over $A$ and $\mathfrak{m}\cap A$ is a maximal ideal $\mathfrak{n}\subset R_0$ (resp. $\mathfrak{m}\cap A=([\overline\pi]^{s(a)})$) for every maximal ideal $\mathfrak{m}\subset B$, it is enough to check that the image of $u-[\overline\pi]^{s(a)}$ is not a zero-divisor in $M/\mathfrak{n}$ (resp. $M/[\overline\pi]^{s(a)}$).  But $M/\mathfrak{m}\cong N/\mathfrak{n}\htimes\Z_p[\![[\overline\pi]^{s(a)}]\!]$ (resp. $M/[\overline\pi]^{s(a)}\cong N$), the image of $u-[\overline\pi]^{s(a)}$ is the image of $[\overline\pi]^{s(a)}$ since every maximal ideal of $R_0$ contains $u$ (resp. the image of $u-[\overline\pi]^{s(a)}$ is $u$), and $(N/\mathfrak{n})\htimes\Z_p[\![[\overline\pi]^{s(a)}]\!]$ is $[\overline\pi]^{s(a)}$-torsion-free (resp. $N$ is $u$-torsion-free).

Thus, we conclude that $(N\htimes\Z_p[\![[\overline\pi]^{s(a)}]\!])/(u-[\overline\pi]^{s(a)})$ is flat over $R_0$ and $\Z_p[\![[\overline\pi]^{s(a)}]\!]$, as desired.
\end{proof}

When $R$ has positive characteristic, we may relax the hypothesis on $N$ and apply the same argument:
\begin{corollary}\label{lemma:u-pi-torsion-free}
If $R$ is topologically finite type over $\F_p(\!(u)\!)$, $R_0\subset R$ is a ring of definition strictly topologically of finite type over $\F_p[\![u]\!]$, and $N$ is a finite $u$-torsion-free $R_0$-module, then $(N\htimes\A_{\mathrm{inf}})/(u-\overline\pi^{s(a)})$ has no $u$- or $\overline\pi$-torsion.  In particular, $\Lambda_{R_0,[0,\infty]}/(u-\overline\pi^{s(a)})$ has no $u$- or $\overline\pi$-torsion.
\end{corollary}
\begin{proof}
We again prove that $(N\htimes\F_p[\![\overline\pi^{s(a)}]\!])/(u-\overline\pi^{s(a)})$ is flat over $\F_p[\![\overline\pi^{s(a)}]\!]$ by applying ~\cite[Theorem 22.6]{matsumura} with $A=\F_p[\![\overline\pi^{s(a)}]\!]$, $B=R_0\htimes\F_p[\![\overline\pi^{s(a)}]\!]$, $M=N\htimes\F_p[\![\overline\pi^{s(a)}]\!]$, and $b=u-\overline\pi^{s(a)}$.  This implies that the module $(N\htimes\F_p[\![\overline\pi^{s(a)}]\!])/(u-\overline\pi^{s(a)})$ has no $\overline\pi$-torsion, and hence no $u$-torsion.
\end{proof}

\begin{lemma}
If $R$ is topologically of finite type over $\F_p(\!(u)\!)$, $R_0\subset R$ is a ring of definition strictly topologically of finite type over $\F_p[\![u]\!]$, and $N$ is a finite $u$-torsion-free $R_0$-module, then the natural map 
\[	N\htimes\A_{\mathrm{inf}}^{H_K}\rightarrow \left((N\htimes\A_{\mathrm{inf}})/(u-[\overline\pi]^{s(a)})\right)^{H_K}	\]
is surjective.
\end{lemma}
\begin{proof}
We have an isomorphism
\[	(N\htimes\A_{\mathrm{inf}})/(u-[\overline\pi]^{s(a)})\cong \left((N\htimes\F_p[\![\overline\pi^{s(a)}]\!])/(u-\overline\pi^{s(a)})\right)\htimes_{\F_p[\![\overline\pi^{s(a)}]\!]}\mathscr{O}_{\C_p}^\flat	\]
Since $N$ is $u$-torsion-free, Lemma~\ref{lemma:u-pi-torsion-free} implies $(N\htimes\F_p[\![\overline\pi^{s(a)}]\!])/(u-\overline\pi^{s(a)})$ is $\overline\pi$-torsion-free.  Then Lemma~\ref{lemma: m-regular top free} implies it is topologically free over $\F_p[\![\overline\pi^{s(a)}]\!]$, and the result follows.

It follows that 
\[	\left((N\htimes\A_{\mathrm{inf}})/(u-[\overline\pi]^{s(a)})\right)^{H_K}=(N\htimes\A_{\mathrm{inf}}^{H_K})/(u-[\overline\pi]^{s(a)})	\]
as desired.
\end{proof}

Now we can begin to bootstrap to the case where $R$ is $\Z_p$-flat.  

Recall that when $p\notin R^\times$, we defined a sequence of ideals $I_j:=p^jR\cap R_0$.  Each ideal $I_j$ is finitely generated (since $R_0$ is noetherian), so there is a sequence of integers $k_j\geq 1$ such that $u^{k_j}I_j\subset p^jR_0$.
\begin{lemma}
With notation as above, $k_j\leq jk_1$.
\end{lemma}
\begin{proof}
We proceed by induction on $j$. The case $j=1$ is trivial, so assume the result holds for $j-1$.  If $x\in I_j$, then $u^{k_j}x=u^{k_j-k_1}u^{k_1}x\in p^jR_0$.  But since $I_j\subset I_1$, it follows that $u^{k_1}x=px'$, and since $p$ is not a zero-divisor, $u^{k_j-k_1}x'\in p^{j-1}R_0$.  By the inductive hypothesis, $k_j-k_1\leq (j-1)k_1$, and the result follows.
\end{proof}

\begin{lemma}\label{lemma: int ij noeth}
With notation as above, 
\begin{enumerate}
\item	$\cap_jI_j=\{0\}$, 
\item	$\cap_jI_j\left((R_0\htimes\Z_p[\![[\overline\pi]^{s(a)}]\!])/(u-[\overline\pi]^{s(a)})\right)=\{0\}$.
\end{enumerate}
\end{lemma}
\begin{proof}
If $x\in \cap_jI_j$, then for all $j\geq 1$, $u^{jk_1}x=p^jx_j$ for some $x_j\in R_0$.  This implies that $x\in \left(\frac{p}{u^{k_1}}\right)^jR$ for all $j$.  Since $p\notin R^\times$, $\left(\frac{p}{u^{k_1}}\right)R$ is a proper ideal, and Krull's intersection theorem implies that $\cap_j \left(\frac{p}{u^{k_1}}\right)^jR=\{0\}$.  Since $\left((R_0\htimes\Z_p[\![[\overline\pi]^{s(a)}]\!])/(u-[\overline\pi]^{s(a)})\right)$ is noetherian, the same argument applies to $\cap_jI_j\left((R_0\htimes\Z_p[\![[\overline\pi]^{s(a)}]\!])/(u-[\overline\pi]^{s(a)})\right)$.
\end{proof}

We first treat the case where $R$ is a $\Z/p^n$-algebra (and $I_n=(0)$).
\begin{corollary}
If $R$ is topologically of finite type over $(\Z/p^n)(\!(u)\!)$ and $R_0\subset R$ is a ring of definition strictly topologically of finite type over $(\Z/p^n)[\![u]\!]$, then the natural map 
\[	\Lambda_{R_0,[0,\infty]}^{H_K}\rightarrow \left(\Lambda_{R_0,[0,\infty]}/(u-[\overline\pi]^{s(a)})\right)^{H_K}	\]
is surjective.
\end{corollary}
\begin{proof}
We have seen that 
\[	\left(((R_0/I_1)\htimes\A_{\mathrm{inf}})/(u-[\overline\pi]^{s(a)})\right)^{H_K}\cong  ((R_0/I)\htimes\A_{\mathrm{inf}}^{H_K})/(u-[\overline\pi]^{s(a)})	\]
and we will proceed by induction on $j$.  We have a commutative diagram
\begin{equation*}
\resizebox{\displaywidth}{!}{
$
\begin{tikzcd}[ampersand replacement=\&]
0\arrow[r] \& (I_j/I_{j+1})\htimes\A_{\mathrm{inf}}^{H_K}\arrow[r]\arrow[two heads,d] \& (R_0/I_{j+1})\htimes\A_{\mathrm{inf}}^{H_K}\arrow[r]\arrow[d] \& (R_0/I_j)\htimes\A_{\mathrm{inf}}^{H_K}\arrow[r]\arrow[two heads,d] \& 0	\\
0\arrow[r] \& \left((I_j/I_{j+1})\htimes\A_{\mathrm{inf}}/(u-[\overline\pi]^{s(a)}\right)^{H_K}\arrow[r] \& \left((R_0/I_{j+1})\htimes\A_{\mathrm{inf}}/(u-[\overline\pi]^{s(a)})\right)^{H_K}\arrow[r] \& \left((R_0/I_{j})\htimes\A_{\mathrm{inf}}/(u-[\overline\pi]^{s(a)})\right)^{H_K} \& \&
\end{tikzcd}
$
}
\end{equation*}
Then the snake lemma implies that the middle arrow is surjective, as well.
\end{proof}

Now we return to the case of a general pseudoaffinoid algebra $R$ (where $p$ is not a unit).
\begin{lemma}\label{lemma: u-pi torsion}
The natural map 
\[	\Lambda_{R_0,[0,\infty]}^{H_K}/(u-[\overline\pi]^{s(a)})\rightarrow \Lambda_{R_0,[0,\infty]}/(u-[\overline\pi]^{s(a)})	\]
is injective.
\end{lemma}
\begin{proof}
We need to check that 
\[	(u-[\overline\pi]^{s(a)})\Lambda_{R_0,[0,\infty]}\cap \Lambda_{R_0,[0,\infty]}^{H_K}=(u-[\overline\pi]^{s(a)})\Lambda_{R_0,[0,\infty]}^{H_K}	\]
and it suffices to check that $\Lambda_{R_0,[0,\infty]}$ has no $u-[\overline\pi]^{s(a)}$-torsion.  But $(R_0/u)\htimes\A_{\mathrm{inf}}$ has no $[\overline\pi]^{s(a)}$-torsion, so if $x\in \Lambda_{R_0,[0,\infty]}$ is annihilated by $u-[\overline\pi]^{s(a)}$, it is a multiple of $u$.  In addition, $\Lambda_{R_0,[0,\infty]}$ has no $u$-torsion, so if $x=ux'$ is killed by $u-[\overline\pi]^{s(a)}$, so is $x'$.  Replacing $x$ with $x'$ and repeating the argument, we see that $x\in u^n\Lambda_{R_0,[0,\infty]}$ for all $n$, so $x=0$.
\end{proof}

\begin{lemma}\label{lemma: invts of u-mults}
If $x\in \Lambda_{R_0,[0,\infty]}$ and the image of $x$ in $\Lambda_{R_0,[0,\infty]}/(u-[\overline\pi]^{s(a)})$ is a multiple of $u$, then the image of $x$ in $\Lambda_{R_0,[0,\infty]}^{H_K}/(u-[\overline\pi]^{s(a)})$ is a multiple of $u$.
\end{lemma}
\begin{proof}
We may write $x=ux'+(u-[\overline\pi]^{s(a)})y$ for $x',y\in \Lambda_{R_0,[0,\infty]}$.  Reducing modulo $u$, we have $x\equiv -[\overline\pi]^{s(a)}y$; since $(R_0/u)\htimes\A_{\mathrm{inf}}$ has no $[\overline\pi]^{s(a)}$-torsion, we see that $y\equiv y'\pmod{u}$, where $y'\in \Lambda_{R_0,[0,\infty]}^{H_K}$.  In other words,
\[	x=ux'+(u-[\overline\pi]^{s(a)})(y'+uz)	\]
for some $z\in \Lambda_{R_0,[0,\infty]}$.  But then 
\[	x-(u-[\overline\pi]^{s(a)})y'\in \Lambda_{R_0,[0,\infty]}^{H_K}\cap u\Lambda_{R_0,[0,\infty]}	\]
Since $\Lambda_{R_0,[0,\infty]}$ has no $u$-torsion, 
\[	\Lambda_{R_0,[0,\infty]}^{H_K}\cap u\Lambda_{R_0,[0,\infty]}=u\Lambda_{R_0,[0,\infty]}^{H_K}	\]
and we are done.
\end{proof}

\begin{lemma}\label{lemma: inj mod u Ij}
The natural map 
\[	(R_0/(u,I_j))\otimes(\A_{\mathrm{inf}}^{H_K}/[\overline\pi]^{s(a)})\rightarrow (R_0/(u,I_j))\otimes(\A_{\mathrm{inf}}/[\overline\pi]^{s(a)})	\]
is injective for all $j\geq 1$.
\end{lemma}
\begin{proof}
We first show that the natural map $\widehat{\mathcal{O}}_{K_\infty}^\flat/[\overline\pi]^{s(a)}\rightarrow \mathcal{O}_{\C_p}^\flat/[\overline\pi]^{s(a)}$ is injective.  But the cokernel of the injection $\widehat{\mathcal{O}}_{K_\infty}^\flat\rightarrow \mathcal{O}_{\C_p}^\flat$ is an $\widehat{\mathcal{O}}_{K_\infty}^\flat$-module with no $[\overline\pi]^{s(a)}$-torsion, so this follows.

We proceed by induction on $j$.  If $j=1$, then $R_0/(u,I_1)$ is a discrete $\F_p$-vector space, and therefore the map $\widehat{\mathcal{O}}_{K_\infty}^\flat/[\overline\pi]^{s(a)}\rightarrow \mathcal{O}_{\C_p}^\flat/[\overline\pi]^{s(a)}$ remains injective after tensoring with $R_0/(u,I_1)$.  So assume the result for $j-1$.  We have a commutative diagram
\begin{equation*}
\resizebox{\displaywidth}{!}{
$
\begin{tikzcd}[ampersand replacement=\&]
0\ar[r] \& (I_{j-1}/(u,I_j))\otimes(\A_{\mathrm{inf}}^{H_K}/[\overline\pi]^{s(a)}) \ar[r]\ar[d] \& (R_0/(u,I_j))\otimes(\A_{\mathrm{inf}}^{H_K}/[\overline\pi]^{s(a)}) \ar[r]\ar[d] \& (R_0/(u,I_{j-1}))\otimes(\A_{\mathrm{inf}}^{H_K}/[\overline\pi]^{s(a)}) \ar[r]\ar[d] \& 0	\\
0 \ar[r] \& (I_{j-1}/(u,I_j))\otimes(\A_{\mathrm{inf}}/[\overline\pi]^{s(a)}) \ar[r] \& (R_0/(u,I_j))\otimes(\A_{\mathrm{inf}}/[\overline\pi]^{s(a)}) \ar[r] \& (R_0/(u,I_{j-1}))\otimes(\A_{\mathrm{inf}}/[\overline\pi]^{s(a)}) \ar[r] \& 0
\end{tikzcd}
$
}
\end{equation*}
Since $I_{j-1}/(u,I_j)$ is annihilated by $p$, the left vertical arrow is injective.  The right vertical arrow is injective by the inductive hypothesis.  A diagram chase then implies that the middle vertical arrow is injective, as desired.
\end{proof}

Applying this to $\Lambda_{R_0,[0,\infty]}^{H_K}/(u-[\overline\pi]^{s(a)})$ and $\Lambda_{R_0,[0,\infty]}/(u-[\overline\pi]^{s(a)})$ yields the following:
\begin{lemma}\label{lemma: a-inf crt}
For each $j\geq 1$, there are exact sequences

\begin{equation*}
\resizebox{\displaywidth}{!}{$
		0\rightarrow\Lambda_{R_0,[0,\infty]}^{H_K}/(u-[\overline\pi]^{s(a)},uI_{j})\rightarrow \substack{\Lambda_{R_0,[0,\infty]}^{H_K}/(u,[\overline\pi]^{s(a)})\\ \oplus \\ \Lambda_{R_0,[0,\infty]}^{H_K}/(u-[\overline\pi]^{s(a)},I_{j})} \rightarrow \Lambda_{R_0,[0,\infty]}^{H_K}/(u,[\overline\pi]^{s(a)},I_{j})\rightarrow 0
$}
\end{equation*}
and
\begin{equation*}
\resizebox{\displaywidth}{!}{$
	0\rightarrow\Lambda_{R_0,[0,\infty]}/(u-[\overline\pi]^{s(a)},uI_{j})\rightarrow \substack{\Lambda_{R_0,[0,\infty]}/(u,[\overline\pi]^{s(a)}) \\ \oplus \\ \Lambda_{R_0,[0,\infty]}/(u-[\overline\pi]^{s(a)},I_{j})}\rightarrow \Lambda_{R_0,[0,\infty]}/(u,[\overline\pi]^{s(a)},I_{j})\rightarrow 0
$}
\end{equation*}

\end{lemma}
\begin{proof}
By Lemma~\ref{lemma: fancy crt}, we have exact sequences
\[	0\rightarrow R_0/((u)\cap I_{j})\rightarrow R_0/(u)\oplus R_0/I_{j}\rightarrow R_0/((u)+I_{j})\rightarrow 0	\]
Certainly $uI_j\subset (u)\cap I_{j}$. On the other hand, if $uf\in I_{j}$ for some $f\in R_0$, then $u^kuf\in p^{j}R_0$ for some $k\gg 0$, and so $f\in I_{j}$.  
Thus, the inclusion $uI_{j}\subset (u)\cap I_{j}$ is actually an equality, and we have exact sequences
\[	0\rightarrow R_0/uI_{j}\rightarrow R_0/(u)\oplus R_0/I_{j}\rightarrow R_0/((u)+I_{j})\rightarrow 0	\]
By Lemma~\ref{lemma: u-pi r0-flat}, $\Lambda_{R_0,[0,\infty]}^{H_K}/(u-[\overline\pi]^{s(a)})$ and $\Lambda_{R_0,[0,\infty]}/(u-[\overline\pi]^{s(a)})$ are both flat over $R_0$. Thus, we may extend scalars from $R_0$ to $\Lambda_{R_0,[0,\infty]}^{H_K}/(u-[\overline\pi]^{s(a)})$ to obtain the desired result.
\end{proof}

\begin{proposition}\label{hk-invts-pseudorigid}
If $x\in\left(\Lambda_{R_0,[0,\infty]}/(u-[\overline\pi]^{s(a)})\right)^{H_K}$ and $\alpha\in\mathfrak{m}_{\widehat{K}_\infty^\flat}$, then $[\alpha]x$ is in the image of $\Lambda_{R_0,[0,\infty]}^{H_K}$.
\end{proposition}
\begin{proof}

We first consider the image of $x$ modulo $I_1$.  Since it is fixed by $H_K$, it defines an element of $((R_0/I_1)\htimes \A_{\mathrm{inf}}^{H_K})/(u-[\overline\pi]^{s(a)})$, and therefore so does $[\alpha]x$.

Considering instead the image of $x$ modulo $u$, we obtain an element of $\left((R_0/u)\otimes(\A_{\mathrm{inf}}/[\overline\pi]^{s(a)})\right)^{H_K}$.  There is a sequence of $H_K$-equivariant maps
\begin{equation*}
	\begin{split}
\A_{\mathrm{inf}}/[\overline\pi]^{s(a)}\xrightarrow{\times [\alpha]}[\alpha]/[\overline\pi^{s(a)}\alpha]\rightarrow [\alpha]/&[\overline\pi^{s(a)}\mathfrak{m}_{\C_p^\flat}]\rightarrow [\mathfrak{m}_{\C_p^\flat}]/[\overline\pi^{s(a)}\mathfrak{m}_{\C_p^\flat}]	\\
&\rightarrow [\mathfrak{m}_{\C_p^\flat}]/[\overline\pi]^{s(a)}\A_{\mathrm{inf}}\rightarrow \A_{\mathrm{inf}}/[\overline\pi]^{s(a)}
\end{split}
\end{equation*}
Since $\left([\mathfrak{m}_{\C_p^\flat}]/[\overline\pi^{s(a)}\mathfrak{m}_{\C_p^\flat}]\right)^{H_K}=[\mathfrak{m}_{\widehat{K}_\infty^\flat}]/[\overline\pi^{s(a)}\mathfrak{m}_{\widehat{K}_\infty^\flat}]$, $[\alpha]x$ defines an element of $(R_0/u)\otimes(\A_{\mathrm{inf}}^{H_K}/[\overline\pi]^{s(a)})$.

It follows from Lemma~\ref{lemma: a-inf crt} and Lemma~\ref{lemma: inj mod u Ij} that there is some 
\[	a_0\in \Lambda_{R_0,[0,\infty]}^{H_K}/(u-[\overline\pi]^{s(a)})	\]
such that $[\alpha]x-a_0\in uI_1\Lambda_{R_0,[0,\infty]}/(u-[\overline\pi]^{s(a)})$.  We may therefore write $[\alpha]x-a_0=ux_1$, where
\[	x_1\in I_1\left(\Lambda_{R_0,[0,\infty]}/(u-[\overline\pi]^{s(a)})\right)^{H_K}	\]

Then we have 
\[	ux_1=[\overline\pi]^{s(a)}x_1=[\overline\pi]^{(p-1)^2/p^2a}[\overline\pi]^{(p-1)/p^2a}x_1	\]
and we may apply the previous argument to $[\overline\pi]^{(p-1)/p^2a}x_1$.  We obtain some $a_1\in \Lambda_{R_0,[0,\infty]}^{H_K}/(u-[\overline\pi]^{s(a)})$ such that 
\[	[\overline\pi]^{(p-1)/p^2a}x_1-a_1\in uI_2\Lambda_{R_0,[0,\infty]}/(u-[\overline\pi]^{s(a)})	\]
and $[\overline\pi]^{(p-1)/p^2a}x_1-a_1$ remains fixed by $H_K$.

Continuing in this fashion, we obtain a sequence $\{a_j\}_{j\geq 0}$ of elements of $\Lambda_{R_0,[0,\infty]}^{H_K}/(u-[\overline\pi]^{s(a)})$, such that 
\[	[\alpha]x-\sum_{j=0}^{n-1}[\overline\pi]^{(p-1)^2j/p^2a}a_j\in uI_n\Lambda_{R_0,[0,\infty]}/(u-[\overline\pi]^{s(a)})	\]
Since the terms $[\overline\pi]^{(p-1)^2j/p^2a}a_j$ tend to $0$, the sum $\sum_{j\geq 0}[\overline\pi]^{(p-1)^2j/p^2a}a_j$ converges in $\Lambda_{R_0,[0,\infty]}^{H_K}/(u-[\overline\pi]^{s(a)})$, and 
\[	[\alpha]x-\sum_{j\geq 0}[\overline\pi]^{(p-1)^2j/p^2a}a_j\in\cap_j I_j\Lambda_{R_0,[0,\infty]}/(u-[\overline\pi]^{s(a)})	\]

To finish, we need to show that $\cap_j I_j\Lambda_{R_0,[0,\infty]}/(u-[\overline\pi]^{s(a)})=\{0\}$.  But this follows by combining Corollary~\ref{cor: top free} and Lemma~\ref{lemma: int ij noeth}.
\end{proof}

\begin{corollary}\label{cor: hk invts mod u-pi}
The natural map 
\[	\Lambda_{R_0,[0,\infty]}^{H_K}/(u-[\overline\pi]^{s(a)})\rightarrow \left(\Lambda_{R_0,[0,\infty]}/(u-[\overline\pi]^{s(a)})\right)^{H_K}	\]
is an isomorphism.
\end{corollary}
\begin{proof}
This follows by combining Proposition~\ref{hk-invts-pseudorigid} with Lemma~\ref{lemma: u-pi torsion} and Lemma~\ref{lemma: invts of u-mults}.
\end{proof}

We are finally in a position to compute $\widetilde\Lambda_{R_0,[a,b]}^{H_K}$.
\begin{corollary}
Suppose $[a,b]\subset (0,\infty)$.  If $R$ is a pseudoaffinoid algebra and $R_0\subset R$ is a noetherian ring of definition formally of finite type over $\Z_p$, then 
\[	\widetilde\Lambda_{R_0,[a,b]}^{H_K}=\Lambda_{R_0,[0,\infty]}^{H_K}\left\langle\frac{u}{[\overline\pi]^{s(b)}},\frac{[\overline\pi]^{s(a)}}{u}\right\rangle\left[\frac 1 u\right]	\]
\end{corollary}
\begin{proof}
We consider $\widetilde\Lambda_{R_0,[a,\infty]}^{H_K}$ and $\widetilde\Lambda_{R_0,[0,b]}^{H_K}$ separately.

There is a natural map 
\[	\Lambda_{R_0,[0,\infty]}\left\langle\frac{[\overline\pi]^{s(a)}}{u}\right\rangle\rightarrow \Lambda_{R_0,[0,\infty]}/(u-\overline\pi^{s(a)})	\]
with kernel $\left(1-\frac{[\overline\pi]^{s(a)}}{u}\right)$, extending the quotient map 
\[	\Lambda_{R_0,[0,\infty]}\twoheadrightarrow\Lambda_{R_0,[0,\infty]}/(u-[\overline\pi]^{s(a)})	\]
Given $x\in \Lambda_{R_0,[0,\infty]}\left\langle\frac{[\overline\pi]^{s(a)}}{u}\right\rangle$, there is a non-decreasing sequence $\{\alpha_i\}_{i\geq 1}$ of integers with $0\leq\alpha_i\leq i-1$ for all $i$ and $\lim \frac{\alpha_i}{i}=0$ such that 
\[	x\in u^{-\alpha_i}\Lambda_{R_0,[0,\infty]}+(\frac{u-[\overline\pi]^{s(a)}}{u})^i\Lambda_{R_0,[0,\infty]}\left\langle\frac{[\overline\pi]^{s(a)}}{u}\right\rangle	\]
for all $i$ (as in the proof of ~\cite[Lemme 2.29]{berger}).  If $x$ is fixed by $H_K$, Corollary~\ref{cor: hk invts mod u-pi} implies that there is some $a_0\in \Lambda_{R_0,[0,\infty]}^{H_K}$ such that $x\equiv a_0 \pmod{u-[\overline\pi]^{s(a)}}$.  
Moreover, $a_0\in u^{-\alpha_1}\Lambda_{R_0,[0,\infty]}^{H_K}$.

Suppose we have a sequence $a_0,\ldots,a_{n-1}$ of elements of $\Lambda_{R_0,[0,\infty]}^{H_K}$ such that 
\[	a_i\in u^{i-\alpha_{i+1}}\Lambda_{R_0,[0,\infty]}^{H_K}	\]
and 
\[	x-\left(\sum_{i=0}^{n-1}a_i\left(\frac{u-[\overline\pi]^{s(a)}}{u}\right)^i\right)\in \left(\frac{u-[\overline\pi]^{s(a)}}{u}\right)^n\Lambda_{R_0,[0,\infty]}\left\langle\frac{[\overline\pi]^{s(a)}}{u}\right\rangle	\]
Then it follows from Corollary~\ref{cor: hk invts mod u-pi} that there is some $a_{n}'\in \Lambda_{R_0,[0,\infty]}^{H_K}$ such that 
\[	x-\left(\sum_{i=0}^{n-1}a_i\left(\frac{u-[\overline\pi]^{s(a)}}{u}\right)^i\right)-a_n'\left(\frac{u-[\overline\pi]^{s(a)}}{u}\right)^n	\]
belongs to the ideal
\[	\left(\frac{u-[\overline\pi]^{s(a)}}{u}\right)^{n+1}\Lambda_{R_0,[0,\infty]}\left\langle\frac{[\overline\pi]^{s(a)}}{u}\right\rangle^{H_K}	\]
Since the sequence $\{\alpha_i\}$ is non-decreasing, the summand $a_{n}'\left(\frac{u-[\overline\pi]^{s(a)}}{u}\right)^{n}$
belongs to
\[	u^{-\alpha_{n+1}}\Lambda_{R_0,[0,\infty]}+\left(\frac{u-[\overline\pi]^{s(a)}}{u}\right)^{n+1}\Lambda_{R_0,[0,\infty]}\left\langle\frac{[\overline\pi]^{s(a)}}{u}\right\rangle	\]
and we may write 
\[	ua_n'(u-[\overline\pi]^{s(a)})^n=u^{n+1-\alpha_{n+1}}b_n+(u-[\overline\pi]^{s(a)})^{n+1}c_n	\]
with $b_n\in \Lambda_{R_0,[0,\infty]}$ and $c_n\in\Lambda_{R_0,[0,\infty]}\left\langle\frac{[\overline\pi]^{s(a)}}{u}\right\rangle$.  This implies that 
\[	u^{n+1-\alpha_{n+1}}b_n\in (u-[\overline\pi]^{s(a)})^n\Lambda_{R_0,[0,\infty]}	\]
Since $\Lambda_{R_0,[0,\infty]}/(u-[\overline\pi]^{s(a)})$ has no $u$-torsion, it follows that $b_n$ is a multiple of $(u-[\overline\pi]^{s(a)})^n$ in $\Lambda_{R_0,[0,\infty]}$.  Thus, 
\[	a_n'=u^{n-\alpha_{n+1}}\left(\frac{b_n}{(u-[\overline\pi]^{s(a)})^n}\right)+\left(\frac{u-[\overline\pi]^{s(a)}}{u}\right)c_n	\]
If we consider the image of $a_n'$ in $\left(\Lambda_{R_0,[0,\infty]}/(u-[\overline\pi]^{s(a)})\right)^{H_K}$, we see that it is equal to the image of $u^{n-\alpha_{n+1}}\left(\frac{b_n}{(u-[\overline\pi]^{s(a)})^n}\right)$.  Thus, there is some $a_n\in u^{n-\alpha_{n+1}}\Lambda_{R_0,[0,\infty]}^{H_K}$ such that $a_n\equiv a_n'\pmod{u-[\overline\pi]^{s(a)}}$.  It follows that $ux-\left(\sum_{i=0}^{n}a_i(\frac{u-[\overline\pi]^{s(a)}}{u})^i\right)\in (\frac{u-[\overline\pi]^{s(a)}}{u})^{n+1}\widetilde\Lambda_{R_0,[a_0,\infty]}^{H_K}$.

By induction, we obtain a sequence $\{a_i\}$ of elements of $\Lambda_{R_0,[0,\infty]}^{H_K}$ such that the sum $\sum_{i=0}^{\infty}a_i(\frac{u-[\overline\pi]^{s(a)}}{u})^i$ converges in $\widetilde\Lambda_{R_0,[a,\infty]}^{H_K}$ to $x$.

A similar argument applies to elements of $\widetilde\Lambda_{R_0,[0,b]}^{H_K}$.

\end{proof}

\subsection{Imperfect overconvergent rings}

We now define \emph{imperfect} period rings, which will be noetherian pseudoaffinoid algebras over $R$.  As in the case of perfect overconvergent rings, we would like to consider the fiber products of $\Spa R_0$ and $\Spa R$ with analytic subspaces of $\Spa \A_K^+$.  However, because we only have an explicit description of $\Lambda_{[0,b]}^{H_K}$ and $\Lambda_{[0,b]}^{\circ,H_K}$ for sufficiently small $b$, we restrict our definitions to that setting.

Let $K/\Q_p$ be a finite extension and let $F'\subset K_\infty$ be its maximal unramified subextension.  Recall that we defined 
\[	\mathscr{A}_{F'}^{(0,b]}:=\{\sum_{m\in\Z}a_mX^m:a_m\in\mathscr{O}_{F'}, v_p(a_m)+mb\rightarrow\infty\text{ as }m\rightarrow -\infty\}	\] 
to be the ring of integers of the ring of bounded analytic functions on the half-open annulus $0<v_p(X)\leq b$ over $F'$.  Let 
\[	r_K:=\begin{cases}(2v_{\C_p^\flat}(\mathfrak{d}_{\E_K/\E_F}))^{-1} & \text{if }\E_K/\E_F\text{ is ramified} \\ 1 & \text{otherwise}\end{cases}	\]
  Then we have the following:
\begin{proposition}\cite[Proposition 7.5]{colmez} For $b<r_K$, the assignment $f\mapsto f(\pi_K)$ is an isomorphism of topological rings from $\mathscr{A}_{F'}^{(0,bv_{\C_p^\flat}(\pi_K)]}$ to $\Lambda_{[0,b],K}$.  Furthermore, if we define a valuation $v^{(0,b]}$ on $\mathscr{A}_{F'}^{(0,b]}$ by 
\[	v^{(0,b]}(\sum_{m\in \Z}a_mX^m):=\inf_{m\in\Z}(v_p(a_m)+mb)	\]
then 
\[	\val^{[0,b]}(f(\pi_K))=\frac{1}{b}v^{(0,b]}(f)	\]
where $\val^{[0,b]}$ is the restriction of the corresponding valuation on $\widetilde\Lambda_{[0,b]}$.
\end{proposition}
In the special case $b=0$, $\Lambda_{[0,0]}^{H_K}=\A_K\cong \mathscr{O}_{F'}[\![\pi_K]\!]\left[\frac 1 {\pi_K}\right]^\wedge$, where the completion is $p$-adic.

When $0<b<r_K$, we see in particular that $\pi_K$ is a pseudo-uniformizer of $\Lambda_{[0,b],K}$.  Thus, the pre-adic space $\Spa R_0\times\Spa\Lambda_{[0,b],K}$ is exhausted by affinoid pre-adic spaces of the form $\Spa\left(R_0\htimes \Lambda_{[0,b],K}^\circ\right)\left\langle\frac{u}{\pi_K^{1/(b'\cdot v_{\C_p^\flat}(\overline\pi_K))}}\right\rangle$ for positive rational numbers $b'$ satisfying $\frac{1}{b'\cdot v_{\C_p^\flat}(\overline\pi_K)}\in \N$.  Similarly, the pre-adic space $\Spa R\times\Spa\Lambda_{[0,b],K}$ is exhausted by affinoid pre-adic spaces of the form $\Spa\left(R_0\htimes \Lambda_{[0,b],K}^\circ\right)\left\langle\frac{\pi_K^{1/(a\cdot v_{\C_p^\flat}(\overline\pi_K))}}{u}, \frac{u}{\pi_K^{1/(b'\cdot v_{\C_p^\flat}(\overline\pi_K))}}\right\rangle$ for positive rational numbers $a, b'$ satisfying $0<a\leq b'$ and $\frac{1}{a\cdot v_{\C_p^\flat}(\overline\pi_K)}, \frac{1}{b'\cdot v_{\C_p^\flat}(\overline\pi_K)}\in \N$.

With this in mind, we may reason as in the perfect case and show:
\begin{proposition}\label{prop: imperfect localization}
Suppose $b\in (0,r_K)$, and suppose $R$ is topologically of finite type over $D_\lambda$ for some $\lambda\in\Q_{>0}$.  Then if $b'\leq b\lambda$, the affinoid pre-adic space $\Spa\left(R_0\htimes \Lambda_{[0,b],K}^\circ\right)\left\langle\frac{u}{\pi_K^{1/(b'\cdot v_{\C_p^\flat}(\overline\pi_K))}}\right\rangle$ is isomorphic to the localization $\Spa\left(R_0\htimes \mathscr{O}_{F'}[\![\pi_K]\!]\right)\left\langle\frac{u}{\pi_K^{1/(b'\cdot v_{\C_p^\flat}(\overline\pi_K))}}\right\rangle$ (and is therefore actually a pseudorigid adic space).
\end{proposition}

This motivates the following definition.
\begin{definition}
Let $R$ be a pseudoaffinoid $\Z_p$-algebra such that $R$ is topologically of finite type over $D_\lambda$, and let $K/\Q_p$ be a finite extension.  Fix rational numbers $a\in\Q_{>0}$ and $b\in\Q_{\geq 0}$ with $a\leq b< r_K\cdot\lambda$ such that $\frac{1}{a\cdot v_{\C_p^\flat}(\overline\pi_K))},\frac{1}{b\cdot v_{\C_p^\flat}(\overline\pi_K))}\in\Z$.  Then we define the $\Z_p$-algebra $\Lambda_{R_0,[a,b],K}$ to be the evaluation of the sheaf of rings $\mathscr{O}_{(R_0\otimes\mathscr{O}_{F'})[\![\pi_K]\!]}$ on the affinoid subspace of $\Spa(R_0\otimes\mathscr{O}_{F'})[\![\pi_K]\!]$ defined by the conditions 
\[	u\leq \pi_K^{1/(b\cdot v_{\C_p^\flat}(\overline\pi_K))} \qquad\text{ and }\qquad\pi_K^{1/(a\cdot v_{\C_p^\flat}(\overline\pi_K))}\leq u	\]
This is a pseudoaffinoid algebra with rings of definition 
\[	\Lambda_{R_0,[a,b],0,K}:=(R_0\otimes\mathscr{O}_{F'})[\![\pi_K]\!]\left\langle\frac{u}{\pi_K^{1/(b\cdot v_{\C_p^\flat}(\overline\pi_K))}},\frac{\pi_K^{1/(a\cdot v_{\C_p^\flat}(\overline\pi_K))}}{u}\right\rangle	\]
and 
\[	\Lambda_{R_0,[a,b],0,K,\lambda} \text{ the image of }\left(R_0\htimes\Lambda_{[0,\frac{b}{\lambda}]}^{\circ,H_K}\right)\left\langle\frac{u}{\pi_K^{1/(b\cdot v_{\C_p^\flat}(\overline\pi_K))}},\frac{\pi_K^{1/(a\cdot v_{\C_p^\flat}(\overline\pi_K))}}{u}\right\rangle	\]
and pseudo-uniformizers $u$ and $\pi_K$.

We make an auxiliary definition $\Lambda_{R_0,[0,0],K}:=(R_0\htimes\mathscr{O}_{F'})[\![\pi_K]\!]\left[\frac 1 {\pi_K}\right]_u^\wedge$, where the completion is $u$-adic.

If $I\subset (0,\infty)$ is any interval (with either open or closed endpoints), we set 
\[	\Lambda_{R,I,K}:=\varprojlim_{[a,b]\subset I}\Lambda_{R,[a,b],K}	\]

If $p=0$ in $R$, then we may take $\lambda$ arbitrarily large, and hence $b$ arbitrarily large.  Thus, we additionally define $\Lambda_{R_0,[0,\infty],K}:=(R_0\otimes\mathscr{O}_{F'})[\![\pi_K]\!]$ in this case.
\end{definition}

\begin{remark}
Since $\Lambda_{R,[a,b],K}$ has noetherian ring of definition, the associated space $\Spa \Lambda_{R,[a,b],K}$ is an adic space, not merely a pre-adic space.  Thus, the sheaf property with respect to covers of $\Spa R$ or with respect to change of intervals is automatic.
\end{remark}

The rings $\Lambda_{R_0,I,K}$ are equipped with actions of Frobenius and $\Gamma_K$.  We have ring homomorphisms
\[      \varphi:\Lambda_{R_0,[a,b],K}\rightarrow \Lambda_{R_0,[a/p,b/p],K}      \]
but they are not isomorphisms.
\begin{lemma}
	The Frobenius operator $\varphi$ makes $\Lambda_{R_0,[0,b/p],K}$ into a free $\varphi(\Lambda_{R_0,[0,b],K)})$-module, with basis $\{1,[\varepsilon],\ldots,[\varepsilon]^{p-1}\}$.
	\label{lemma: imperfect free module}
\end{lemma}
\begin{proof}
	Let $\varpi_{F'}$ be a uniformizer for $\O_{F'}$.  Since $\Lambda_{R_0,[0,b],0,K}$ is complete for the $(\varpi_{F'},\pi_K^{1/(b'\cdot v_{\C_p^\flat}(\overline\pi_K))})$-adic topology, it suffices by ~\cite[Theorem 8.4]{matsumura} to prove the corresponding statement for 
\[	\Lambda_{R_0,[0,b],0,K}/(\varpi_{F'},\pi_K^{1/(b'\cdot v_{\C_p^\flat}(\overline\pi_K))})\cong (R/u)\otimes (\O_{F'}/\varpi_{F'})[\![\overline\pi_K]\!][X]	\]
But since $\varphi$ acts trivially on $R/u$, this follows from the classical case.
\end{proof}

We define categories of $(\varphi,\Gamma)$-modules over $\Spa R$:
\begin{definition}
        A $\varphi$-module over $\Lambda_{R,(0,b],K}$ is a coherent sheaf $D$ of  modules over the pseudorigid space $\bigcup_{a\rightarrow 0}\Spa(\Lambda_{R,[a,b],K})$ equipped with an isomorphism
        \[      \varphi_D:\varphi^\ast D\xrightarrow{\sim}\Lambda_{R,(0,b/p],K}\otimes_{\Lambda_{R,(0,b],K}}D   \]
        If $a\in (0,b/p]$, a $\varphi$-module over $\Lambda_{R,[a,b],K}$ is a finite $\Lambda_{R,[a,b],K}$-module $D$ equipped with an isomorphism
        \[      \varphi_{D,[a,b/p]}:\Lambda_{R,[a,b/p],K}\otimes_{\Lambda_{R,[a/p,b/p],K}}\varphi^\ast D\xrightarrow{\sim} \Lambda_{R,[a,b/p],K}\otimes_{\Lambda_{R,[a,b],K}}D  \]

        A $(\varphi,\Gamma_K)$-module over $\Lambda_{R,(0,b],K}$ (resp. $\Lambda_{R,[a,b],K}$) is a $\varphi$-module over $\Lambda_{R,(0,b],K}$ (resp. $\Lambda_{R,[a,b],K}$) equipped with a semi-linear action of $\Gamma_K$ which commutes with $\varphi_D$ (resp. $\varphi_{D,[a,b/p]}$).

	Let $\Lambda_{R,\rig,K}:=\varinjlim_{b\rightarrow 0}\varprojlim_{a\rightarrow 0}\Lambda_{R,[a,b],K}$.  A $(\varphi,\Gamma_K)$-module over $R$ is a module $D$ over $\Lambda_{R,\rig,K}$ which arises via base change from a $(\varphi,\Gamma_K)$-module over $\Lambda_{R,(0,b],K}$ for some $b>0$.
\end{definition}

If $L/K$ is a Galois extension, $\widetilde\Lambda_{R_0,I}^{H_L}$ and $\Lambda_{R_0,I,L}$ are also equipped with actions of $H_{L/K}:=H_K/H_L$.  Thus, it makes sense to introduce $(\varphi,\Gamma)$-modules ``equipped with a $\Gal_{L/K}$-action'' and ask about descent:
\begin{definition}
        If $L/K$ is a finite Galois extension and $D$ is a $(\varphi,\Gamma_L)$-module, we say that $D$ is \emph{equipped with an action of $\Gal_{L/K}$} if the Galois group $\Gal_K$ acts on $D$ and in addition
        \begin{itemize}
                \item   the subgroup $H_L\subset \Gal_K$ acts trivially on $D$, and
                \item   the induced action of $\Gal_L/H_L$ coincides with the action of $\Gamma_L$.
        \end{itemize}
        We also say that $D$ is a \emph{$(\varphi,\Gamma_L, \Gal_{L/K})$-module}.
\end{definition}

In fact, if we restrict ourselves to \emph{projective} $(\varphi,\Gamma)$-modules, we have not enlarged our category.
\begin{lemma}\label{lemma: galois descent phi gamma}
        If $L/K$ is a finite Galois extension, then 
	\[	\Lambda_{R_0,[0,b],K}\rightarrow\Lambda_{R_0,[0,b],L}	\]
	is a finite free extension of rings and $\Lambda_{R_0,I,L}^{H_K}=\Lambda_{R_0,I,K}$.
\end{lemma}
\begin{proof}
        Let $F'\subset K_\infty:=K(\mu_{p^\infty}), F''\subset L_\infty:=L(\mu_{p^\infty})$ be the maximal unramified subfields.  A basis for $\O_{F''}$ over $\O_{F'}$ provides a basis for $(R_0\htimes\O_{F''})[\![\pi_K]\!]$ over $(R_0\htimes\O_{F'})[\![\pi_K]\!]$, so we may assume that $F'=F''$.  Then if $e:=e_{L_\infty/K_\infty}=[L_\infty:K_\infty]$, the set $\{1,\pi_L,\ldots,\pi_L^{e-1}\}$ is a basis for $\Lambda_{R_0,[0,0],L}$ over $\Lambda_{R_0,[0,0],K}$.

        The trace map defines a perfect pairing
        \begin{equation*}
                \begin{split}
                        \Lambda_{R_0,[0,0],L}\times \Lambda_{R_0,[0,0],L} &\rightarrow \Lambda_{R_0,[0,0],K}    \\
                        (x,y) &\mapsto \Tr(xy)
                \end{split}
        \end{equation*}
        The dual basis $\{f_1^\ast=1,\ldots, f_e^\ast\}$ with respect to this pairing is the same as that constructed in ~\cite[\textsection 6.3]{colmez}.  Since $(R_0\htimes\Lambda_{[0,b/\lambda],L})\left\langle\frac{u}{\pi_K^{1/(b\cdot v_{\C_p^\flat}(\overline\pi_L))}}\right\rangle$ is a ring of definition of $\Lambda_{R_0,[0,b],L}$ by Proposition~\ref{prop: imperfect localization}, ~\cite[Corollaire 6.10]{colmez} implies that $f_i^\ast\in \Lambda_{R_0,[0,b],L}$ for all $i$.  Then for any $x\in \Lambda_{R_0,[0,b],L}$, we may uniquely write $x=\sum_i \Tr(x\pi_L^i)f_i^\ast$, as desired.
\end{proof}

\begin{corollary}\label{cor: projectivity descends}
        If $D$ is a projective $(\varphi,\Gamma_L.\Gal_{L/K})$-module of rank $d$ over $\Lambda_{R_0,[0,b],K}$, for some Galois extension $L/K$ and some $b>0$, then $D^{H_K}$ is a projective $(\varphi,\Gamma_K)$-module over $\Lambda_{R_0,[0,b],K}$.
\end{corollary}
\begin{proof}
        By Lemma~\ref{lemma: galois descent phi gamma}, the extension $\Lambda_{R_0,[0,b],K}\rightarrow\Lambda_{R_0,[0,b],L}$ is a finite flat cover; descent of modules is effective and $D^{H_K}$ is the descent of $D$ to $\Lambda_{R_0,[0,b],K}$, so the natural map
        \[      \Lambda_{R_0,[0,b],L}\otimes_{\Lambda_{R_0,[0,b],K}}D^{H_K}\rightarrow D        \]
        is an isomorphism.  We can check flatness after an fppf base change, so $D^{H_K}$ is flat over $\Lambda_{R_0,[0,b],K}$.  We can also check finiteness of a module after an fppf base change, so $D^{H_K}$ is a finite $\Lambda_{R_0,[0,b],K}$-module.  Since $\Lambda_{R_0,[0,b],K}$ is noetherian, it is finitely presented, so projective.

        It remains to define the $\Gamma_K$-action and show that $\varphi:\varphi^\ast D^{H_K}\rightarrow D^{H_K}$ is an isomorphism.  But by assumption $D$ is equipped with an action of $\Gal_K$, so $D^{H_K}$ acquires an action of $\Gal_K/H_K\cong \Gamma_K$ (which is compatible with the action of $\Gamma_L$, by assumption).  Finally, we can check that $\varphi:\varphi^\ast D^{H_K}\rightarrow D^{H_K}$ is an isomorphism after a finite flat base base change, so it follows from the corresponding statement for $D$.
\end{proof}

\begin{proposition}\label{prop: imperfect ring R-flat}
For $0< a\leq b<\infty$, the ring $\Lambda_{R_0,[a,b],K}$ is flat as an $R_0$-module.
\end{proposition}
\begin{proof}
The set $\left\{\left(\frac{u}{\pi_K^{1/(b\cdot v_{\C_p^\flat}(\overline\pi_K))}}\right)^n\right\}_{n\geq 0}\cup\left\{\left(\frac{\pi_K^{1/(a\cdot v_{\C_p^\flat}(\overline\pi_K))}}{u}\right)^n\right\}_{n\geq 1}$ provides a topological basis for $\Lambda_{R_0,[a,b],0,K}$ as an $R_0\otimes\mathscr{O}_{F'}$-module.  Then ~\cite[Tag 06LE]{stacks-project} implies it is flat.
\end{proof}

There are evident maps $\Lambda_{R_0,[a,b],K}\rightarrow \widetilde\Lambda_{R_0,[a,b]}^{H_K}$, and $\Lambda_{R_0,[a,b],K}$ inherits the valuations $v_{R,[a,b]}$ and $v_{R,[a.b],\lambda}$.  We will compute $v_{R,[a,b],\lambda}$ explicitly in the case where $R = D_{\lambda}$, where $\lambda=\frac 1 m$ for some $m\geq 1$.

Every element of $\Lambda_{R_0,[0,b],K}$ can be written uniquely in the form $\sum_{i\in\Z}a_i\pi_K^i$, where $a_i\in \mathscr{O}_{F'}\otimes R_0$ and $a_iu^{ibv_{\C_p^\flat}(\overline\pi_K)}\rightarrow 0$ as $i\rightarrow -\infty$.  When $R=D_\lambda$, this condition can be translated as $v_{D_\lambda}(a_i)+\frac{ibv_{\C_p^\flat}(\overline\pi_K)}{\lambda}\rightarrow \infty$ as $i\rightarrow -\infty$.
\begin{proposition}\label{prop: d1 valuation}
If $R=D_1$ and $\frac{b}<r_K$, then $\inf_{i\in\Z}\{v_{D_1}(a_i)+ibv_{\C_p^\flat}(\overline\pi_K)\}$ is a valuation on $\Lambda_{R_0,[0,b],K}$ whose ring of integers is $\Lambda_{R_0,[0,b],0,K}$, and $v_{D_1,b}(\sum_{i\in\Z}a_i\pi_K^i)=\frac{1}{b}\inf_{i\in\Z}\{v_{D_1}(a_i)+ibv_{\C_p^\flat}(\overline\pi_K)\}$.
\end{proposition}
\begin{proof}
It is straightforward to check that $\inf_{i\in\Z}\{v_{D_1}(a_i)+ibv_{\C_p^\flat}(\overline\pi_K)\}\geq 0$ if and only if $\sum_{i\in\Z}a_i\pi_K^i\in\Lambda_{R_0,[0,b],0,K}$.  Moreover, $v_{D_1,b}(a_i\pi_K^i)=iv_{\C_p^\flat}(\overline\pi_K)+\frac{v_{D_1}(a_i)}{b}$, yielding the second claim. 
\end{proof}

Before we turn to the Tate--Sen axioms, we make a remark about Frobenius on imperfect rings.  Since the Frobenius on $\widetilde{\Lambda}_{R,[0,0]}$ acts via $\varphi(\pi)=(1+\pi)^p-1$ and $\Gamma$ acts via $\gamma(\pi)=(1+\pi)^{\chi(\gamma)}-1$, we see that $\Lambda_{R_0,[0,0],F}$ is stable under the actions of $\varphi$ and $\Gamma$.  Since $\A_K$ is also stable under the actions of $\varphi$ and $\Gamma$, we see that $\varphi$ and $\Gamma$ act on $\Lambda_{R_0,[0,0],K}$, as well.  Since we have isomorphisms $\varphi:\widetilde\Lambda_{R,[a,b]}^{H_K}\rightarrow\widetilde\Lambda_{R,[a/p,b/p]}^{H_K}$, we have induced maps 
\[	\varphi:\Lambda_{R,[a,b],K}\rightarrow\Lambda_{R,[a/p,b/p],K}	\]
However, $\varphi:\Lambda_{R_0,[0,0],K}\rightarrow\Lambda_{R_0,[0,0],K}$ is no longer surjective.  Indeed, $\Lambda_{R_0,[0,0],K}$ is free over $\varphi(\Lambda_{R_0,[0,0],K})$ of rank $p$, and a basis is given by $\{1,[\varepsilon],\ldots,[\varepsilon]^{p-1}\}$.  We may therefore define a left inverse $\psi:\Lambda_{R_0,[0,0],K}\rightarrow\Lambda_{R_0,[0,0],K}$ of $\varphi$ via $\psi(\varphi(a_0)+\varphi(a_1)[\varepsilon]+\ldots +\varphi(a_{p-1})[\varepsilon]^{p-1}) = a_0$, where $a_i\in \Lambda_{R_0,[0,0],K}$.  Note that as $p$ may not be invertible in $R$, we cannot use the definition $\psi=\frac 1 p \varphi^{-1}\circ\Tr_{\Lambda_{R_0,[0,0],K}/\varphi(\Lambda_{R_0,[0,0],K})}$ from the classical case.

\section{The Tate--Sen axioms for families}

Given a Galois representation with coefficients in $\Z_p$, base extension gives us a vector bundle over $\mathcal{Y}$.  The various $(\varphi,\Gamma)$-modules associated to the representation are constructed by studying the $H_K$-invariants of restrictions of this vector bundle to various rational subdomains of $\mathcal{Y}$.  

Now suppose $R$ is a pseudoaffinoid $\Z_p$-algebra and $R_0\subset R$ is a noetherian ring of definition.  If we have a Galois representation with coefficients in $R$ which admits a Galois-stable $R_0$-lattice, we may similarly pass by base extension to a vector bundle over $\Spa R_0\times\mathcal{Y}$.  The Tate--Sen axioms will let us descend that vector bundle (restricted to an affinoid subdomain) to a vector bundle over an \emph{imperfect} overconvergent ring.

The Tate--Sen axioms concern a profinite group $G_0$, an open normal subgroup $H_0\subset G_0$ such that $G_0/H_0$ contains $\Z_p$ as an open subgroup, a valued ring $\widetilde\Lambda$ with a continuous action of $G_0$, and a collection of subrings $\{\Lambda_{H,k}\}_{k\gg 0}$ of $\widetilde\Lambda^H$, where $H$ is any open subgroup of $H_0$.  These axioms permit us to descend continuous $1$-cocycles of $G_0$ from $\widetilde\Lambda$ to some $\Lambda_{H,k}$.  

The axioms are as follows:
\begin{enumerate}
\item[{\bf TS1}]	There is a constant $c_1\in\R_{>0}$ such that for all open subgroups $H_1\subset H_2$ in $H_0$ that are normal in $G$, there is some $\alpha\in\widetilde{\Lambda}^{H_1}$ satisfying $v_{\Lambda}>-c_1$ and $\sum_{\tau\in H_2/H_1}\tau(\alpha)=1$.

\item[{\bf TS2}]	There is a constant $c_2\in\R_{>0}$ such that for all open subgroups $H\subset H_0$ that are normal in $G$, there is a collection $\{\Lambda_{H,k},R_{H,k}\}_{k\geq n(H)}$, where $\Lambda_{H,k}\subset \widetilde{\Lambda}^H$ is a closed subalgebra and $R_{H,k}:\widetilde{\Lambda}^H\rightarrow \Lambda_{H,k}$ is a $\Lambda_{H,k}$-linear map such that
	\begin{enumerate}
	\item	if $H_{L_1}\subset H_{L_2}$, and $k\geq\max\{n(H_{L_1}),n(H_{L_2})\}$, then $\Lambda_{H_2,k}\subset\Lambda_{H_1,k}$ and $R_{H_1,k}|_{\Lambda_{H_2,k}}=R_{H_2,k}$
	\item	$R_{H,k}$ is a $\Lambda_{H,k}$-linear section to the inclusion $\Lambda_{H,k}\hookrightarrow \widetilde{\Lambda}^H$
	\item	$g(\Lambda_{H,k})=\Lambda_{H,k}$ and $g(R_{H,k}(x))=R_{H,k}(gx)$ for all $x\in\widetilde{\Lambda}^H$ and $g\in G_0$
	\item	$v_\Lambda(R_{H,k}(x))\geq v_\Lambda(x)-c_2$ for all $x\in\widetilde{\Lambda}^H$
	\item	$\lim_{k\rightarrow\infty}R_{H,k}(x)=x$ for all $x\in\widetilde{\Lambda}^H$
	\end{enumerate}

\item[{\bf TS3}]	There is a constant $c_3\in\R_{>0}$ such that for every open normal subgroup $G\subset G_0$ (setting $H:=G\cap H_0$) there is an integer $n(G)\geq\max\{n_1(G),n(H)\}$ such that
	\begin{enumerate}
	\item	$\gamma-1$ is invertible on $X_{H,k}:=\ker(R_{H,k})$
	\item	$v_\Lambda(x)\geq v_{\Lambda}((\gamma-1)(x))-c_3$ for all $x\in X_{H,k}$
	\end{enumerate}
for all $k\geq n(G)$ and all $\gamma\in G_0/H$ with $n(\gamma)\leq k$.
\end{enumerate}
In other words, $\Lambda_{H,k}$ is a summand of $\widetilde{\Lambda}^H$ (as a $\Lambda_{H,k}$-module), and a topological generator of $\Gamma$ acts invertibly (with continuous inverse) on its complement.

Colmez showed that if we take $G_0:=\Gal_{\Q_p}$, $H_0:=\ker(\chi)$, $\widetilde{\Lambda}:=\widetilde{\Lambda}_{[0,1]}$ (with the valuation $\val^{(0,1]}$), and $\Lambda_{H_K,k}:=\varphi^{-k}\left(\Lambda_{[0,1]}^{H_K}\right)$,
then the Tate--Sen axioms are satisfied for any choices $c_1>0$, $c_2>0$, and $c_3>1/(p-1)$~\cite[Proposition 4.2.1]{berger-colmez}.  Here $K$ is a finite extension of $\Q_p$.

Suppose that $R$ is topologically of finite type over $D_\lambda$.  We will check that the Tate--Sen axioms hold for $\widetilde\Lambda_{R_0,[0,b]}$ for $b$ sufficiently small, using the ring of definition $\widetilde\Lambda_{R_0,[0,b],0,\lambda}\subset\widetilde\Lambda_{R_0,[0,b]}$ and the associated valuation $v_{R,b,\lambda}$.

\begin{proposition}\label{prop:ts-1-r}
The ring $\widetilde\Lambda_{R_0,[0,b]}$ satisfies the first Tate--Sen axiom for any $b>0$ and any $c_1>0$.
\end{proposition}
\begin{proof}
Choose $c_1>0$.  Then for any appropriate subgroups $H_{L_1}\subset H_{L_2}$ of $H$, the proof of~\cite[Lemme 10.1]{colmez} constructs $\beta\in\widehat{L_{\infty}^\flat}$ such that $\Tr_{\widehat{L_\infty^\flat}/\widehat{K_\infty^\flat}}(\beta)=1$, with $v_{\C_p^\flat}(\beta)$ arbitrarily close to $0$.  This implies that $\Tr_{\widehat{L_\infty^\flat}/\widehat{K_\infty^\flat}}([\beta])=\sum_{i\geq 0}p^i[x_i]$ is a unit of $\widetilde\Lambda_{[0,\frac{b}{\lambda}]}^{\circ,H_K}$, and therefore that $(\Tr_{\widehat{L_\infty^\flat}/\widehat{K_\infty^\flat}}([\beta]))^{-1}[\beta]\in \widetilde\Lambda_{[0,\frac{b}{\lambda}]}^{\circ,H_K}$ satisfies $v_{R,b,\lambda}\left((\Tr_{\widehat{L_\infty^\flat}/\widehat{K_\infty^\flat}}([\beta]))^{-1}[\beta]\right)\geq v_{\C_p^\flat}(\beta)$.  Thus, we merely need to choose $\beta$ such that $v_{\C_p^\flat}(\beta)>-c_1$.
\end{proof}

\begin{corollary}\label{cor: ts-1-r hk-invts}
Suppose $M$ is a finite free $R_0$-module of rank $d$ equipped with a continuous $R_0$-linear action of $\Gal_K$.  Then there is some finite extension $L/K$ such that $\widetilde{D}_{L}(M):=\left(\widetilde\Lambda_{R_0,[0,1],0,\lambda}\otimes_{R_0}M\right)^{H_{L}}$ is free over $\widetilde\Lambda_{R_0,[0,1],0,\lambda}^{H_{L}}$ of rank $d$.
\end{corollary}
\begin{proof}
Choose a basis of $M$ and let $\rho:\Gal_K\rightarrow\GL_d(R_0)$ denote the Galois representation corresponding to $M$.  Let $c_\tau\in \H^1(H_K,\GL_d(\widetilde\Lambda_{R_0,[0,1],0,\lambda}))$ be the corresponding cocycle.  If we let $L/K$ be the finite extension corresponding to the kernel of the homomorphism $\overline\rho:\Gal_K\rightarrow\GL(M/u)$, the proofs of~\cite[Lemme 3.2.1]{berger-colmez} and \cite[Corollaire 3.2.2]{berger-colmez} applied to the image of $c$ in $\H^1(H_{L},\GL_d(\widetilde\Lambda_{R_0,[0,1],0,\lambda}))$ carry over nearly verbatim (we need to work modulo powers of $u$ rather than $p$, since $p$ might not be a pseudo-uniformizer), and we conclude that the restriction of $c$ is trivial.  The result follows.
\end{proof}

\subsection{Normalized trace maps}

The next step is to construct so-called normalized trace maps.  In the classsical setting, this has the following form:
\begin{proposition}[{\cite[Cor. 8.11]{colmez}}]
Suppose $0<b$ and $p^{-n}b<r_K$.  Then there is a constant $c_K(b)$ (depending on $K$ and $b$) and a $\varphi^{-n}\left(\Lambda_{[0,p^{-n}b]}^{H_K}\right)$-linear map $R_{K,n}:\widetilde\Lambda_{[0,b]}^{H_K}\rightarrow \varphi^{-n}\left(\Lambda_{[0,p^{-n}b]}^{H_K}\right)$ such that 
\begin{enumerate}
\item	$R_{K,n}$ is a section to the inclusion $\varphi^{-n}\left(\Lambda_{[0,p^{-n}b]}^{H_K}\right)\rightarrow \widetilde\Lambda_{[0,b]}^{H_K}$
\item	$R_{K,n}(x)\rightarrow x$ as $n\rightarrow\infty$ and $v_b(R_{K,n}(x))\geq v_b(x)-p^{-n}c_K(b)$.
\item	$R_{K,n}$ commutes with the action of $\Gamma_K$.
\end{enumerate}
\end{proposition}

The construction of $R_{K,n}$ uses the fact that $\{[\varepsilon]^i\mid i\in \Z[\frac 1 p]\cap [0,1)\}$ provides a topological basis for $\widetilde\Lambda_{[0,0]}^{H_K}$ over $\Lambda_{[0,0]}^{H_K}$ (and in fact for $\widetilde\Lambda_{[0,\infty]}^{H_F}$ over $\Lambda_{[0,\infty]}^{H_F}$ when $F/\Q_p$ is unramified).  In other words, for $x\in\widetilde\Lambda_{[0,0]}^{H_K}$, one can write $x=\sum_i a_i(x)[\varepsilon]^i$ for unique $a_i(x)$ tending to $0$ with respect to the cofinite filter.  Then one bounds $a_i(x)$ in terms of $x$ and shows that $a_i(x)$ has the correct analyticity properties when $x\in\widetilde\Lambda_{[0,b]}^{H_K}$ for $b>0$, and defines $R_{K,n}(x) := \sum_{v_p(i)\geq -n}a_i(x)[\varepsilon]^i$.

If $R$ is a classical affinoid algebra and $R_0$ is its ring of definition, one can extend $R_{K,n}$ by linearity to define normalized trace maps $R_0\htimes \widetilde\Lambda_{[0,b]}^{H_K}\rightarrow R_0\htimes\varphi^{-n}\Lambda_{[0,p^{-n}b]}^{H_K}$, as in ~\cite[Proposition 3.1.4]{berger-colmez}. We wish to extend this to the setting of pseudoaffinoid algebras where $p\notin R^\times$ and construct maps $R_{K,n}:\widetilde\Lambda_{R,[0,b]}^{H_K}\rightarrow\varphi^{-n}\left(\Lambda_{R,[0,p^{-n}b],K}\right)$ for sufficiently small $b$.  

We first observe that if $R$ is topologically of finite type over $D_\lambda$ for some $\lambda\in\Q_{>0}$, then $\widetilde\Lambda_{R_0,[a,b]}^{H_K}\cong R_0\htimes_{D_\lambda^\circ}\widetilde\Lambda_{D_\lambda^\circ,[a,b]}^{H_K}$ and $\Lambda_{R_0,[a,b],K}\cong R_0\htimes_{D_\lambda^\circ}\Lambda_{D_\lambda^\circ,[a,b],K}$.  Since there is always some choice of $\lambda=\frac{m'}{m}$ such that $R$ is topologically of finite type over $D_\lambda$, we may construct normalized trace maps when $R=D_\lambda$ and extend by linearity.

\begin{proposition}
If $\frac{b}{\lambda}<r_K$, then  $\widetilde\Lambda_{D_\lambda^\circ,[0,b]}^{H_K}$ is a topologically free $\Lambda_{D_\lambda^\circ,[0,b],K}$-module, with basis $\{[\varepsilon]^i\}$.  If $x\in \widetilde\Lambda_{D_\lambda^\circ,[0,b]}^{H_K}$ and $a_i(x)$ denotes the coefficient of $[\varepsilon]^i$ when we decompose $x$, then $v_{D_\lambda,b,\lambda}(x)\geq \inf_i v_{D_\lambda,b,\lambda}(a_i(x))\geq v_{D_\lambda,b,\lambda}(x)-c_K(b,\lambda)$, where $c_K(b,\lambda)$ is a constant depending on the field $K$, $b$, and $\lambda$.
\end{proposition}
\begin{proof}
Since $\frac{b}{\lambda} < r_K$, ~\cite[Lemme 6.5]{colmez} implies that $\frac{\pi_K}{[\overline\pi_K]}$ is a unit of $\widetilde\Lambda_{[0,\frac{b}{\lambda}]}^{\circ, H_K}$.  It follows that $\Spa\left(\widetilde\Lambda_{D_\lambda^\circ,[0,b]}^{H_K}\right)$ can also be constructed as the rational localization $\Spa(D_\lambda^\circ\htimes\widetilde\Lambda_{[0,\frac{b}{\lambda}]}^{\circ,H_K})\left\langle \frac{u}{\pi_K^{1/(b\cdot v_{\C_p^\flat}(\overline\pi_K))}}\right\rangle$.

Since $\widetilde\Lambda_{[0,\frac{b}{\lambda}]}^{\circ,H_K}$ is topologically free over $\Lambda_{[0,\frac{b}{\lambda}]}^{\circ,H_K}$ with basis $\{[\varepsilon]^i\}$, the same is true for $\left(D_\lambda^\circ\htimes\widetilde\Lambda_{[0,\frac{b}{\lambda}]}^{\circ,H_K}\right)\left\langle\frac{u}{\pi_K^{1/(b\cdot v_{\C_p^\flat}(\overline\pi_K))}}\right\rangle$ as a $\left(D_\lambda^\circ\htimes\Lambda_{[0,\frac{b}{\lambda}]}^{\circ,H_K}\right)\left\langle\frac{u}{\pi_K^{1/(b\cdot v_{\C_p^\flat}(\overline\pi_K))}}\right\rangle$-module. Inverting $\pi_K$ gives the desired result.

The claim about valuations follows directly from~\cite[Proposition 8.10]{colmez}; $c_K(b,\lambda)$ is the constant Colmez denotes $c_K\left(\frac{b}{\lambda}\right)$.

\end{proof}

As in the classical setting, we now define 
\[	R_{K,n}:\widetilde\Lambda_{R_0,[0,b]}^{H_K}\rightarrow\varphi^{-n}\left(\Lambda_{R_0,[0,p^{-n}b],K}\right)	\]
	when $0<\frac{b}{\lambda}<r_K$ and $p^{-n}\frac{b}{\lambda}<r_K$, by setting
\[	R_{K,n}(x) = \sum_{v_p(i)\geq -n}a_i(x)[\varepsilon]^i	\]
We see from the construction that $R_{K,n}$ is a continuous $\varphi^{-n}\left(\Lambda_{R_0,[0,p^{-n}b],K}\right)$-linear section to the inclusion $\varphi^{-n}\left(\Lambda_{R_0,[0,p^{-n}b],K}\right)\rightarrow\Lambda_{R_0,[0,b]}^{H_F}$.  In addition, the construction of $a_i$ makes clear that $a_{\chi(\gamma)}(\gamma(x))=\gamma(a_i(x))$ for $\gamma\in\Gamma_K$ and $x\in \widetilde\Lambda_{R_0,[0,b]}^{H_K}$, so $R_{K,n}$ commutes with the action of $\Gamma_K$.

Moreover, $R_{K,n} = \varphi^{-n}\circ R_{K,0}\circ \varphi^n$, so
\begin{align*}
v_{R,b,\lambda}(R_{K,n}(x))&\geq p^{-n}v_{R,p^nb,\lambda}(a_{0}(\varphi^n(x))	\\
&\geq p^{-n}(v_{R,p^nb,\lambda}(\varphi^n(x))-c_K(b,\lambda)	\\
&\geq v_{R,b,\lambda}(x)-p^{-n}c_K(b,\lambda)
\end{align*}

In summary, we have shown the following:
\begin{proposition}
Suppose $0<\frac{b}{\lambda}<r_K$ and $p^{-n}\frac{b}{\lambda}<r_K$.  Then there are constants $c_K(b,\lambda)$ and a continuous $\varphi^{-n}\left(\Lambda_{R_0,[0,p^{-n}b],K}\right)$-linear map $R_{K,n}:\widetilde\Lambda_{R_0,[0,b]}^{H_K}\rightarrow \varphi^{-n}\left(\Lambda_{R_0,[0,p^{-n}b],K}\right)$ such that
\begin{enumerate}
\item	$R_{K,n}$ is a section to the inclusion $\varphi^{-n}\left(\Lambda_{R_0,[0,p^{-n}b],K}\right)\rightarrow\widetilde\Lambda_{R_0,[0,b]}^{H_K}$
\item	$R_{K,n}(x)\rightarrow x$ as $n\rightarrow\infty$ and $v_{R,b,\lambda}(R_{K,n}(x))\geq v_{R,b,\lambda}(x)-p^{-n}c_K(b,\lambda)$
\item	$R_{K,n}$ commutes with the action of $\Gamma_K$.
\end{enumerate}
\end{proposition}

\begin{corollary}
Suppose $0<\frac{b}{\lambda}<r_K$.  Then for any $c_2>0$, the collection $\{\varphi^{-n}(\Lambda_{D_\lambda^\circ,[0,b],K}), R_{K,n}\}$ satisfies the second Tate--Sen axiom for sufficiently large $n$ (depending on the choice of $c_2$).
\end{corollary}

\subsection{The action of $\Gamma$}

For any finite extension $K/F$, the cyclotomic character defines a homomorphism $\chi:\Gamma_K\rightarrow\Z_p^\times$.  For any $\gamma\in\Gamma_K$ of infinite order, we let $n(\gamma):=v_p(\chi(\gamma)-1)\in\Z_{\geq 0}\cup\{\infty\}$.  

\begin{lemma}
Let $R_0\subset R$ be a noetherian ring of definition, formally of finite type over $\Z_p$.  If $\gamma\in\Gamma_K$ has infinite order, then 
\[	\Lambda_{R_0,[0,0],K}^{\gamma=1}=\widetilde\Lambda_{R_0,[0,0]}^{H_K,\gamma=1}=R_0\otimes\mathscr{O}_{F'}^{\gamma=1}	\]
\end{lemma}
\begin{proof}
	We first consider the case where $p=0$ in $R$, and we compute the subspaces of $(R_0/u)\htimes k_{F'}(\!(\pi_K)\!)$ and $(R_0/u)\htimes \widehat{K}_\infty^\flat$ fixed by $\gamma$.  Since $R_0/u$ is an $\F_p$-vector space, we may choose a basis $\{e_i\}_{i\in I}$ and write
\[      (R_0/u)\htimes k_{F'}(\!(\pi_K)\!)\cong \{\sum_{i\in I}a_ie_i\mid a_i\in k_{F'}(\!(\pi_K)\!), a_i\rightarrow 0\}      \]
and
\[	(R_0/u)\htimes \widehat{K}_{\infty}^\flat\cong \{\sum_{i\in I}a_ie_i\mid a_i\in \widehat{K}_{\infty}^\flat, a_i\rightarrow 0\}	\]
The action of $\Gamma_K$ on $R_0/u$ is trivial, so ~\cite[Proposition 9.1]{colmez} implies that 
\[	\left(\Lambda_{R_0,[0,0],K}/u\right)^{\gamma=1} = \left(\widetilde\Lambda_{R_0,[0,0]}^{H_K}/u\right)^{\gamma=1} = (R_0/u)\otimes k_{F'}^{\gamma=1}	\]
Then we approximate elements of  $\left(\Lambda_{R_0,[0,0],K}\right)^{\gamma=1}$ and $\left(\widetilde\Lambda_{R_0,[0,0]}\right)^{H_K,\gamma=1}$ $u$-adically to obtain the desired result.

To bootstrap to the case where $R_0$ is $\Z_p$-flat, we may assume that $p\notin R^\times$ (as the classical case is handled by ~\cite{berger-colmez}) and again filter $R_0$ by the ideals $\{I_j\}$ where $I_j:=p^jR\cap R_0$.  Then $I_j/I_{j+1}$ is a finite $u$-torsion-free $R_0/I_1$-module (and therefore an $\F_p$-vector space), so we may apply the previous argument to calculate the $\gamma$-invariants of $(I_j/I_{j+1})\htimes k_{F'}(\!(\pi_K)\!)$ and $(I_j/I_{j+1})\htimes \widehat{K}_{\infty}^\flat$. Since $\cap_j I_j=(0)$, the result follows.
\end{proof}

\begin{proposition}\label{prop: gamma-1 inv d}
If $D$ is a finite $\Lambda_{R_0,[0,b],K}$-module equipped with commuting semi-linear actions of $\varphi$ and $\Gamma_K$ (such that the action of $\Gamma_K$ is continuous), there is some $n\geq 1$ such that $\gamma-1$ acts on 
\[	\left(\Lambda_{R_0,[0,bp^{-n}],K}\otimes_{\Lambda_{R_0,[0,b],K}}D\right)^{\psi=0}	\]
with continuous inverse for any $\gamma\in\Gamma_K$.
\end{proposition}
\begin{proof}
Let $D_{[0,b/p^n]}:=\Lambda_{R_0,[0,bp^{-n}],K}\otimes_{\Lambda_{R_0,[0,b],K}}D$.  We have a decomposition $(D_{[0,b/p^n]})^{\psi=0}=\oplus_{j\in(\Z/p^n)^\times}[\varepsilon]^{\widetilde j}\varphi^n(D)$, so it suffices to show that $\gamma-1$ has a continuous inverse on $[\varepsilon]\varphi^n(D)$ for some sufficiently large $n$.  Moreover, since $\gamma^n-1=(\gamma-1)(\gamma^{n-1}+\cdots+1)$, we may replace $\Gamma_K$ with a finite-index subgroup.  

Thus, it suffices to consider $\gamma_n\in\Gamma_K$ such that $\chi(\gamma_n)=1+p^n$.  In that case, for $x\in D$,
\begin{equation*}
\begin{split}
\gamma_n\left([\varepsilon]\varphi^n(x)\right)-[\varepsilon]\varphi^n(x) &= [\varepsilon][\varepsilon]^{p^n}\varphi^n(\gamma_n(x)) - [\varepsilon]\varphi^n(x)	\\
&=[\varepsilon]\varphi^n([\varepsilon]\gamma_n(x)-x)	\\
&=[\varepsilon]\varphi^n(G_{\gamma_n}(x))
\end{split}
\end{equation*}
where $G_{\gamma_n}(x):=[\varepsilon]\gamma_n(x)-x = ([\varepsilon]-1)\cdot\left(1+\frac{[\varepsilon]}{[\varepsilon]-1}(\gamma_n-1)\right)(x)$.  If we can find $n$ such that $\sum_{k=0}^\infty\left(-\frac{[\varepsilon]}{[\varepsilon]-1}(\gamma_n-1)\right)^k$ converges on $D$, we will therefore be done.  But the existence of such an $n$ follows from continuity of the action of $\Gamma_K$ on $D$.
\end{proof}

In the special case $D=\Lambda_{R_0,[0,b],K}$, we can deduce constructive bounds for $n$ and the operator norm of $(\gamma-1)^{-1}$.
\begin{lemma}\label{lemma: gamma-1 inv}
If $R$ is a $D_\lambda$-algebra for $\lambda=\frac 1 m$
, $\gamma\in\Gamma_K$ satisfies $n(\gamma)\geq n_0(K)$ and $p^{n(\gamma)}> \frac{2p}{p-1}$ (where $n_0(K)$ is a constant defined in \cite{colmez} depending on the conductor of $K$), and $\frac{b}{\lambda}<r_K$, then $\gamma-1$ is continuously invertible on $\Lambda_{R_0,[0,p^{-n(\gamma)}b],K}^{\psi=0}$, and
\[	v_{R,p^{-n(\gamma)}b}\left(\gamma-1)^{-1}(x)\right)\geq v_{R,p^{-n(\gamma)}b}(x) + p^{n(\gamma)}v_{\C_p^\flat}(\overline\pi)	\]

\end{lemma}
\begin{remark}
We may always assume that $R$ is a $D_\lambda$-algebra with $\lambda=\frac 1 m$
, since we are free to shrink $\lambda$.
\end{remark}
\begin{proof}
If $R$ is a $D_\lambda$-algebra, then $\Lambda_{R_0,[0,b],K}\cong R_0\htimes_{D_\lambda^\circ}\Lambda_{D_\lambda^\circ,[0,b],K}$, and it suffices to prove that $\gamma-1$ is continuously invertible on $\Lambda_{D_\lambda^\circ,[0,b],K}^{\psi=0}$.  Furthermore, $\Lambda_{D_\Lambda^\circ,[0,b],K}\rightarrow \Lambda_{D_1^\circ,[0,\frac{b}{\lambda}],K}$ via $u\mapsto u^m$ makes $\Lambda_{D_1^\circ,[0,\frac{b}{\lambda}],K}$ a finite free $\Lambda_{D_\Lambda^\circ,[0,b],K}$-module, and it suffices to check that $\gamma-1$ is continuously invertible on $\Lambda_{D_1^\circ,[0,\frac{b}{\lambda}],K}^{\psi=0}$.

We therefore bound the operator $\gamma-1$ on $\Lambda_{D_1^\circ,[0,\frac{b}{\lambda}],K}$.  As in the proof of~\cite[Corollaire 9.5]{colmez}, given $f(\pi_K)=\sum_{i\in\Z}a_i\pi_K^i\in \Lambda_{D_1^\circ,[0,\frac{b}{\lambda}],K}$, we write the Taylor expansion of $f(\gamma(\pi_K))-f(\pi_K)$ around $\gamma=1$:
\begin{align*}
f(\gamma(\pi_K))-f(\pi_K) &= \sum_{k\geq 1}\frac{f^{(k)}(\pi_K)}{k!}(\gamma(\pi_K)-\pi_K)^k 	\\
&= \sum_{k\geq 1}\frac{f^{(k)}(\pi_K)\cdot \pi_K^k}{k!}\left(\frac{\gamma(\pi_K)}{\pi_K}-1\right)^k
\end{align*}
Since $\frac{f^{(k)}(\pi_K)\cdot \pi_K^k}{k!} = \sum_{i\in\Z}\binom i k a_i\pi_K^i$ and $\binom i k\in\Z$, we see that 
\[	v_{D_1,b/\lambda}\left(\frac{f^{(k)}(\pi_K)\cdot \pi_K^k}{k!}\right)\geq v_{D_1,b/\lambda}(f(\pi_K)	\]
(using Proposition~\ref{prop: d1 valuation}).  We conclude that 
\begin{equation*}
\begin{split}
v_{D_1,b/\lambda}(\gamma(f(\pi_K))-f(\pi_K))&\geq v_{D_1,b/\lambda}(f(\pi_K))+v_{D_1,b/\lambda}\left(\frac{\gamma(\pi_K)}{\pi_K}-1\right)	\\
&\geq  v_{D_1,b/\lambda}(f(\pi_K))+p^{n(\gamma)}v_{\C_p^\flat}(\overline\pi) - c_K - v_{\C_p^\flat}(\overline\pi)
\end{split}
\end{equation*}
by~\cite[Lemme 9.4]{colmez}, where $c_K$ is a constant satisfying $c_K\leq \frac{1}{p-1}p^{n_0(K)}+v_{\C_p^\flat}(\overline\pi)$~(\cite[Proposition 4.12]{colmez}).  
Since $\frac{b}{\lambda}<r_K$ and $n(\gamma)\geq n_0(K)$, we have
\[	v_{D_1,b/\lambda}\left(-\frac{[\varepsilon]}{[\varepsilon]-1}(\gamma-1)(f(\pi_K))\right)\geq v_{D_1,b/\lambda}(f(\pi_K)) + (p^{n(\gamma)}-1)v_{\C_p^\flat}(\overline\pi)-c_K	\]
so the desired inverse exists on $\Lambda_{D_1,[0,p^{-n(\gamma)}b/\lambda],K}^{\psi=0}$ so long as $p^{n(\gamma)+1}-p^{n_0(K)}-2p>0$.  The assumption that $p^{n(\gamma)}> \frac{2p}{p-1}$ is sufficient to ensure this.
\end{proof}

\begin{proposition}
Suppose $R$ is a $D_\lambda$-algebra for $\lambda=\frac 1 m$
, suppose $n$ satisfies $n\geq n_0(K)$ and $p^n>\frac{2p}{p-1}$, and $\frac{b}{\lambda}<r_K$. If $\gamma\in\Gamma_K$ satisfies $n(\gamma)\leq n$, then $\gamma-1$ is invertible on $X_{R_0,[0,b],K}^n:=\ker(R_{K,n})$, and its inverse is continuous.
\end{proposition}
\begin{proof}
We may again begin by replacing $R$ with $D_1$ and $\Lambda_{R_0,[0,b]}$ with $\Lambda_{D_1^\circ,[0,\frac{b}{\lambda}]}$, so that $v_{D_1,b/\lambda}=v_{D_1,b/\lambda,1}$.

As in the proof of~\cite[Proposition 9.9]{colmez}, we first observe that $\gamma-1$ is injective on $X_{D_1^\circ,[0,\frac{b}{\lambda}],K}^n$, since $\widetilde\Lambda_{D_1^\circ,[0,0],K}^{\gamma=1}=\Lambda_{D_1^\circ,[0,0],K}^{\gamma=1}$.

If $x\in X_{D_1^\circ,[0,\frac{b}{\lambda}],K}^n$, we can write 
\[	x=\sum_{j=n+1}^\infty\left(R_{K,j}(x)-R_{K,j-1}(x)\right) = \sum_{\substack{i\in\Z[1/p]\cap [0,1) \\ v_p(i)=-j}}a_i(x)[\varepsilon]^i	\]
and we have $v_{D_1,b/\lambda}\left(R_{K,j}(x)-R_{K,j-1}(x)\right) \geq v_{D_1,b/\lambda}(x)-p^{1-j}c_K(b,\lambda)$.  If $v_p(i)=-j$, then $\varphi^j(a_i(x)[\varepsilon]^i)\in \oplus_{\widetilde i=1}^{p-1}[\varepsilon]^{\widetilde i}\Lambda_{D_1^\circ,[0,p^{-j}b/\lambda],K}=\Lambda_{D_1^\circ,[0,p^{-j}b/\lambda],K}^{\psi=0}$; it follows by Lemma~\ref{lemma: gamma-1 inv} that if $j\geq n+1$, there is some $y_j\in \Lambda_{D_1^\circ,[0,p^{-j}b/\lambda]}^{\psi=0}$ with
\[	\varphi^j\left(R_{K,j}(x)-R_{K,j-1}(x)\right)=(\gamma-1)(y_j)	\]
and 
\[	v_{D_1,p^{-j}b/\lambda}(y_j)\geq p^jv_{D_1,b/\lambda}\left(R_{K,j}(x)-R_{K,j-1}(x)\right)-p^jv_{\C_p^\flat}(\overline\pi)	\]
Therefore, 
\[	v_{D_1,b/\lambda}(\varphi^{-j}(y_j))\geq v_{D_1,b/\lambda}\left(R_{K,j}(x)-R_{K,j-1}(x)\right)-v_{\C_p^\flat}(\overline\pi)	\]
Since the terms $R_{K,j}(x)-R_{K,j-1}(x)$ tend to $0$ in $\widetilde\Lambda_{D_1,[0,\frac{b}{\lambda}]}^{H_K}$ as $j\rightarrow \infty$, the same is true for $\varphi^{-j}(y_j)$.  Thus, $\sum_{j\geq n+1}\varphi^{-j}(y_j)$ converges to an element $y\in\widetilde\Lambda_{D_1,[0,b]}^{H_K}$ such that $(\gamma-1)(y)=x$ and
\begin{equation*}
\resizebox{\displaywidth}{!}{$
	v_{D_1,b/\lambda}(y)\geq \inf_{j\geq n+1}v_{D_1,b/\lambda}(\varphi^{-j}(y_j))\geq v_{D_1,b/\lambda}(x)-\sup_{j\geq n+1}\{p^{1-j}c_K(b,\lambda)+v_{\C_p^\flat}(\overline\pi)\}
$}
\end{equation*}
This supremum exists, so $(\gamma-1)^{-1}$ is continuous.
\end{proof}

\begin{corollary}
Suppose $R$ is a $D_\lambda$-algebra, where $\lambda=\frac{1}{m}$
, and $0<\frac{b}{\lambda}<r_K$. Then for any $c_3>\frac{p}{p-1}$, the third Tate--Sen axiom holds for the ring $\widetilde\Lambda_{R,[0,b]}^{H_K}$, the maps $\{R_{H,n}\}_{n\gg1}$, and the natural action of $\Gamma_K$.
\end{corollary}

\subsection{The construction of $(\varphi,\Gamma)$-modules}

Now we may apply the arguments of~\cite{berger-colmez} (working $u$-adically rather than $p$-adically) to construct $(\varphi,\Gamma)$-modules over rings $\Lambda_{R,[0,b],K}$.  Let $R$ be a pseudoaffinoid Tate ring over $D_\lambda$, where $\lambda=\frac 1 m$ with $p\nmid m$, with ring of definition $R_0\subset R$ and pseudo-uniformizer $u\in R_0$, and let $M$ be a free $R$-module of rank $d$ equipped with a continuous $R$-linear action of $\Gal_K$.

Following~\cite[Lemme 3.18]{chenevier}, we first find a Galois-stable lattice in $M$.
\begin{lemma}
Let $R$ be as above, and let $M$ be a free $R$-module of rank $d$ equipped with a continuous $R$-linear action of a compact topological group $G$.  Then there is a formal scheme $\mathcal{Y}\rightarrow \Spf(R_0)$ and a finite projective $\mathscr{O}_{\mathcal{Y}}$-module $\mathscr{M}$ equipped with a continuous $\mathscr{O}_{\mathcal{Y}}$-linear action of $G$ such that the natural map $\Spa R\rightarrow \Spa R_0$ factors through a morphism $f:\Spa R\rightarrow\mathcal{Y}^{\mathrm{ad}}$ and $M\cong f^\ast\mathscr{M}$.
\end{lemma}
\begin{proof}
We first observe that there is a finitely generated $R_0$-module $M_0\subset M$ such that $R\otimes_{R_0}M_0=M_0[1/u]=M$.  Indeed, we may simply consider a basis of $M$ and let $M_0$ be the $R_0$-module it generates inside $M$.

Since $R_0$ is noetherian and $u$-adically complete, $\Spf R_0$ is an admissible formal scheme, and the argument of~\cite[Lemme 3.18]{chenevier} goes through verbatim to produce an admissible formal blow-up $\mathcal{Y}\rightarrow \Spf R_0$ and a locally free $\mathscr{O}_{\mathcal{Y}}$-module $\mathscr{M}$ equipped with a continuous $\mathscr{O}_{\mathcal{Y}}$-linear action of $G$, such that $M=\Gamma(\mathcal{Y},\mathscr{M})\left[\frac 1 u\right]$.

It remains to see that $\Spa R\rightarrow\Spa R_0$ factors through a morphism $f:\Spa R\rightarrow\mathcal{Y}$.  If $\mathcal{Y}$ is the blow-up of $\Spf R_0$ along $I=(f_1,\ldots,f_r)\subset R_0$, then $\mathcal{Y}$ has a cover of the form $\left\{\Spf R_0\left\langle\frac{f_1,\ldots,f_r}{f_i}\right\rangle\right\}_i$.  The ring $R_0\left\langle\frac{f_1,\ldots,f_r}{f_i}\right\rangle$ is a noetherian ring of definition for the rational localization 
\[	U_i:=\Spa R\left\langle\frac{f_1,\ldots,f_r}{f_i}\right\rangle\subset \Spa R	\]
so the natural morphism $\Spa R|_{U_i}\rightarrow \Spa R_0$ factors through $\mathcal{Y}^{\mathrm{ad}}$ for each $i$.  Since these morphisms agree on overlaps by construction, we are done.
\end{proof}

Now we can construct $(\varphi,\Gamma)$-modules, exactly as in~\cite{berger-colmez}.
\begin{thm}
Let $R$ and $M$ be as above, and choose $b>0$ such that $\frac{b}{\lambda}<r_K$ and constants $c_1,c_2,c_3$ as in the statement of the Tate--Sen axioms.  Then there is some finite Galois extension $L/K$ and some integer $n\geq 0$ such that $\widetilde\Lambda_{R,[0,b]}\otimes_RM$ contains a unique projective sub-$\varphi^{-n}\left(\Lambda_{R,[0,p^{-n}b],L}\right)$-module $D_{b,L,n}(M)$ such that 
\begin{itemize}
\item	$D_{b,L,n}(M)$ is stable by $\Gal_L$ and fixed by $H_L$,
\item	The natural map $\widetilde\Lambda_{R,[0,b]}\otimes_{\varphi^{-n}\left(\Lambda_{R,[0,p^{-n}b],L}\right)}D_{b,L,n}(M)\rightarrow \widetilde\Lambda_{R,[0,b]}\otimes_RM$ is an isomorphism
\item	Locally on $\Spa R$, $D_{b,L,n}(M)$ admits a basis which is $c_3$-fixed by $\Gamma_K$, that is, if $\gamma\in\Gamma_K$, the matrix $G$ with respect to this basis satisfies $v_\Lambda(G-\mathrm{Id})>c_3$.
\end{itemize}
\end{thm}
\begin{proof}
After making an admissible formal blow-up on $\Spf R_0$ and localizing on $\Spa R$, we may assume that $M$ has a Galois-stable $R_0$-lattice $M_0\subset M$.  Let $k$ be an integer such that $v_{R,b}(u^k)>c_1+2c_2+2c_3$, and let $L/K$ be a finite Galois extension such that $\Gal_L$ acts trivially on $M_0/u^kM_0$.  Then by Corollary~\ref{cor: ts-1-r hk-invts}, $(\widetilde\Lambda_{R_0,[0,1]}\otimes_{R_0}M_0)^{H_L}$ is a free $\widetilde\Lambda_{R_0,[0,1]}^{H_L}$-module of rank $d$.  We choose a basis and let $\sigma\mapsto U_\sigma$ denote the corresponding cocycle.  The proof of ~\cite[Proposition 3.2.6]{berger-colmez} carries over nearly verbatim (working modulo powers of $u$ rather than $p$), and yields $B\in 1+u^k\Mat_d(\widetilde\Lambda_{R_0,[0,b]})$ such that $v_{R,b}(B-1)>c_2+c_3$ and $\sigma\mapsto B^{-1}U_\sigma\sigma(B)$ is trivial on $H_L$ and valued in $\Mat(\varphi^{-n(G)}(\Lambda_{R_0,[0,p^{-n(G)}b],L}))$.  This shows the existence of $D_{b,L,n}(M)$; it remains to check that it is unique. 

Suppose there are two such submodules. We may choose bases for each, and we obtain corresponding cocycles $\sigma\mapsto W_\sigma$ and $\sigma\mapsto W_\sigma'$ valued in $\Mat(\varphi^{-n(G)}(\Lambda_{R_0,[0,p^{-n(G)}b],L}))$.  Since these submodules generate the same $\widetilde\Lambda_{R_0,[0,b]}$-module, there is some matrix $C\in \Mat\widetilde\Lambda_{R_0,[0,b]}$ such that $W_\sigma'=C^{-1}W_\sigma(C)$.  But ~\cite[Proposition 3.2.5]{berger-colmez} also carries over nearly verbatim, and shows that $C$ actually has coefficients in $\varphi^{-n(G)}(\Lambda_{R_0,[0,p^{-n}b],L}))$.
\end{proof}

\begin{definition}
Let $M$ be a rank-$d$ representation of $\Gal_K$ with coefficients in $R$, and choose $b>0$ with $\frac{b}{\lambda}<r_K$, where $R$ is a $D_\lambda$-algebra and $\lambda=\frac 1 m$ with $p\nmid m$.  Then we define 
\begin{enumerate}
	\item	$D_{b,K}(M) := \left(\varphi^{n(L)}\left(D_{p^{n(L)}b,L,n}\right)\right)^{H_K}$
\item	If $I\subset [0,b]$ is an interval (which may have open endpoints), we define ${D}_{I,K}(M):={\Lambda}_{R,I,K}\otimes_{{\Lambda}_{R,[0,b],K}}{D}_{b,K}(M)$
\item	$D_{\rig,K}(M) := \varinjlim_{b\rightarrow 0} D_{(0,b],K}(M)$
\end{enumerate}
If $K$ is clear from context, we often drop it from the notation.
\end{definition}

After localizing on $\Spa R$, we may assume that $M$ has a Galois-stable $R_0$-lattice.  Then it follows from Corollary~\ref{cor: projectivity descends} that $D_{b,K}(M)$, and hence $D_{I,K}(M)$ and $D_{\rig,K}(M)$, is a projective $(\varphi,\Gamma_K)$-module.

\begin{remark}
The uniqueness of $D_{b,L,n}(M)$ ensures that the construction is functorial.
\end{remark}

\section{Galois cohomology}

We conclude by giving a definition of general $(\varphi,\Gamma)$-moduels over a pseudoaffinoid algebra $R$, and explaining how to compute the Galois cohomology of a Galois representation $M$ in terms of its $(\varphi,\Gamma)$-module $D_{\rig}(M)$.

\begin{definition}
        A $\varphi$-module over $\Lambda_{R,(0,b],K}$ is a coherent sheaf $D$ of  modules over the pseudorigid space $\bigcup_{a\rightarrow 0}\Spa(\Lambda_{R,[a,b],K})$ equipped with an isomorphism
        \[      \varphi_D:\varphi^\ast D\xrightarrow{\sim}\Lambda_{R,(0,b/p],K}\otimes_{\Lambda_{R,(0,b],K}}D   \]
        If $a\in (0,b/p]$, a $\varphi$-module over $\Lambda_{R,[a,b],K}$ is a finite $\Lambda_{R,[a,b],K}$-module $D$ equipped with an isomorphism
        \[      \varphi_{D,[a,b/p]}:\Lambda_{R,[a,b/p],K}\otimes_{\Lambda_{R,[a/p,b/p],K}}\varphi^\ast D\xrightarrow{\sim} \Lambda_{R,[a,b/p],K}\otimes_{\Lambda_{R,[a,b],K}}D  \]

        A $(\varphi,\Gamma_K)$-module over $\Lambda_{R,(0,b],K}$ (resp. $\Lambda_{R,[a,b],K}$) is a $\varphi$-module over $\Lambda_{R,(0,b],K}$ (resp. $\Lambda_{R,[a,b],K}$) equipped with a semi-linear action of $\Gamma_K$ which commutes with $\varphi_D$ (resp. $\varphi_{D,[a,b/p]}$).

        A $(\varphi,\Gamma_K)$-module over $R$ is a module $D$ over $\Lambda_{R,\rig,K}$ which arises via base change from a $(\varphi,\Gamma_K)$-module over $\Lambda_{R,(0,b],K}$ for some $b>0$.
\end{definition}

For any finite extension $K/\Q_p$, we may write $\Gamma_K\cong \Gamma_K^{p\mathrm{-tors}}\times \Gamma_K^{\cyc}$, where $\Gamma_K^{p\mathrm{-tors}}$ denotes the $p$-torsion subgroup of $\Gamma_K$ and $\Gamma_K^{\cyc}$ is its procyclic quotient; $\Gamma_K^{p\mathrm{-tors}}$ is trivial unless $p=2$, in which case it could be $(\Z/4)^\times$.  Then for any topological generator $\gamma$ of $\Gamma_K^{\cyc}$, we define the Fontaine--Herr--Liu complex 
\[      C_{\varphi,\Gamma}^\bullet: D\xrightarrow{\varphi_D-1,\gamma-1} D\oplus D\xrightarrow{(\gamma-1)\oplus(1-\varphi_D)} D      \]
(concentrated in degrees $0$, $1$, and $2$).  We let $H_{\varphi,\Gamma_K}^i(D)$ denote its cohomology in degree $i$.  We remark that, as in ~\cite[\textsection 2.3]{kpx}, the complex $C_{\varphi,\Gamma}^\bullet$ is independent of the choice of $\gamma$ up to canonical $R$-linear quasi-isomorphism.

If $M$ is a $\Q_p$-linear representation of $\Gal_K$ and $D_{\rig,K}(M)$ is the associated Galois representation over $\varinjlim_{b\rightarrow 0}\Lambda_{(0,b],K}$, then we have a canonical quasi-isomorphism $R\Gamma(\Gal_K,M)\xrightarrow\sim C_{\varphi,\Gamma}^\bullet$ between (continuous) Galois cohomology and Fontaine--Herr--Liu cohomology~\cite[Theorem 2.3]{liu2007}.  The same result holds for families of projective Galois representations with coefficients in classical $\Q_p$-affinoid algebras~\cite[Theorem 2.8]{pottharst2013}; we will prove the corresponding result in the pseudorigid setting.

\begin{thm}\label{thm: coh comp}
        Let $R$ be a pseudoaffinoid algebra, and let $M$ be a finite projective $R$-module equipped with a continuous $R$-linear action of $\Gal_K$.  Then the Galois cohomology $R\Gamma(\Gal_K,M)$ and the Fontaine--Herr--Liu cohomology $R\Gamma\left(\Gamma_K^{p\mathrm{-tors}},C_{\varphi,\Gamma_K^{\cyc}}^\bullet(D_{K,[0,b]}(M))\right)$ of the associated $(\varphi,\Gamma_K)$-module are canonically isomorphic.
\end{thm}

We also deduce the analogous result for $D_{\rig}(M)$:
\begin{corollary}
	Let $R$ be a pseudoaffinoid algebra, and let $M$ be a finite projective $R$-module equipped with a continuous $R$-linear action of $\Gal_K$.  Then the Galois cohomology $H^\bullet(\Gal_K,M)$ and the cohomology of the Fontaine--Herr--Liu complex $R\Gamma\left(\Gamma_K^{p-\mathrm{tors}}, C_{\varphi,\Gamma}^\bullet(D_{\rig,K}(M))\right)$ of the associated $(\varphi,\Gamma_K)$-module are canonically isomorphic.
\end{corollary}
\begin{proof}
	Replacing $K$ with the extension corresponding to the quotient $\Gamma_K\twoheadrightarrow\Gamma_K^{p-\mathrm{tors}}$, the Hochschild--Serre spectral sequence implies that we may assume that $\Gamma_K^{p-\mathrm{tors}}$ is trivial.  After making an admissible formal blowup on $\Spf R_0$ and localizing on $\Spa R$, we may again assume that $M$ is a free $R$-module containing a free $\Gal_K$-stable $R_0$-lattice, and there is some finite Galois extension $L/K$ and some $b>0$ such that $D_b':=\Lambda_{R_0,[0,b],L}\otimes_{\Lambda_{R_0,[0,b],K}}D_{b,K}(M)$ is free, by Lemma~\ref{lemma: galois descent phi gamma} and the proof of Corollary~\ref{cor: projectivity descends}.  Copying the proof of~\cite[Proposition 1.2.6]{kedlaya08} verbatim, we see that the natural morphism
	\[      \left[\varinjlim_{b\rightarrow0}D_b'\xrightarrow{\varphi-1}\varinjlim_{b\rightarrow0}D_{b/p}'\right]\rightarrow \left[\Lambda_{R,\rig,K}\otimes D_b\xrightarrow{\varphi-1}\Lambda_{R,\rig,K}D_{b/p}'\right]        \]
	is a quasi-isomorphism.  Since $D_{b,K}(M)$ is a direct summand of $D_b'$ as a $\varphi$-module, the same holds for the natural morphism
\[      \left[\varinjlim_{b\rightarrow0}D_{b,K}(M)\xrightarrow{\varphi-1}\varinjlim_{b\rightarrow0}D_{b/p,K}(M)\right]\rightarrow \left[D_{\rig,K}(M)\xrightarrow{\varphi-1}D_{\rig,K}(M)\right]        \]

The Fontaine--Herr--Liu complex is the total complex of the double complex
\[
	\begin{tikzcd}
		D_{\rig,L}(M) \arrow{r}{\varphi-1} & D_{\rig,L}(M)	\\
		D_{\rig,L}(M) \arrow{r}{\varphi-1}\arrow{u}{\gamma-1} & D_{\rig,L}(M) \arrow{u}{\gamma-1}
	\end{tikzcd}
\]
and colimits commute with taking cohomology; we therefore get a quasi-isomorphism
\[	\varinjlim_{b\rightarrow 0}C_{\varphi,\Gamma_L^{\cyc}}^\bullet(D_{b,L}(M))\xrightarrow\sim D_{\rig,L}(M)	\]
as desired.
\end{proof}

\begin{remark}
	If $p\neq 2$, then we see that Galois cohomology is computed by the Fontaine--Herr--Liu complex.
\end{remark}

The key is an Artin--Schreier calculation:
\begin{proposition}\label{prop: phi-invts}
Let $R$ be a pseudoaffinoid $D_\lambda$-algebra and let $R_0\subset R$ be a ring of definition strictly topologically of finite type over $D_\lambda^\circ$.  Then for any $b>0$, there is an exact sequence of $R_0$-modules
\[	0\rightarrow R_0\rightarrow\widetilde\Lambda_{R_0,[0,b],0}\xrightarrow{\varphi-1} \widetilde\Lambda_{R_0,[0,b/p],0}\rightarrow 0	\]
\end{proposition}

We first compute modulo $u$:
\begin{lemma}
If $R$ is a pseudoaffinoid $D_\lambda$-algebra and $R_0\subset R$ is a ring of definition strictly topologically of finite type over $D_\lambda^\circ$, then for any $b\in (0,\infty]$ we have an exact sequence
\[	0\rightarrow R_0/u\rightarrow \widetilde\Lambda_{R_0,[0,b],0}/u\xrightarrow{\varphi-1}\widetilde\Lambda_{R_0,[0,b/p],0}/u\rightarrow 0	\]
\end{lemma}
\begin{proof}
We may write 
\[	\widetilde\Lambda_{R_0,[0,b],0}/u\cong \Lambda_{R_0,[0,\infty]}[Y]/([\varpi]^{1/b}Y,u)	\]
and 
\[	\widetilde\Lambda_{R_0,[0,b/p],0}/u\cong \Lambda_{R_0,[0,\infty]}[Y']/([\varpi]^{p/b}Y',u)	\]
so that the map $\varphi:\widetilde\Lambda_{R_0,[0,b],0}/u\rightarrow\widetilde\Lambda_{R_0,[0,b/p],0}/u$ carries $Y$ to $Y'$ and the identity map carries $Y$ to $[\varpi]^{(p-1)/b}Y'$.  We may filter $R_0/u$ by powers of $p$; if we reduce modulo $p$, $R_0/(u,p)$ is an $\F_p$-vector space, and it suffices to prove that the sequence
\[	0\rightarrow \F_p\rightarrow \mathscr{O}_{\C_p}^\flat[Y]/(\varpi^{1/b}Y)\xrightarrow{\varphi-1}\mathscr{O}_{\C_p}^\flat[Y']/(\varpi^{p/b}Y')\rightarrow 0	\]
is exact.  

Given a polynomial $f(Y):=\sum_i a_iY^i\in \mathscr{O}_{\C_p}^\flat[Y]$,
\[	(\varphi-1)(f(Y)) = \sum_i (\varphi(a_i)-\varpi^{i(p-1)/b}a_i){Y'}^i \]
To compute the kernel of $\varphi-1$, we may assume that $v_{\C_p^\flat}(a_i)<\frac 1 b$ for all $i$ with $a_i\neq 0$, and that $v_{\C_p^\flat}(\varphi(a_i)-\varpi^{i(p-1)/b}a_i)\geq \frac p b$ for $i\geq 1$.  We have $v_{\C_p^\flat}(\varphi(a_i)-\varpi^{i(p-1)/b}a_i)\geq\min\{pv_{\C_p^\flat}(a_i),\frac{i(p-1)}{b}+v_{\C_p^\flat}(a_i)\}$, with equality unless $v_{\C_p^\flat}(a_i)=\frac{i}{b}$.  If $i\geq 1$, this contradicts the assumption that $v_{\C_p^\flat}(a_i)<\frac 1 b$, so in that case $v_{\C_p^\flat}(\varphi(a_i)-\varpi^{i(p-1)/b}a_i)\geq \frac p b$ implies $\min\{pv_{\C_p^\flat}(a_i),\frac{i(p-1)}{b}+v_{\C_p^\flat}(a_i)\}\geq\frac p b$.  But $pv_{\C_p^\flat}(a_i)<\frac p b$ by assumption, so this is impossible.  Thus, if $f(Y)$ represents an element of the kernel of $\varphi-1$, its coefficients in positive degree have valuation at least $\frac 1 b$, and therefore vanish in $\mathscr{O}_{\C_p}^\flat[Y]/(\varpi^{1/b}Y)$. Thus, we may assume that $f(Y)\in\mathscr{O}_{\C_p}^\flat$, and therefore that it is an element of $\F_p$.

To see that $\varphi-1:\mathscr{O}_{\C_p}^\flat[Y]/(\varpi^{1/b}Y)\rightarrow \mathscr{O}_{\C_p}^\flat[Y']/(\varpi^{p/b}Y')$ is surjective, we may lift an element of $\mathscr{O}_{\C_p}^\flat[Y']/(\varpi^{p/b}Y')$ to a polynomial $g(Y'):=\sum_i b_i{Y'}^i\in\mathscr{O}_{\C_p}^\flat[Y']$, and choose $a_i$ such that $a_i^p-\varpi^{i(p-1)/b}a_i=b_i$.  Then if $f(Y):=\sum_ia_iY^i$, we have $(\varphi-1)(f(Y))=g(Y)$, as desired.

Now suppose that we have an exact sequence 
\[	0\rightarrow R_0/(u,p^k)\rightarrow\widetilde\Lambda_{R_0,[0,b],0}/(u,p^k)\xrightarrow{\varphi-1}\widetilde\Lambda_{R_0,[0,b/p],0}/(u,p^k)\rightarrow 0	\]
and consider the diagram
\begin{equation*}
\resizebox{\displaywidth}{!}{
$
\begin{tikzcd}[ampersand replacement=\&]
	\& 0 \ar[d] \& 0 \ar[d] \& 0 \ar[d] \&	\\
0 \ar[r] \& p^kR_0/(u,p^{k+1}) \ar[r]\ar[d] \& p^k\widetilde\Lambda_{R_0,[0,b],0}/(u,p^{k+1}) \arrow{r}{\varphi-1}\ar[d] \& p^k\widetilde\Lambda_{R_0,[0,b/p],0}/(u,p^{k+1}) \ar[r]\ar[d] \& 0	\\
0\ar[r] \& R_0/(u,p^{k+1}) \ar[r]\ar[d] \& {\widetilde\Lambda}_{R_0,[0,b],0}/(u,p^{k+1}) \arrow{r}{\varphi-1}\ar[d] \& {\widetilde\Lambda}_{R_0,[0,b/p],0}/(u,p^{k+1}) \ar[r]\ar[d] \& 0	\\
0\ar[r] \& R_0/(u,p^{k}) \ar[r]\ar[d] \& {\widetilde\Lambda}_{R_0,[0,b],0}/(u,p^{k}) \arrow{r}{\varphi-1}\ar[d] \& {\widetilde\Lambda}_{R_0,[0,b/p],0}/(u,p^{k}) \ar[r]\ar[d] \& 0	\\
	\& 0 \& 0 \& 0 \&
\end{tikzcd}
$
}
\end{equation*}
The bottom row is exact by assumption, and the columns are exact by construction.  Moreover, 
\[	p^k\widetilde\Lambda_{R_0,[0,b],0}/(u,p^{k+1})\cong p^kR_0/(u,p^{k+1})\htimes\mathscr{O}_{\C_p}^\flat[Y]/(\varpi^{1/b}Y)	\]
and 
\[	p^k\widetilde\Lambda_{R_0,[0,b/p],0}/(u,p^{k+1})\cong p^kR_0/(u,p^{k+1})\htimes\mathscr{O}_{\C_p}^\flat[Y']/(\varpi^{p/b}Y')	\]
Since $p^kR_0/(u,p^{k+1})$ is an $\F_p$-vector space, the preceding calculation shows that the top row is exact, as well.  A diagram chase then shows that the middle row is exact, as desired.
\end{proof}

\begin{proof}[Proof of Proposition~\ref{prop: phi-invts}]
We work modulo successive powers of $u$; we claim that for any $k\geq 1$, the sequence
\[	0\rightarrow R_0/u^k\rightarrow \widetilde\Lambda_{R_0,[0,b],0}/u^k\xrightarrow{\varphi-1}\widetilde\Lambda_{R_0,[0,b/p],0}/u^k\rightarrow 0	\]
is exact.  We have proved the result for $k=1$, so we proceed by induction on $k$.  Assume the result for $k$ and consider the diagram
\begin{equation*}
\resizebox{\displaywidth}{!}{$
\begin{tikzcd}[ampersand replacement=\&]
	\& 0 \ar[d] \& 0 \ar[d] \& 0 \ar[d] \&	\\
0 \ar[r] \& u^kR_0/u^{k+1} \ar[r]\ar[d] \& u^k\widetilde\Lambda_{R_0,[0,b],0}/u^{k+1} \arrow{r}{\varphi-1}\ar[d] \& u^k\widetilde\Lambda_{R_0,[0,b/p],0}/u^{k+1} \ar[r]\ar[d] \& 0	\\
0\ar[r] \& R_0/u^{k+1} \ar[r]\ar[d] \& {\widetilde\Lambda}_{R_0,[0,b],0}/u^{k+1} \arrow{r}{\varphi-1}\ar[d] \& {\widetilde\Lambda}_{R_0,[0,b/p],0}/u^{k+1} \ar[r]\ar[d] \& 0	\\
0\ar[r] \& R_0/u^{k} \ar[r]\ar[d] \& {\widetilde\Lambda}_{R_0,[0,b],0}/u^{k} \arrow{r}{\varphi-1}\ar[d] \& {\widetilde\Lambda}_{R_0,[0,b/p],0}/u^{k} \ar[r]\ar[d] \& 0	\\
	\& 0 \& 0 \& 0 \&
\end{tikzcd}
$}
\end{equation*}
The bottom row is exact by assumption and the columns are exact by construction.  Since $R_0$ has no $u$-torsion, multiplication by $u^k$ defines an isomorphism $R_0/u\xrightarrow{\times u^k}u^kR_0/u^{k+1}$, and the top row is isomorphic to
\[	0\rightarrow R_0/u\rightarrow \widetilde\Lambda_{R_0,[0,b],0}/u\xrightarrow{\varphi-1}\widetilde\Lambda_{R_0,[0,b/p],0}/u\rightarrow 0	\]
which is exact.  Then a diagram chase shows that the middle row is exact, as well. 

Now we consider the inverse limit as $k\rightarrow \infty$; $\widetilde\Lambda_{R_0,[0,b],0}$ and $\widetilde\Lambda_{R_0,[0,b/p],0}$ are $u$-adically separated and complete (since $u\in \varpi^{1/b}$), so we have an exact sequence 
\[	0\rightarrow R_0\rightarrow \widetilde\Lambda_{R_0,[0,b],0}\xrightarrow{\varphi-1}\widetilde\Lambda_{R_0,[0,b/p],0}	\]
Moreover, the transition maps $R_0/u^{k+1}\rightarrow R_0/u^k$ are surjective, so the Mittag-Leffler condition ensures that $\varphi-1: \widetilde\Lambda_{R_0,[0,b],0}\rightarrow \widetilde\Lambda_{R_0,[0,b/p],0}$ is surjective, so we are done.
\end{proof}

\begin{lemma}\label{lemma: H_K coh triv}
For any finite extension $K/\Q_p$ and all sufficiently small $b>0$, there is a quasi-isomorphism
\[	[\widetilde\Lambda_{R_0,[0,b]}^{H_K}]\xrightarrow{\sim} C_{\mathrm{cont}}^\bullet\left(H_K,\widetilde\Lambda_{R_0,[0,b]}\right)	\]
where $C_{\mathrm{cont}}^\bullet(H_K,\widetilde\Lambda_{R_0,[0,b]})$ is the continuous Galois cohomology.
\end{lemma}
\begin{proof}
We need to prove that $H_{\mathrm{cont}}^i(H_K,\widetilde\Lambda_{R_0,[0,b]})=0$ for $i\geq 1$.  But this follows from the first Tate--Sen axiom, as in~\cite[Proposition 14.3.2]{brinon-conrad}.
\end{proof}

Now we can prove the main comparison.
\begin{proof}[Proof of Theorem~\ref{thm: coh comp}]
Let $K'/K$ denote the extension corresponding to the quotient $\Gamma_K\twoheadrightarrow \Gamma_K^{p\mathrm{-tors}}$.  The Hochschild--Serre theorem implies that the natural map $R\Gamma\left(\Gamma_K^{p\mathrm{-tors}}, R\Gamma(\Gal_{K'},M)\right)\rightarrow R\Gamma(\Gal_K,M)$ is a quasi-isomorphism, so we may replace $K$ with $K'$ and prove that $R\Gamma(\Gal_K,M)$ is computed by $C_{\varphi,\Gamma_K}(D_{b,K}(M))$ when $\Gamma_K^{p\mathrm{-tors}}$ is trivial.

After making an admissible formal blow-up on $\Spf R_0$ and localizing on $\Spa R$, we may assume that $M$ is a free $R$-module and contains a free $\Gal_K$-stable $R_0$-lattice $M_0$.  There is an associated finite projective $(\varphi,\Gamma_L)$-module $D_{b,K}(M_0)$ for some $b$, and the natural comparison map 
\[	\widetilde\Lambda_{R_0,[0,b]}\otimes_{\Lambda_{R_0,[0,b],K}} D_{b,K}(M_0)\rightarrow \widetilde\Lambda_{R_0,[0,b]}\otimes_{R_0}M_0	\]
is an isomorphism, equivariantly for the actions of $\Gal_K$, $\varphi$, and $\Gamma_K$.  We deduce that there is a natural exact sequence
\begin{equation*}
\resizebox{\displaywidth}{!}{$
0\rightarrow M_0\rightarrow \widetilde\Lambda_{R_0,[0,b]}\otimes_{\Lambda_{R_0,[0,b],K}} D_{b,K}(M_0)\xrightarrow{\varphi-1}\widetilde\Lambda_{R_0,[0,b/p]}\otimes_{\Lambda_{R_0,[0,b/p],K}} D_{b/p,K}(M_0)\rightarrow 0
$}
\end{equation*}
Applying continuous $H_K$-cohomology, Lemma~\ref{lemma: H_K coh triv} implies that we have a quasi-isomorphism
\begin{equation*}
\resizebox{\displaywidth}{!}{
$C_{\mathrm{cont}}^\bullet(H_K, M_0)\xrightarrow{\sim}\left[\widetilde\Lambda_{R_0,[0,b]}^{H_K}\otimes_{\Lambda_{R_0,[0,b],K}} D_{b,K}(M_0)\xrightarrow{\varphi-1}\widetilde\Lambda_{R_0,[0,b/p]}^{H_K}\otimes_{\Lambda_{R_0,[0,b/p],K}} D_{b/p,K}(M_0)\right]	$}
\end{equation*}
A Hochschild--Serre argument shows that we have a quasi-isomorphism
\[	C_{\mathrm{cont}}^\bullet(\Gal_K,M_0)\xrightarrow{\sim} C_{\varphi,\Gamma_K}^\bullet(\widetilde\Lambda_{R_0,[0,b]}^{H_K}\otimes_{\Lambda_{R_0,[0,b],K}} D_{b,K}(M_0))	\]
so we need to show that the natural map
\[	C_{\varphi,\Gamma_K}^\bullet(D_{b,K}(M))\rightarrow C_{\varphi,\Gamma_K}^\bullet(\widetilde\Lambda_{R_0,[0,b]}^{H_K}\otimes_{\Lambda_{R_0,[0,b],K}} D_{b,K}(M))   \]
is a quasi-isomorphism.  It suffices to show that the cohomology of $\Gamma_K$ acting on $X_{R_0,[0,b],K}^n\otimes_{\varphi^{-n}(\Lambda_{R_0,[0,p^{-n}b],K,0})}D_{b,K,n}(M_0)$ is trivial for sufficiently small $b>0$, or more concretely, if $\gamma$ is a topological generator of the procyclic part of $\Gamma_K$, that the action of $\gamma-1$ is continuously invertible.  But we have computed an explicit topological basis for $X_{R_0,[0,b],K}^n$, so an argument as in Proposition~\ref{prop: gamma-1 inv d} gives the desired result.
\end{proof}

\appendix
\section{Fiber products}\label{appendix}

In this appendix, we explain how to construct fiber products 
\[	\Spa(R,R^+)\times_{\Spa(S,S^+)}\Spa(R',{R'}^+)	\]
f Tate pre-adic spaces.  Huber~\cite[Proposition 1.2.2]{huber2013} explains how to do this for Huber rings under various sets of hypotheses: either $(R,R^+)$ is ``locally of weakly finite type'' over $(S,S^+)$ and $(S,S^+)\rightarrow (R',{R'}^+)$ is an adic morphism, or $(R,R^+)$ can be ``locally of finite type'' over $(S,S^+)$.  However, $\Z_p[\![u]\!]\left[\frac p u\right]^\wedge\left[\frac 1 u\right]$ is not locally of finite type over $\Z_p$, and neither are rings like $\widetilde{\A}^{(0,r]}$ and $\A_K^{(0,p^{-n}r]}$; moreover, none of these rings have rings of definition which are adic $\Z_p$-algebras.

Let $(R,R^+)$ and $(R',{R'}^+)$ be complete admissible Huber pairs over $(S,S^+)$.  That is, $R$, ${R'}$, and $S$ are finitely generated over rings of definition $R_0\subset R^+$, ${R'}_0\subset {R'}^+$, and $S_0\subset S^+$, respectively; let $I\subset R_0$ and $I'\subset {R_0'}$ be ideals of definition.  It is asserted in~\cite{scholze-weinstein-berkeley} that the fiber product $X:=\Spa(R,R^+)\times_{\Spa(S,S^+)}\Spa(R',{R'}^+)$ can be constructed in the category of pre-adic spaces; we provide a construction here for the convenience of the reader.

Recall the definition of a pre-adic space from~\cite{scholze-weinstein-berkeley}:
\begin{definition}\label{def: pre-adic}
Let $(V)^{\mathrm{ind}}$ be the category of triples $(X,\mathscr{O}_X,(\lvert\cdot(x)\rvert)_{x\in X})$, where $X$ is a topological space, $\mathscr{O}_X$ is a sheaf of ind-topological rings, and for each $x\in X$, $\lvert\cdot(x)\rvert$ is an equivalence class of continuous valuations on $\mathscr{O}_{X,x}$.  For a Huber pair $(A,A^+)$, we define $\Spa^{\mathrm{ind}}(A,A^+)\in (V)^{\mathrm{ind}}$ to be $(X,\mathscr{O}_X^{\mathrm{ind}},(\lvert\cdot(x)\rvert)_{x\in X})$, where $X=\Spa(A,A^+)$, $\mathscr{O}_X^{\mathrm{ind}}$ is the sheafification of the presheaf $\mathscr{O}_X$ in the category of ind-topological rings, and the valuations stay the same.

A \emph{pre-adic space} is an object of $(V)^{\mathrm{ind}}$ which is locally isomorphic to $\Spa^{\mathrm{ind}}(A,A^+)$ for some complete Huber pair $(A,A^+)$.
\end{definition}
By~\cite[Proposition 3.4.2]{scholze-weinstein-berkeley}, 
\[	\Hom_{(V)^{\mathrm{ind}}}(\Spa^{\mathrm{ind}}(A,A^+),\Spa^{\mathrm{ind}}(B,B^+))=\Hom_{\mathrm{CAff}}((B,B^+),(A,A^+))	\] for all complete Huber pairs $(A,A^+)$, $(B,B^+)$ (where $\mathrm{CAff}$ denotes the category of complete Huber pairs).  This permits us to study pre-adic spaces via a functor of points approach.

We wish to construct the fiber product $X:=\Spa^{\mathrm{ind}}(R,R^+)\times_{\Spa^{\mathrm{ind}}(S,S^+)}\Spa^{\mathrm{ind}}(R',{R'}^+)$.  This fiber product must be final among spaces $Y$ equipped with maps $Y\rightarrow\Spa(R,R^+)$ and $Y\rightarrow \Spa(R',{R'}^+)$ such that the compositions $Y\rightarrow\Spa(R,R^+)\rightarrow \Spa(S,S^+)$ and $Y\rightarrow\Spa(R',{R'}^+)\rightarrow \Spa(S,S^+)$ agree.

\begin{proposition}
Let notation be as above. The fiber product 
\[	\Spa^{\mathrm{ind}}(R,R^+)\times_{\Spa^{\mathrm{ind}}(S,S^+)}\Spa^{\mathrm{ind}}(R',{R'}^+)	\]
is representable in the category of pre-adic spaces.
\end{proposition}
\begin{proof}
If $R=R^+=R_0$, $R'={R'}^+=R_0'$, and $S=S^+=S_0$, then $R^+$, ${R'}^+$, and $S^+$ are rings of definitions, and we may simply take $R_0\otimes_{S_0}{R_0'}$ and complete $(I\otimes {R_0'}+R_0\otimes I')$-adically.

To handle the general case, we write $R=R_0[\underline X]/J$ and $R'=R'_0[\underline X']/J'$, where $\underline X$ and $\underline X'$ are finite collections of elements generating $R$ and ${R'}$, and $J\subset R_0[\underline X]$ and $J'\subset {R}_0'[\underline X']$ are ideals; let $\underline r, \underline r'$ be the images of $\underline X, \underline X'$ in $R, R'$, respectively.  Let $g:(S,S^+)\rightarrow (R,R^+)$ and $g':(S,S^+)\rightarrow (R',{R'}^+)$ be the structure maps.  If necessary, we replace $S_0$ with $S_0':=S_0\cap g^{-1}(R_0)\cap {g'}^{-1}(R_0')$; $S_0'$ is an open and bounded subring of $S$, hence a ring of definition, and it satisfies $g(S_0')\subset R_0$ and $g'(S_0')\subset R_0'$.  

Let $(T,T^+)$ be a complete affinoid Huber pair over $(S,S^+)$, and suppose that there are homomorphisms $f:(R,R^+)\rightarrow (T,T^+)$ and $f':(R',{R'}^+)\rightarrow (T,T^+)$ over $(S,S^+)$.  Therefore, there are homomorphisms $R^+\otimes_{S^+}{R'}^+\rightarrow T^+$ and $R\otimes_SR'\rightarrow T$.  For $r\in R$, consider $f(r)\in T$.  Since $T^+\subset T$ is an open subring, there is some open neighborhood $V\subset T^+$ of $0$ such that $f(r)\cdot V\subset T^+$.  Since $f'(I')$ must consist of topologically nilpotent elements, there is some $n\gg 0$ such that $f'({I'}^n)\subset V$, and therefore $f(r)\cdot f'({I'}^n)\subset T^+$ consists of integral elements.  Similarly, for each $r'\in R'$, there is some $n'\gg 0$ such that $f'(r')\cdot f(I^{n'})$ consists of integral elements.  Since $R$ and ${R'}$ are finitely generated over $R_0$ and $R'_0$, respectively, we may choose $n,n'\gg 0$ such that that $f(\underline r)\cdot f'({I'}^n),f'(\underline r')\cdot f(I^{n'})\subset T^+$.

We topologize $R\otimes_SR'$ so that 
\[	(R_0\otimes_{S_0}R_0')\left[(\underline r\otimes 1)(R_0\otimes I')^n, (1\otimes\underline r')(I\otimes R_0')^n\right]	\]
is a ring of definition:  We let $(R_0\otimes_{S_0}R_0')[\underline X,\underline X']_{(n)}$ be the polynomial ring $(R_0\otimes_{S_0}R_0')[\underline X,\underline X']$ and we equip it with the $R_0\otimes I'+I\otimes R'_0$-adic topology.  That is,
\begin{equation*}
	\resizebox{\displaywidth}{!}{$
U_m:=\left\{\sum_{\nu,\nu'}a_{\nu,\nu'}\underline X^\nu{\underline X'}^{\nu'}\mid a_{\nu,\nu'}\in (R_0\otimes I'+I\otimes R'_0)^{n(\nu+\nu')+m}\text{ for all }\nu,\nu'\right\}
$}
\end{equation*}
is a basis of neighborhoods of $0$.  We set 
\[	(R\otimes_SR')_{(n)}:=(R_0\otimes_{S_0}R_0')[\underline X,\underline X']_{(n)}/(J\otimes R_0', R_0\otimes J')	\] 
so that $(R\otimes_SR')_{(n),0}:=U_0/(J\otimes R_0', R_0\otimes J')$ is a ring of definition, and we let $(R\otimes_{S}R')_{(n)}^+$ be the integral closure of the image of 
\[	(R^+\otimes_{S^+}{R'}^+)\left[(\underline r\otimes 1)(R_0\otimes I')^n, (1\otimes\underline r')(I\otimes R_0')^n\right]	\]
in $(R\otimes_{S}R')_{(n)}$.

Now we define $((R\htimes_SR')_{(n)}, (R\htimes_{S}{R'})_{(n)}^+)$ to be the completion of the pair $(R\otimes_{S}R')_{(n)}, (R\otimes_{S}R')_{(n)}^+)$, and we set 
\[	X_{(n)}:=\Spa^{\mathrm{ind}}((R\htimes_{S}R')_{(n)},(R\htimes_{S}R')_{(n)}^+)	\]
The homomorphisms $g:(R,R^+)\rightarrow (T,T^+)$ and $g':(R',{R'}^+)\rightarrow (T,T^+)$ induce a unique morphism $\Spa^{\mathrm{ind}}(T,T^+)\rightarrow X_{(n)}$, by construction.

Furthermore, there are natural maps $(R\htimes_{S}R')_{(n+1)}\rightarrow (R\htimes_{S}R')_{(n)}$, and they are compatible with the natural maps from $R$, $R'$, and $S$.  The induced maps $X_{(n)}\rightarrow X_{(n+1)}$ make $X_{(n)}$ into a rational subset of $X_{(n+1)}$ for all $n$, so we may define a pre-adic space $X:=\cup_{n}X_{(n)}$.
Then $X$ is the pre-adic space representing $\Spa^{\mathrm{ind}}(R,R^+)\times_{\Spa^{\mathrm{ind}}(S,S^+)}\Spa^{\mathrm{ind}}(R',{R'}^+)$.
\end{proof}

\begin{example}
Let $R=R^+=\Z_p[\![u]\!]$, $R'=\Q_p$, ${R'}^+=\Z_p$, and $S=S^+=\Z_p$, so that we may take $R_0=\Z_p[\![u]\!]$, $R_0'={R'}^+=\Z_p$, and $S_0=S^+=\Z_p$.  Then we need $u^n\cdot\left(\frac 1 p\right)$ to be bounded for varying $n$, so we need to consider quotients of
\begin{multline*}
		R\langle X\rangle_{(p,u)^n}= \left\{\sum_{\nu\geq 0}a_{\nu}X^\nu\mid a_\nu \in \Z_p[\![u]\!] \text{ and for all } m, \right. 	\\
		a_{\nu}\in (p,u)^{n\nu+m}	\text{ for almost all }\nu	\left.\vphantom{\sum_{\nu\geq 0}}\right\}
\end{multline*}
and its subring
\[	\{\sum_{\nu\geq 0}a_{\nu}X^\nu\mid a_\nu\in\Z_p[\![u]\!], a_{\nu}\in (p,u)^{n\nu}\text{ for all }\nu	\}	\]
After we quotient by the ideal $(pX-1)$, the latter ring becomes $\Z_p[\![u]\!]\left[\frac{u^n}{p}\right]^\wedge$ and the former becomes $\Z_p[\![u]\!]\left[\frac{u^n}{p}\right]^\wedge\left[\frac 1 p\right]$.  Thus, we get the standard construction for the generic fiber of $\Spf\Z_p[\![u]\!]$ (cf. ~\cite[\textsection 7]{dejong}.
\end{example}

\begin{proposition}
Suppose $(R,R^+)$ and $(R',{R'}^+)$ are Tate, with pseudo-uniformizers $u\in R^+$ and $u'\in {R'}^+$, respectively, and suppose $S=S^+$.  Then $(R\widehat\otimes_{S} R')_{(n)}$ is also Tate, and the images of $u$ and $u'$ are pseudo-uniformizers.
\end{proposition}
\begin{proof}

There are natural continuous maps $R_0\htimes_{S_0} R'_0\rightarrow (R\htimes_{S}R')_{(n),0}$ and $R\otimes_SR'\rightarrow (R\htimes_{S}R')_{(n)}$ for all $n\geq 0$.  Since $u\otimes 1$ and $1\otimes u'$ are topologically nilpotent in $R_0\htimes {R'}_0$ and invertible in $R\otimes_SR'$, they are topologically nilpotent units in $(R\htimes_{S}R')_{(n)}$.
\end{proof}

\begin{remark}
We do not know whether this fiber product preserves properties such as being noetherian or being sheafy.  In particular, we do not know whether the fiber product of two adic spaces is again an adic space (or merely a pre-adic space).
\end{remark}

\bibliographystyle{plain}
\bibliography{bc-tate}

\end{document}